\documentclass[10pt]{article}

\usepackage{import}
\usepackage[utf8]{inputenc}
\usepackage[USenglish]{babel}
\usepackage{geometry}
\usepackage[titletoc,title]{appendix}
\usepackage{xpatch}
\xpretocmd{\appendixpagename}{\sffamily}{}{}
\usepackage{amssymb,amsmath,amsthm,amsfonts,bm}
\usepackage{subcaption}
\usepackage{xcolor}
\usepackage{enumitem}
\usepackage{thmtools}
\usepackage{array}
\usepackage{xstring}
\usepackage{xcolor}
\usepackage{xparse}
\usepackage{etoolbox}
\usepackage{calc}
\usepackage{dsfont}
\usepackage{mathtools}
\usepackage{pict2e}%
\usepackage{accents}%

\usepackage{tikz}
\usetikzlibrary{calc}
\usetikzlibrary{positioning}
\newcommand*{\connectorH}[4][]{
  \draw[#1] (#3) -| ($(#3) !#2! (#4)$) |- (#4);%
}

\usepackage[bibliography=common,appendix=append]{apxproof}

\usepackage{booktabs}
\usepackage{multirow}
\newcommand\myhrule{\addlinespace[-0.2\aboverulesep]\cmidrule[0.4pt](l{2.1pt}r{2.1pt}){1-6}\addlinespace[-0.8\belowrulesep]}

\usepackage{graphicx}

\usepackage{hyperref}%
\usepackage[capitalise]{cleveref}
\newcommand{\crefenv}[2]{%
  \namecref{#1}~\hyperref[#2]{\labelcref*{#1}.\ref*{#2}}%
}

\usepackage{pgfplots}
\pgfplotsset{compat=newest}

\usepackage{algpseudocode}
\usepackage{algorithm}
\makeatletter
\renewcommand{\ALG@beginalgorithmic}{\small}%
\makeatother
\algnewcommand{\Initialize}[1]{%
  \Statex \textbf{Initialize:}
  #1
}

\algblockdefx{Nest}{EndNest}{}{}%
\algtext*{Nest}{}%
\algtext*{EndNest}{}%

\newcounter{algsubstate}

\newcounter{algsubsubstate}
\renewcommand{\thealgsubsubstate}{\roman{algsubsubstate}}

\algnewcommand{\lFor}[1]{\State\algorithmicfor\ #1\ \algorithmicdo} %
\algnewcommand{\EndlFor}{\unskip\ \algorithmicend\ \algorithmicfor} %
\algnewcommand{\IIf}[1]{\State\algorithmicif\ #1\ \algorithmicthen}%
\algnewcommand{\EndIIf}{\State\algorithmicend\ \algorithmicif}%
\algnewcommand{\EndIIfi}{\algorithmicend\ \algorithmicif}%
\algnewcommand{\ElseIIf}[1]{\State\algorithmicelse\ \algorithmicif\ #1\ \algorithmicthen} %

\makeatletter
\renewenvironment{subequations}{%
  \refstepcounter{equation}%
  \protected@edef\theparentequation{\theequation}%
  \setcounter{parentequation}{\value{equation}}%
  \setcounter{equation}{0}%
  \def\theequation{\theparentequation.\arabic{equation}}%
  \ignorespaces
}{%
  \setcounter{equation}{\value{parentequation}}%
  \ignorespacesafterend
}
\makeatother

\let\texdisplaystyle\displaystyle
\def\displaytotextstyle{\textstyle\let\displaystyle\texdisplaystyle}

\newenvironment{talign*}
 {\let\displaystyle\displaytotextstyle\csname align*\endcsname}
 {\endalign}

\newcommand{\mathopfont}{\mathsf}

\DeclareMathOperator*{\argmin}{arg\,min}

\DeclareMathOperator*{\projection}{\mathopfont{proj}}

\DeclareRobustCommand{\lowerrighttriangle}{%
  \begingroup
  \setlength{\unitlength}{1ex}%
  \begin{picture}(1,1)
  \polyline(0,1)(1,0)(0,0)(0,1)
  \end{picture}%
  \endgroup
}

\newcommand{\R}{\mathbb{R}}
\newcommand{\ind}{\mathds{1}}
\DeclarePairedDelimiter\abs{\lvert}{\rvert}%

\DeclarePairedDelimiter\floor{\lfloor}{\rfloor}
\DeclarePairedDelimiterXPP\sign[1]{\operatorname{\mathopfont{sign}}}[]{}{#1}%

\DeclarePairedDelimiter\norm{\lVert}{\rVert}%
\DeclarePairedDelimiterX{\inner}[2]{\langle}{\rangle}{#1, #2}

\DeclarePairedDelimiterXPP\mydet[1]{\operatorname{\mathopfont{det}}}[]{}{#1}%

\let\oldvec\vec
\newcommand{\sorth}[1]{\reflectbox{\ensuremath{\oldvec{\reflectbox{\ensuremath{#1}}}}}}

\DeclarePairedDelimiterXPP\myvec[1]{\operatorname{\mathopfont{vec}}}[]{}{#1}%
\renewcommand{\vec}[1][]{\myvec{#1}}

\DeclarePairedDelimiterX{\mathargument}[1]{(}{)}{\mathargumentA{#1}}
\NewDocumentCommand{\mathargumentA}{>{\SplitArgument{1}{|}}m}{%
  \mathargumentB#1%
}
\NewDocumentCommand{\mathargumentB}{mm}{%
  \IfNoValueTF{#2}{%
    #1%
  }{%
    #1\;\delimsize|\;#2%
  }%
}

\NewDocumentCommand{\indicator}{e{_}}{%
  \ind\IfValueT{#1}{_{\{#1\}}}\mathargument
}

\DeclarePairedDelimiter\rdbracket{(}{)}
\NewDocumentCommand\maxsum{t{'}e{_^}}{%
\mathopfont{T}%
\IfBooleanT{#1}{'}%
\IfNoValueF{#2}{_{(#2)}}%
\IfNoValueF{#3}{^{#3}}%
\rdbracket}
\NewDocumentCommand\proj{t{'}e{_^}}{%
\projection%
\textstyle%
\IfBooleanT{#1}{'}%
\IfNoValueF{#2}{_{\!#2\!}}%
\IfNoValueF{#3}{^{\!#3\!}}%
\rdbracket}
\NewDocumentCommand\prox{t{'}e{_^}}{%
\mathopfont{prox}%
\textstyle%
\IfBooleanT{#1}{'}%
\IfNoValueF{#2}{_{#2}}%
\IfNoValueF{#3}{^{#3}}%
\rdbracket}

\DeclarePairedDelimiterX{\probargument}[1]{(}{)}{\probargumentA{#1}}
\NewDocumentCommand{\probargumentA}{>{\SplitArgument{1}{|}}m}{%
  \probargumentB#1%
}
\NewDocumentCommand{\probargumentB}{mm}{%
  \IfNoValueTF{#2}{%
    #1%
  }{%
    #1\;\delimsize|\;#2%
  }%
}
\NewDocumentCommand{\prob}{e{_}}{%
  \mathopfont{P}\IfValueT{#1}{_{#1}}\probargument
}
\NewDocumentCommand{\expect}{e{_}}{%
  \mathopfont{E}\IfValueT{#1}{_{#1}}\probargument
}
\NewDocumentCommand{\cvar}{e{_}}{%
  \mathopfont{CVaR}\IfValueT{#1}{_{#1}}\probargument
}
\NewDocumentCommand{\quantile}{e{_}}{%
  \mathopfont{Q}\IfValueT{#1}{_{#1}}\probargument
}
\NewDocumentCommand{\superquantile}{e{_}}{%
  \mathopfont{S}\IfValueT{#1}{_{#1}}\probargument
}

\newcommand{\maxsumball}[2]{\mathcal{B}^{#2}_{(#1)}}%
\renewcommand{\complement}{\mathsf{c}}
\newcommand{\kkt}{\mathrm{kkt}}
\newcommand{\ikkt}{\mathtt{kkt}}

\newtheorem{assumption}{Assumption}
\crefname{assumption}{Assumption}{Assumptions}

\newtheorem{proposition}{Proposition}

\theoremstyle{plain}

\newtheoremrep{lemma}[theorem]{Lemma}
\theoremstyle{plain}
\newtheorem{claim}{Claim}
\crefname{claim}{Claim}{Claims}

\newcommand{\ie}{\emph{i.e.,\space}}
\newcommand{\eg}{\emph{e.g.,\space}}
\newcommand{\cf}{\emph{cf.\space}}

\allowdisplaybreaks

\hypersetup{
	linktoc=all,
	colorlinks=true,
        linkcolor={black},%
	filecolor=black,
	urlcolor=black,
	citecolor={black},
}

\newcommand{\blue}[1]{{\textcolor{black}{#1}}}
\newcommand{\refone}[1]{{\textcolor{black}{#1}}}
\newcommand{\reftwo}[1]{{\textcolor{black}{#1}}}
\newcommand{\refthree}[1]{{\textcolor{black}{#1}}}

\title{\Large On $O(n)$ Algorithms for Projection onto the Top-$k$-sum \blue{Sublevel Set}}
\author{Jake Roth \thanks{Department of Industrial and Systems Engineering, 
University of Minnesota, Minneapolis, MN 55414 (roth0674@umn.edu).}
\and Ying Cui 
\thanks{Department of Industrial Engineering and Operations Research, University of California, Berkeley,  Berkeley, CA  94720 (yingcui@berkeley.edu). The author is partially supported by the National Science Foundation under grants CCF-2416172 and DMS-2416250, and the National Institutes of Health under grant R01CA287413-01.} 
}
\date{}

\begin{document}
\maketitle

\begin{abstract}
The \emph{top-$k$-sum} operator computes the sum of the largest $k$ components of a given vector.
The Euclidean projection onto the top-$k$-sum sublevel set serves as a crucial subroutine in iterative methods to solve composite superquantile optimization problems.
In this paper, we introduce a solver that implements two finite-termination algorithms to compute this projection. 
Both algorithms have \refthree{$O(n)$ complexity of floating point operations} when applied to a sorted $n$-dimensional input vector, where the absorbed constant is \emph{independent of $k$}.
This stands in contrast to \blue{an} existing grid-search-inspired method that has $O(k(n-k))$ complexity, \blue{a partition-based method with $O(\blue{n}+D\log D)$ complexity, where $D\leq n$ is the number of distinct elements in the input vector, and a semismooth Newon method with a finite termination property but unspecified floating point complexity}.
The improvement \blue{of our methods over the first method} is significant when $k$ is linearly dependent on $n$, which is frequently encountered in practical superquantile optimization applications.
In instances where the input vector is unsorted, an additional cost is incurred to (partially) sort the vector, \blue{whereas a full sort of the input vector seems unavoidable for the other two methods}.
To reduce this cost, we further derive a rigorous procedure that leverages approximate sorting to compute the projection, which is particularly useful when solving a sequence of similar projection problems.
Numerical results show that our methods solve problems of scale $n=10^7$ and $k=10^4$ within $0.05$ seconds, whereas the most competitive alternative, \reftwo{the semismooth Newton-based method, takes about $1$ second.}
The existing grid-search method and Gurobi's QP solver can take from minutes to hours.
\end{abstract}


\section{Introduction}
\nosectionappendix
We consider the Euclidean projection  onto the  top-$k$-sum (also referred to as max-$k$-sum in various works such as  \cite{todd2018max,eigen2021topkconv,cucala2023correspondence}) sublevel set.
Specifically, 
given a scalar budget $r\in\R$, an index $k\in\{1,2,\ldots,n\}$, and an input vector $x^0\in\R^n$, our aim is to develop a fast and finite-termination algorithm
to obtain the exact solution of the strongly convex problem
\begin{equation}\label{eq:maxksum_projection}
\begin{array}{rl}
\displaystyle\operatornamewithlimits{minimize}_{x\in \mathbb{R}^n} \quad & \displaystyle\frac{1}{2}\norm{x-x^0}^2_2 \\
\mbox{subject to}
& x\in\maxsumball{k}{r} \coloneqq \bigl\{x\in\R^n : \displaystyle\maxsum_k{x} \coloneqq\sum_{i=1}^k \sorth{x}_{i}\leq r \bigr\},
\end{array}
\end{equation}
where for any $x\in \mathbb{R}^n$, we write  $\sorth{x}\in \mathbb{R}^n$ as its sorted counterpart satisfying $\sorth{x}_{1} \geq \sorth{x}_{2} \geq \cdots \geq \sorth{x}_{n}$, and
$\maxsum_k{x}$ represents the sum of the largest $k$ elements of $x$.

The top-$k$-sum operator $\maxsum_k{\bullet}$ is closely related to  the \emph{superquantile} of a random variable, which is also known as the \emph{conditional value-at-risk} (CVaR) \cite{rockafellar1999cvar}, \emph{average-top-$k$} \cite{hu2023rank},  \emph{expected shortfall}, among other names.
Specifically, consider $X$ as a random variable.
Its superquantile at confidence level $\tau\in (0,1)$ is defined as
$\superquantile_\tau{X}\coloneqq \min\bigl\{t+(1-\tau)^{-1}\mathbb{E}[\max(X-t,0)]\bigr\}$.
When $X$ is  supported on $n$ atoms $x \coloneqq (x_1, \ldots, x_n)^\top$, each with equal probability, then $ \maxsum_k{x}/k=\superquantile_{\tau(k)}{X}$
averages the largest $k$ realizations of $X$, where $\tau(k)\coloneqq 1-k/n$.
In the context where $X$ follows a continuous distribution, one may select $n$ samples
$x= \{x_i\}_{i=1}^n$ and construct an empirical sample average approximation of its superquantile at confidence $\tau=1-k/n$ using $\maxsum_k{x}$.

Owing to the close relationship between the top-$k$-sum and the superquantile, the projection problem in  \eqref{eq:maxksum_projection} has
applicability as a subroutine in solving composite optimization problems 
of the form
\vspace{-1mm}
\begin{equation}\label{eq:source_problems}
\begin{array}{rl}
\displaystyle\operatornamewithlimits{minimize}_{z} \quad \displaystyle f(z) + \superquantile_{\tau_0}[\big]{G^0(z;\omega_0)} 
\quad \mbox{subject to} \quad 
 \superquantile_{\tau}[\big]{G(z;\omega)} \leq r
\end{array}
\vspace{-2mm}
\end{equation}
\refthree{where $f$ is a deterministic function, $G^{0}$ and $G$ are random mappings that depend on both decision $z$ and random vectors $\omega_0$ and $\omega$ with finite supports, respectively; $r$ is the sublevel-set parameter; and $\tau_0,\tau$ are superquantile confidence parameters.}
Problem \eqref{eq:source_problems} addresses the empirical or sample average approximation of risk-averse CVaR problems, commonly used in safety-critical applications to manage adverse outcomes, such as in the robust design of complex systems
\cite{dahlgren2003risk,dolatabadi2017stochastic,tavakoli2018resilience,rabih2020distributionally,chaudhuri2022certifiable}.
Additionally, such problems  arise from the convex approximation of chance constrained stochastic programs \cite{nemirovski2006convex,chen2010cvar}, and are relevant to matrix optimization problems involving a matrix's Ky-Fan norm \cite{overton1993optimality,wu2014moreau}, \ie the vector-$k$-norm of its (already sorted) singular values.
Recently, optimization problems involving superquantiles have attracted  significant attention in the machine learning community, proving instrumental in modeling problems which: (i) seek robustness against uncertainty, such as mitigating distributional shifts between training and test datasets \cite{laguel2021superquantiles}, or measuring robustness through probabilistic guarantees on solution quality \cite{robey2022probabilistic}; (ii) handle imbalanced data  \cite{yuan2020group,peng2022imbalanced}; or (iii) pursue notions of fairness \cite{williamson2019fairness,frohlich2022risk,liu2019human}. Interested readers are encouraged to consult a recent survey \cite{royset2022risk} for a comprehensive review of superquantiles.
A fast and reliable solver for computing the projection onto the top-$k$-sum \blue{sublevel set}, especially when dealing with a large number of samples,
is crucial for
first- or second-order methods to solve the large-scale, complex composite superquantile problem \eqref{eq:source_problems}.
\refthree{Interested readers are refereed to  \cite{roth2024fast} for a recent study   addressing this problem, which requires the projection oracle \eqref{eq:maxksum_projection} in each iteration of the augmented Lagrangian method.}

Given that problem \eqref{eq:maxksum_projection} is a strongly convex quadratic program, it accommodates the straightforward use of  off-the-shelf solvers, such as Gurobi, to compute its solution. However, numerical experiments indicate that Gurobi needs about $1$-$2$ minutes at best
to solve problems of size $n=10^7$, thus preventing its use as an effective subroutine  within an iterative approach to solve the composite problem \eqref{eq:source_problems};
see \cref{tab:time} in Section \ref{sec:experiments} for details of the numerical results. In addition,
generic quadratic programming solvers yield inexact solutions.
This can lead to
a challenge in precisely determining the (generalized) Jacobian associated with the projection operator \refone{(see \cite{wu2014moreau,roth2024fast})}, which may be needed in a second-order method to solve composite problems like \eqref{eq:source_problems}
To overcome these issues, 
a finite-termination, grid-search-inspired method is introduced in \cite{wu2014moreau}, which has a complexity of $O(k(n-k))$ for a sorted $n$-dimensional input vector.
In the context of composite problems such as
\eqref{eq:source_problems},
the value $k$ is usually set as a fixed proportion of $n$, \ie $k=\floor{(1-\tau) n}$ for an  exogenous risk-tolerance $\tau\in(0,1)$, resulting in $O(n^2)$ complexity in many practical instances. Consequently, adopting such a method to evaluate the projection repeatedly in an iterative algorithm to solve composite problem
\eqref{eq:source_problems}
is still prohibitively costly when $n$ is large (say in the millions), even if $\tau$ is close to $1$.

\blue{On the other hand, problem \eqref{eq:maxksum_projection} is a special case of the vector-$k$-norm problem studied in \cite{wu2014moreau}, which is a special case of the OWL-norm projection problem studied in \cite{zeng2014ordered,davis2015algorithm,li2021fast}.
The paper
\cite{zeng2014ordered} outlines a scalar rootfinding routine that was not shown to terminate finitely.
An $O(\blue{n}+D\log D)$ implementation of an algorithm is provided in \cite{davis2015algorithm}, where $D$ denotes the number of unique elements in the input vector $x^0$.
A minor extension of the analysis in \cite[Theorem 3.2]{davis2015algorithm} indicates that the procedure can be modified to yield an $O(n)$ method for solving the fully sorted vector-$k$-norm problem, but the implementation is nontrivial to modify.
Finally, \cite{li2021fast} introduced a finite termination semismooth Newton method that exhibited superior performance relative.
However, a potential limitation shared by each of these proposed methods is that a full sort of the input vector $x^0$ seems unavoidable.
}

Lastly, the top-$k$-sum projection problem \eqref{eq:maxksum_projection} is related to the isotonic projection problem
$\min_{x\in\R^n} \{\tfrac12\norm{x-x^0}_{\reftwo{2}} : x_i\geq x_{j},\;\forall (i,j)\in E\}$,
where $E\subseteq \{1,\ldots,n\}\times\{1,\ldots,n\}$ is comprised of the edges of a directed acyclic graph over $n$ nodes.
When $E=\{(i,i+1):1\leq i\leq n-1\}$ forms a chain \cite{bach2013learning}, then the constraint set can be represented by a polyhedral cone known as a (monotone) \emph{isotonic projection cone} \cite{isac1986isotoneprojectioncone,nemeth2012monotonecone} for fully sorted $x^0\in\R^n$. This problem can be solved in $O(n)$ complexity by the well-known \emph{pool adjacent violators algorithm} \cite{barlow1972pava} or its primal variant \cite{best1990isotonic}.
The constraint in our problem  \eqref{eq:maxksum_projection} can be viewed as the intersection of the isotonic constraints and a half space (see the formulation \eqref{eq:maxksum_projection_sort} in the next section). Unfortunately, the projection onto this intersection cannot be done sequentially onto the latter two sets, thereby necessitating  a specialized approach.

Given this context, the primary contribution of this paper is to provide an efficient oracle for obtaining an exact solution to the top-$k$-sum \blue{sublevel set} projection problem \eqref{eq:maxksum_projection}.
We propose two finite-termination algorithms, one is based on \blue{specializing} a pivoting method to solve the parametric linear complementarity problem for $Z$-matrices \cite{cottle2009lcp}
\blue{(see also \cite{pang1980parametric} for an application to a more general problem of the portfolio selection)}%
, and the other is  a variant of a grid-search based method introduced in \cite{wu2014moreau} that we call \emph{early-stopping grid-search}.
Both methods have complexities of $O(n)$ for a sorted input vector $x^0\in\mathbb{R}^n$, where the absorbed constant is independent of $k$. 
When the input vector is unsorted, an additional $O(\bar{k}_1\log n)$ floating point operations are needed, where $\bar{k}_1\leq n$ is not known \emph{a priori} but can be determined dynamically.
To reduce this additional cost  for practical applications in which the procedure to solve \eqref{eq:maxksum_projection} is called repeatedly in an iterative method, 
we further derive a property of the projection  that can make use of the approximate permutation of $x^0$ (\eg from the previous iteration). 
Extensive numerical results show that our solvers are often (multiple) orders magnitude faster than (and never slower than) the grid-search method introduced in \cite{wu2014moreau}, \reftwo{the semismooth Newton method in \cite{li2021fast}}, and Gurobi's inexact QP solver. The \verb|Julia| implementation of our methods is available at \url{https://github.com/jacob-roth/top-k-sum}.

The remainder of the paper is organized as follows.
In \cref{sec:background}, we summarize several equivalent formulations for solving \eqref{eq:maxksum_projection}.
In \cref{sec:proposed_algorithm}, we present a parametric-LCP (PLCP) algorithm 
and a new early-stopping grid-search (ESGS) algorithm. 
We also present modifications of these algorithms to handle the vector-$k$-norm projection problem.
Proofs in the preceding two sections are deferred to \cref{apx:backgroundformulations,,apx:proposedalgorithms}; additional detail on the proposed methods is collected in \cref{apx:detail}.
We compare the numerical performance of PLCP and ESGS with existing projection oracles on a range of problems in \cref{sec:experiments}. The paper ends with a concluding section. %

\subsection*{Notation and preliminaries}
For a matrix $A$, the submatrix formed by the rows in an index set $\mathcal{I}$ and the columns in an index set $\mathcal{J}$ is denoted $A_{\mathcal{I},\mathcal{J}}$, where ``$:$'' denotes MATLAB notation for index sets, \eg $\mathcal{I}=1:n$.
The vector of all ones in dimension $n$ is denoted by $\ind_n$; for an index set $\mathcal{I}\subseteq\{1,\ldots,n\}$, $\ind_{\mathcal{I}}$ denotes the vector with ones in the indices corresponding to $\mathcal{I}$ and zeros otherwise; \reftwo{when clear from context, for example $k\leq n$, we abuse notation and use $\R^n\ni\ind_k\coloneqq(\ind_{1:k},0_{k+1:n})$}; and $e^i$ denotes the $i^{\text{th}}$ standard basis vector.
For a vector $x\in\R^n$, $x_i$ denotes the $i^{\text{th}}$ element of $x$, and
for a vector $\nu$, $\prescript{}{v}{\nu}$ denotes the position of $v$ in $\nu$ (so that, \eg $\prescript{}{x_i}{x}=i$).
For a vector $x\in\R^n$, $\sorth{x}$ denotes a nonincreasing rearrangement of $x$ with the convention that $\sorth{x}_0\coloneqq +\infty$ and $\sorth{x}_{n+1}\coloneqq -\infty$.
For any \refthree{sorted} $x^0\in\R^n$ \refthree{and any positive integer $k$}, we may assume without loss of generality that there exist integers $k_0,k_1$ satisfying $0\leq k_0\leq k-1$ and $k\leq k_1\leq n$ such that
\begin{equation}
  \label{eq:order_structure}
  \sorth x^0_1 \geq\cdots\geq \sorth x^0_{k_0} > \sorth x^0_{k_0+1} =\cdots= \sorth x^0_k =\cdots= \sorth x^0_{k_1} > \sorth x^0_{k_1+1} \geq\cdots\geq \sorth x^0_{n},
\end{equation}
with the convention that $k_0=0$ if $\sorth x^0_1=\sorth x^0_k$ and $k_1=n$ if $\sorth x^0_n = \sorth x^0_k$.
\refthree{Note that the presence of strict inequalities is not limiting: for example, if $x=\ind$, then we may take $k_0=0$ and $k_1=n$.}
The indices $(k_0,k_1)$ denote the \emph{index-pair of $x^0$ associated with $k$} and define the related sets $\alpha\coloneqq\{1,\ldots,k_0\}$, $\beta\coloneqq\{k_0+1,\ldots,k_1\}$ and $\gamma\coloneqq\{k_1+1\ldots,n\}$.
For a vector $x^0\in\R^n$, the inequality $x^0\geq0$ is understood componentwise, and $x^0\odot y^0$ denotes the Hadamard product of $x^0$ and $y^0$.
The binary operators ``$\land$'' and ``$\lor$'' represent ``logical and'' and ``logical or,'' respectively.
\refthree{Algorithmic ``complexity'' refers to floating point operations.}

We also recall the following concepts from convex analysis.
The indicator function $\delta_{S}(x)$ of a set $S\subseteq\R^n$ takes the value 0 if $x\in S$ and $+\infty$ otherwise; the support function is denoted by $\sigma_S(x^0)\coloneqq\sup_{y}\{\inner{x^0}{y} : y\in S\}$; and $\proj_{S}{x^0} = \prox_{\delta_S}{x^0}$ denotes the metric projection and proximal operators, respectively.

Next we recall standard notation from the literature of the linear complementarity problem (LCP) \cite{cottle2009lcp}.
An LCP$(q,M)$, defined by  vector $q$ and matrix $M$, is the collection of all vectors  $z$ such that $0\leq Mz + q\perp z\geq 0$ where ``$\perp$'' denotes orthogonality.
Given a scalar parameter $\lambda$, the parametric LCP (PLCP) is a collection of LCPs represented by $\text{PLCP}(\lambda; q,d,M) = \{\text{LCP}(q+\lambda d,M) : \lambda\in\R\}$ with $d$ being a direction vector.
Finally, we say $M\in \mathbb{R}^{n\times n}$ is a  $Z$-matrix if all  off-diagonal elements are nonpositive \cite{cottle2009lcp}.

\section{Equivalent Formulations and Existing Techniques}
\label{sec:background}
\begin{toappendix}
  \label{apx:backgroundformulations}
\end{toappendix}
In this section, we review some equivalent formulations of the projection problem \eqref{eq:maxksum_projection} and motivate the form that our algorithms use.
In particular, the alternate formulations either (i) introduce additional variables that destroy desirable structure in the original problem; (ii) have structure that we do not presently know how to leverage in designing a finite termination algorithm with complexity independent of parameter $k$; or (iii) are no easier than the formulation we use.
Before studying the projection problem in greater detail, we note that there are at least two cases where  the solution of this projection problem is immediate (assuming that $x^0$ does not belong to the top-$k$-sum \blue{sublevel set}): (i) $k=1$: $\tilde{x}_i = \min\{r,x^0_i\}$
; and (ii) $k=n$: $\tilde{x}_i = x^0_i - (\ind_n^\top x^0-r)/n$.

\subsection{Unsorted formulations}

\subsubsection*{The Rockafellar-Uryasev formulation}
Using the Rockafellar-Uryasev \cite{rockafellar1999cvar} variational form of the superquantile, the projection problem \eqref{eq:maxksum_projection} for $x^0\in\R^n$ can be cast as a convex quadratic program (QP) subject to linear constraints with $n+1$ auxiliary variables:
\begin{align}
  \label{eq:maxksum_projection_RU}
  (\tilde{x},\tilde{t},\tilde{z}) = \argmin_{x,t,z} \Bigl\{
    \tfrac12\norm{x-x^0}_2^2 :
    t + \frac1k\sum_{i=1}^n z_i\leq r/k,\;
    z\geq x - t\ind,\;
    z\geq0
  \Bigr\}.
\end{align}
The introduction of $t$ and $z$ destroys strong convexity of the original problem.
Nonetheless, the interior-point method can be used to obtain a solution in polynomial time.
As an alternative, problem \eqref{eq:maxksum_projection_RU} can be solved via the solution method of the linear complementarity problem (LCP) \cite{guler1995generalized,cottle2009lcp}.

\subsubsection*{The unsorted top-$k$ formulation}
Given $x^0\in\R^n$, let $\kappa$ denote the (possibly unsorted) indices of the $k$ largest elements of $x^0$
and $[k]\coloneqq \prescript{}{(\sorth{x}^0_k)}{x^0}$ denote the position of the $k^{\text{th}}$ largest element of $x^0$.
Then $x^0_{[k]}$ and index $[k]$ can be identified in $O(n + k\log n)$ expected time using \texttt{quickselect} \reftwo{(see \cite{hoare1961find})} or in $O(k + (n-k)\log k)$ time using \reftwo{\texttt{heapsort} (see \cite{williams1964heapsort}) and the max-heap data structure, as well as Algorithm 2 in \cite{condat2016fast}}, and a second scan through $x^0$ can identify $\kappa$.
Given the indices $\kappa$ and $[k]$, consider the following problem
\begin{subequations}
  \label{eq:maxksum_projection_unsorted}
  \begin{align}
    \tilde{x} &= \argmin_x\bigl\{
      \tfrac12\norm{x-x^0}_2^2 : Bx\leq b
    \bigr\},\quad\mbox{where}\; b\coloneqq(r,0_{n-1}^\top)^\top\;\mbox{and}\\
    B&\coloneqq \begin{bmatrix} \ind_{\kappa}^\top \\ B'_{\mathcal{I}_1,:} \\ B'_{\mathcal{I}_2,:} \end{bmatrix},\;
    (\ind_{\kappa})_j = \indicator_{t\in\kappa}{j},\;
    B'_{ij}\coloneqq
    \begin{cases}
      +1,\quad&\text{if $(i\in\kappa) \land (j=[k])$ or if $(i\notin\kappa) \land (j=i)$}\\
      -1,\quad&\text{if $(i\in\kappa) \land (j=i)$ or if $(i\notin\kappa) \land (j=[k])$}
    \end{cases},
  \end{align}
\end{subequations}
with $B'\in\R^{(n-1)\times n}$, $\mathcal{I}_1 = \{i\in\{1,\ldots,n\}:i\in\kappa \setminus [k]\}$, and $\mathcal{I}_2 = \{i\in\{1,\ldots,n\}:i\notin\kappa\}$.
In general, the feasible region $\bigl\{x \in\R^n :
 \sum_{i\in\kappa} x_i\leq r,\;
 x_i\geq x_{[k]},\;\forall i \in \kappa,\;
 x_j\leq x_{[k]},\;\forall j \notin \kappa
\bigr\}$ is a strict subset of $\maxsumball{k}{r}$, but the optimal solutions must coincide, as summarized in the following simple result.
\begin{lemmarep}\label{lem:maxksum_projection_unsorted}
  The optimal solution of problem \eqref{eq:maxksum_projection} is the same as that of the unsorted top-$k$ problem \eqref{eq:maxksum_projection_unsorted}.
\end{lemmarep}
\begin{proof}
  Equivalence follows from a direct application of the observation that for any $y,z\in\R^n$,
  \begin{align*}
    \inner{y}{z}\leq \inner{\sorth{y}}{\sorth{z}}\tag{$*$},
  \end{align*}
  and equality holds if and only if there exists a permutation $\pi$ that simultaneously sorts $y$ and $z$, \ie $y_\pi = \sorth{y}$ and $z_\pi=\sorth{z}$.

  Because the objectives are strongly convex, both problems have unique solutions $\tilde{x}^{\eqref{eq:maxksum_projection}}$ and $\tilde{x}^{\eqref{eq:maxksum_projection_unsorted}}$.
  Let $\mathcal{F}_{\eqref{eq:maxksum_projection}}$ and $\mathcal{F}_{\eqref{eq:maxksum_projection_unsorted}}$ denote the feasible regions of the projection problems, and $\nu_{\eqref{eq:maxksum_projection}}$ and $\nu_{\eqref{eq:maxksum_projection_unsorted}}$ denote the optimal values.
  It is clear that $\mathcal{F}_{\eqref{eq:maxksum_projection_unsorted}}\subseteq \mathcal{F}_{\eqref{eq:maxksum_projection}}$ (in general, strict subset), so $\nu_{\eqref{eq:maxksum_projection}}\leq \nu_{\eqref{eq:maxksum_projection_unsorted}}$.
  On the other hand, by $(*)$ both $\tilde{x}^{\eqref{eq:maxksum_projection}}$ and $\tilde{x}^{\eqref{eq:maxksum_projection_unsorted}}$ must have the same ordering as $x^0$, and thus $\tilde{x}^{\eqref{eq:maxksum_projection}}\in\mathcal{F}_{\eqref{eq:maxksum_projection_unsorted}}$, so $\nu_{\eqref{eq:maxksum_projection_unsorted}}\leq \nu_{\eqref{eq:maxksum_projection}}$ by the optimality of $\tilde{x}^{\eqref{eq:maxksum_projection_unsorted}}$ for \eqref{eq:maxksum_projection_unsorted} and the fact that both problems share the same objective function.
  Therefore the problems are equivalent.
\end{proof}
The KKT conditions expressed in (monotone) LCP form are given by $0\leq  (b - Bx^0) + BB^\top z \perp z \geq0$,
where $z\geq0$ is the dual variable.
\reftwo{Aside from the positive definiteness of $BB^\top$, it is not clear if there is additional structure that can be used to design an efficient LCP solution approach.}

\subsubsection*{The unsorted top-$k$ formulation via the Moreau decomposition}
On the other hand, the Moreau decomposition $x^0 = \prox_{\delta_{\maxsumball{k}{r}}\!\!}{x^0} + \prox_{\delta^*_{\maxsumball{k}{r}}\!\!}{x^0}$ provides an \reftwo{alternative} formulation to compute a solution from $\tilde{x}=x^0-\prox_{\delta^*_{\maxsumball{k}{r}}\!\!}{x^0}$.
The conjugate function $\delta^*_{\maxsumball{k}{r}}$ can be computed easily by using properties of linear programs and is summarized in \cref{lem:conjugate_unsorted}.
Note that the following result is also useful for computing the dual objective value of a problem involving the top-$k$-sum \blue{sublevel set}.
\begin{lemmarep}
  \label{lem:conjugate_unsorted}
  Let $B$ be the unsorted-top-$k$ matrix defined in \eqref{eq:maxksum_projection_unsorted} and $c\in\R^n$ be arbitrary.
  Then
  \begin{align*}
    \delta^*_{\maxsumball{k}{r}}(c) = \begin{cases} \tfrac rk\ind^\top c,\quad &\text{if $B^{-\top}c\geq0$;}\\+\infty,\quad&\text{otherwise.}\end{cases}
  \end{align*}
  In addition, the condition $B^{-\top}c\geq0$ can be checked in $O(n + k + (n-k)\log k)$ time for the worst case and $O(n + k\log n)$ \reftwo{time} in expectation.
  \blue{Furthermore, a sufficient condition for $B^{-\top}c\geq0$ can be checked in $O(n)$ time.}
\end{lemmarep}
  \begin{proof}
    For $c\in\R^n$, we have for $B$ and $b$ defined in \eqref{eq:maxksum_projection_unsorted}
    \begin{align*}
      \delta^*_{\maxsumball{k}{r}}(c) &= \max_x\{c^\top x - \delta_{\maxsumball{k}{r}}(x)\} = \max_x\{c^\top x : Bx \leq b\} &(\text{definition})\\
      &= \max_y\{c^\top B^{-1}y : y\leq b\} = \max_y\{(B^{-\top} c)^\top y : y\leq b\}. &(\text{$B$ is invertible})
    \end{align*}
    Suppose that there is an index $i^*$ such that $(B^{-\top}c)_{i^*}<0$.
    Then by taking $y_{i^*}\downarrow-\infty$, the objective tends to $+\infty$.
    When $B^{-\top}c\geq0$, then the problem has an optimal (finite) solution that occurs at an extreme point, \ie $y^*=b$, with objective value $c^\top B^{-1}b = \tfrac rk\ind^\top c$ by direct verification.

    Next consider verifying the condition $B^{-\top}c\geq0$.
    First identify the index $[k]=\prescript{}{c_{[k]}}{c}$ (the index of the $k^{\text{th}}$ largest element of $c$) in $O(k+(n-k)\log k)$ by using max-heaps or $O(n)$ expected time using \texttt{quickselect}.
    Next, scan $c$ to identify the elements and $\kappa = \{i\in\{1,\ldots,n\} : c_i \geq c_{[k]}\}$, ensuring that $\abs{\kappa}=k$ (ties can be split arbitrarily) in $O(n)$ time.
    Then, the form of $B^{-1}$ can be verified to take the form
    \begin{align*}
      B^{-1} =
      \begin{array}{ccccc}
        i=[k] & i\in\kappa\setminus[k] & i\notin\kappa\\
        \bigl[\tfrac 1k \ind & V & W \bigr]
      \end{array}
    \end{align*}
    where $V$ is a matrix of columns of the form $\tfrac1k\ind - e^i$ for $i\in\kappa\setminus[k]$ for standard basis vector $e^i$ and $W$ is a matrix of columns of the form $e^i$ for $i\notin\kappa$.
    Thus $B^{-\top}c\geq 0$ can be checked in $O(n + k+(n-k)\log k)$ time or $O(n)$ expected time by checking (i) $c_i\geq0$ for $i\notin\kappa$; and (ii) $\tfrac1k \ind^\top c - c_i\geq0$ for $i\in\kappa\setminus[k]$.
    \blue{As a consequence of the form of $B^{-1}$, a sufficient condition for $B^{-\top}c\geq0$ is: $c\geq0$ and $c_i\leq \tfrac{1}{k}\ind^\top c$ for all $i\in\{1,\ldots,n\}$, which can be checked in $O(n)$ time by disregarding the identification of $\kappa$.}
  \end{proof}
By \cref{lem:conjugate_unsorted}, we can compute $\tilde{y} :=\prox_{\delta^*_{\maxsumball{k}{r}}\!\!}{x^0}$ using the following formulation
\begin{align}
  \tilde{y} = \argmin_{y\in\R^n} \bigl\{\tfrac12\norm{y-(x^0-\tfrac rk \ind)}_2^2 : B^{-\top}y \geq0 \bigr\},
\end{align}
with KKT conditions $0\leq B^{-\top}(x^0-\tfrac rk\ind) + B^{-\top}B^{-1}z \perp z\geq0$ for dual variable $z\geq0$.
The matrix $B^{-\top}B^{-1}$ shares properties similar to $BB^\top$, 
but the structure cannot be leveraged in an obvious manner.

\subsection{Sorted formulations}

\subsubsection*{The isotonic formulation}
An alternative approach adopted in \cite{wu2014moreau} is based on the  observation employed in \cref{lem:maxksum_projection_unsorted}:
if the initial input to the projection problem \eqref{eq:maxksum_projection} is sorted in a nonincreasing order, \ie $x^0 = \sorth{x}^0$, then the unique solution will also be sorted in a nonincreasing order, \ie $\tilde{x}_i\geq\tilde{x}_{i+1}$ for $i =1,\ldots,n-1$.
Thus \eqref{eq:maxksum_projection} is equivalent to
\begin{equation}\label{eq:maxksum_projection_sort}
\begin{array}{rl}
\displaystyle\operatornamewithlimits{minimize}_{x\in \mathbb{R}^n} \quad & \displaystyle\frac{1}{2}\norm{x-\sorth{x}^0}^2_2 \\
\mbox{subject to}
&\begin{rcases*}
    \displaystyle \sum_{i=1}^k x_i \leq r\\\vspace{2mm}
    x_i\geq x_{i+1},\;\forall i\in\{1,\ldots,n-1\}
\end{rcases*}
 \eqqcolon \{x\in\R^n : Cx\leq b\},
\end{array}
\end{equation}
where
  \begin{align}
    \label{eq:maxksum_projection_sort_polyhedron}
    b &\coloneqq (r,0_{n-1}^\top)^\top,\quad
    C \coloneqq \begin{bmatrix} (\ind_k^\top,\; 0_{n-k}^\top) \\ -D \end{bmatrix},\quad
    \begingroup
    \renewcommand*{\arraystretch}{-0.1}
    D\coloneqq\begin{bmatrix} +1 & -1\\ & {\ddots} & {\ddots} \end{bmatrix}\in\R^{(n-1)\times n}
    \endgroup,
  \end{align}
which consists of nonseparable isotonic constraints ($Dx\geq0$) and a single budget constraint ($\ind_k^\top x \leq r$).
Problems \eqref{eq:maxksum_projection} and \eqref{eq:maxksum_projection_sort} are equivalent in the sense that $\bar x$ is the solution to \eqref{eq:maxksum_projection_sort} if and only if there exists a permutation $\pi$ of $\{1,\ldots,n\}$ with inverse $\pi^{-1}$ such that $\bar x_{\pi^{-1}} = \tilde x$ is the solution to \eqref{eq:maxksum_projection}.
The solution to \eqref{eq:maxksum_projection_sort} is obtained by translating three contiguous subsets of $\sorth{x}^0$, as depicted in \cref{fig:mks_diagram}, and it obeys the same ordering as $\sorth{x}^0$.
The difficulty is in identifying the breakpoints $k_0$ and $k_1$ at the solution that satisfies \eqref{eq:order_structure}.

\begin{figure}
  \centering
  \resizebox*{\linewidth}{!}{%
  \begin{tikzpicture}[x=5cm,y=1cm]
  \def\xzerosort{{%
    2.0,%
    1.95,%
    1.8,%
    1.5,%
    1.3,%
    0.6,%
    0.4,%
    0.3,%
    0.05,%
    -0.1,%
    -0.2,%
    -0.4,%
    -0.5,%
  }}
  \def\lambda{0.5}
  \def\xbarsort{{%
    2.0-\lambda,%
    1.95-\lambda,%
    1.8-\lambda,%
    1.5-\lambda,%
    1.3-\lambda,%
    0.25,%
    0.25,%
    0.25,%
    0.05,%
    -0.1,%
    -0.2,%
    -0.4,%
    -0.5,%
  }}
  
  \foreach \i in {1,2,3,9,10,11}{
    \node[minimum width=0.6cm,draw,circle,fill=black,opacity=0.6] (a) at (\xzerosort[\i],+1) {};
    \node[minimum width=0.6cm,draw,circle,fill=blue,opacity=0.6] (b) at (\xbarsort[\i],-1) {};
    \draw[->, ultra thick, dotted] (a) -- (b);
  }
  \foreach \i in {0}{
    \node[minimum width=0.6cm,draw,circle,fill=black,opacity=0.6,label={[label distance=0.2cm]90:$\sorth{x}^0_{1}$}] (a) at (\xzerosort[\i],+1) {};
    \node[minimum width=0.6cm,draw,circle,fill=blue,opacity=0.6,label={[label distance=0.2cm]270:$\bar{x}_{1}$}] (b) at (\xbarsort[\i],-1) {};
    \draw[->, ultra thick, dotted] (a) -- (b);
  }
  \foreach \i in {4}{
    \node[minimum width=0.6cm,draw,circle,fill=black,opacity=0.6,label={[label distance=0.2cm]90:$\sorth{x}^0_{\bar{k}_0}$}] (a) at (\xzerosort[\i],+1) {};
    \node[minimum width=0.6cm,draw,circle,fill=blue,opacity=0.6,label={[label distance=0.2cm]270:$\bar{x}_{\bar{k}_0}$}] (b) at (\xbarsort[\i],-1) {};
    \draw[->, ultra thick, dotted] (a) -- (b);
  }
  \foreach \i in {8}{
    \node[minimum width=0.6cm,draw,circle,fill=black,opacity=0.6, label={[label distance=0.2cm]90:$\sorth{x}^0_{\bar{k}_1+1}$}] (a) at (\xzerosort[\i],+1) {};
    \node[minimum width=0.6cm,draw,circle,fill=blue,opacity=0.6, label={[label distance=0.2cm]270:$\bar{x}_{\bar{k}_1+1}$}] (b) at (\xbarsort[\i],-1) {};
    \draw[->, ultra thick, dotted] (a) -- (b);
  }
  \foreach \i in {12}{
    \node[minimum width=0.6cm,draw,circle,fill=black,opacity=0.6, label={[label distance=0.2cm]90:$\sorth{x}^0_{n}$}] (a) at (\xzerosort[\i],+1) {};
    \node[minimum width=0.6cm,draw,circle,fill=blue,opacity=0.6, label={[label distance=0.2cm]270:$\bar{x}_{n}$}] (b) at (\xbarsort[\i],-1) {};
    \draw[->, ultra thick, dotted] (a) -- (b);
  }
  
  \node[minimum width=0.6cm,draw,circle,fill=black,opacity=0.6, label={[label distance=0.2cm]90:$\sorth{x}^0_{\bar{k}_0+1}$}] (a1) at (\xzerosort[5],+1) {};
  \node[minimum width=0.6cm,draw,circle,fill=black,opacity=0.6, label={[label distance=0.2cm]90:$\sorth{x}^0_{k}$}] (a2) at (\xzerosort[6],+1) {};
  \node[minimum width=0.6cm,draw,circle,fill=black,opacity=0.6, label={[label distance=0.2cm]90:$\sorth{x}^0_{\bar{k}_1}\quad$}] (a3) at (\xzerosort[7],+1) {};
  \node[minimum width=0.6cm,draw,circle,fill=blue,opacity=0.6] (b) at (\xbarsort[5],-1) {};
  \draw[->, ultra thick, dotted] (a1) -- (b);
  \draw[->, ultra thick, dotted] (a2) -- (b);
  \draw[->, ultra thick, dotted] (a3) -- (b);
  
  \draw[-, ultra thick, red, opacity=0.5] (2.09, 1.5) rectangle (0.71,-1.5);
  \draw[-, ultra thick, orange, opacity=0.5] (0.69,1.5) rectangle (0.16,-1.5);
  \draw[-, ultra thick, green!50!black, opacity=0.5] (0.14, 1.5) rectangle (-0.59,-1.5);
  \matrix [below left] at (2.4,0.825) {
    \node [shape=rectangle, draw=red, line width=1, opacity=0.5, label=right:$\alpha$] {};\\%
    \node [shape=rectangle, draw=orange, line width=1, opacity=0.5, label=right:$\beta$] {};\\%
    \node [shape=rectangle, draw=green!50!black, line width=1, opacity=0.5, label=right:$\gamma$] {};\\%
  };
  \end{tikzpicture}  
  }
  \caption{\small Schematic of sorted input $\sorth x^0\in\R^n$ (grey, top) and sorted projection $\bar{x}\in\maxsumball{k}{r}$ (blue, bottom).}
  \label{fig:mks_diagram}
\end{figure}
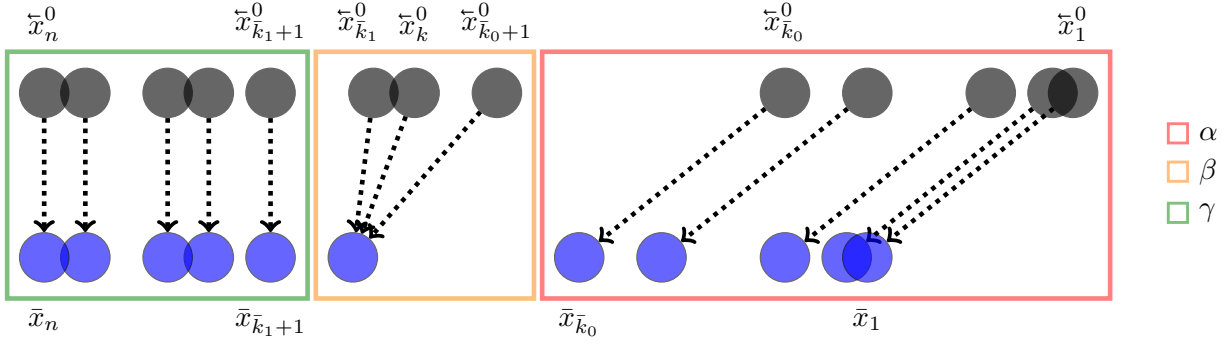

The constraint matrix $C$ associated with the sorted problem \eqref{eq:maxksum_projection_sort_polyhedron} is readily seen to be invertible,
and inspection of the KKT conditions yields the LCP($q,M$) with data
  \begin{align}
  \label{eq:mks_lcp_primal}
    q&\coloneqq (r,0_{n-1}^\top)^\top - C\, \sorth{x}^0\in\R^n, \quad M\coloneqq CC^\top\in\R^{n\times n}.
  \end{align}
By direct computation, $CC^\top$ is a symmetric positive definite $Z$-matrix (\ie a $K$-matrix \cf \cite[Definition 3.11.1]{cottle2009lcp}), so Chandrasekaran's complementary pivoting method \cite{chandrasekaran1970special} can process the LCP in at most $n$ steps (also see the \emph{$n$-step scheme} summarized  in \cite[Algorithm 4.8.2]{cottle2009lcp}).
The matrix $M$ is also seen to be tridiagonal except for possibly the first row and first column, due to contributions from the budget constraint.
As in the unsorted case, using the Moreau decomposition does not further simplify the problem.

\subsubsection*{The KKT grid-search}
\label{sec:kkt_grid}
An alternative method for solving the sorted problem \eqref{eq:maxksum_projection_sort} is based on a careful study of the sorted problem's KKT conditions  introduced in \cite{wu2014moreau}.
It is shown (\cf Step 2 in Algorithm 4 in \cite{wu2014moreau}) that each $(k_0,k_1)\in\{0,\ldots,k-1\}\times\{k,\ldots,n\}$ defines a candidate primal solution.
The true solution can be recovered by performing a grid-search over the sorting-indices ${k}_0$ and ${k}_1$ and terminating once the KKT conditions have been satisfied.
For each $(k_0,k_1)$, the KKT conditions can be checked in constant time that is independent of $n$ and $k$, so the overall complexity is $O(k(n-k))$.

\section{Proposed Algorithms}
\label{sec:proposed_algorithm}
\begin{toappendix}
  \label{apx:proposedalgorithms}
\end{toappendix}
In this section, we describe two efficient procedures for solving the projection problem \eqref{eq:maxksum_projection} \blue{by viewing the top-$k$-sum sublevel set as the intersection of a summation constraint and an ordering constraint}.
The first method is a (dual) parametric pivoting procedure based on the $Z$-matrix structure of the sorted problem's KKT conditions, \blue{which ensures that iterates are sorted and terminates once the summation constraint is satisfied.}
The second method \blue{uses a detailed analysis of the KKT conditions to} refine the (primal) grid-search procedure introduced in \cite{wu2014moreau}, \blue{which ensures that the summation constraint is satisfied (with equality) and terminates once the ordering constraint is satisfied}.

\subsection{A parametric-LCP approach}
\label{sec:plcp}
\reftwo{Penalizing the summation constraint in \eqref{eq:maxksum_projection_sort} with $\lambda\geq 0$ in the objective yields a sequence of subproblems parametrized by $\lambda$ that can be handled by he parametric-LCP (PLCP) approach described in \cite[Section 4.5]{cottle2009lcp}.
Given a candidate solution that fails to satisfy both constraints simultaneously, the PLCP prescribes two forms of corrective actions: increasing the penalty parameter $\lambda\geq0$ and/or expanding the set of active isotonic constraints.
The former action is moderated by the need to seek a sorted ordering, which limits how much $\lambda$ can be increased and results in a piecewise affine solution mapping with respect to $\lambda$.
Due to the structure of the isotonic constraints, both actions can be implemented in constant time, and at most $n$ adjustments to $\lambda$ will be needed since the basis cannot contain more than $n$ variables.~\footnote{
After posting the first version of the manuscript online in October 2023, the authors became aware, through private communication, of independent work by Eric Sager Luxenberg on a similar approach to solve problem \eqref{eq:maxksum_projection}.}
}

For fixed $\lambda$, the subproblem is given by
\begin{equation}\label{eq:plcp_primal}
\begin{array}{rl}
\displaystyle\operatornamewithlimits{minimize}_{x\in \mathbb{R}^n} \quad & \displaystyle\frac{1}{2}\norm{x-\sorth{x}^0}^2_2 + \lambda\cdot\bigl(\ind_k^\top x - r\bigr)\\[0.1in]
\mbox{subject to}
& Dx\geq0,
\end{array}
\end{equation}
where $D$ is defined in \eqref{eq:maxksum_projection_sort_polyhedron} as the isotonic operator associated with the ordered, consecutive  differences in $x$.
Collecting these problems for $\lambda\geq0$ yields the PLCP($\lambda;q,d,M$) where $M\coloneqq DD^\top$ is a tridiagonal positive definite $Z$-matrix.
To solve \eqref{eq:maxksum_projection_sort}, it is sufficient to identify a $\lambda \geq 0$  so that (i) the budget constraint is satisfied; and (ii) the LCP optimality conditions associated with \eqref{eq:plcp_primal} hold.
Let $z\geq0_{n-1}$ be the dual variable associated with $Dx\geq0$.
The KKT conditions of \eqref{eq:plcp_primal} take the form of
\begin{align}
  \label{eq:plcp_dual_kkt}
  0\leq w:=D(\sorth{x}^0 - \lambda\ind_k + D^\top z) \perp z \geq0, 
\end{align}
and yield the following PLCP($\lambda;q,d,M$) formulation of the full projection problem \eqref{eq:maxksum_projection_sort}, with PLCP data:
\begin{align}
  \label{eq:plcp}
  M &\coloneqq DD^\top,\quad q\coloneqq D\,\sorth{x}^0\geq0, \quad
  d \coloneqq -D\ind_k = -e^k.
\end{align}
One can compute that
\begin{align}
  M = \begin{bmatrix} 2 & -1\\ -1 & 2 & -1\\ & \ddots & \ddots & \ddots\\ & & -1 & 2 & -1\\ & & & -1 & 2\end{bmatrix}\in\R^{(n-1)\times (n-1)},
\end{align}
which is a  positive definite $Z$-matrix. By increasing the parameter $\lambda$ from $0$ to $+\infty$, one can solve the problem by identifying the optimal basis $\xi\subseteq\{1,\ldots,n-1\}$ such that the budget constraint is satisfied, $z_{\xi^\complement}=0$ (where $\xi^\complement = \{1,\ldots,n-1\}\setminus\xi$), and $w_{\xi}=0$.
Due to a special property of the $Z$-matrix \refone{that ensures monotonicity of the solution $z$ as a function of $\lambda$ \cite[Discussion of Proposition 4.7.2]{cottle2009lcp}}, the unique solution $z(\bar\lambda)$ can be obtained by solving at most $n$ subproblems.
Finally, since $M$ is tridiagonal, each subproblem for a fixed $\lambda$ can be solved in $O(1)$, and a primal solution can be recovered from the optimal dual vector by $\bar{x} = \sorth{x}^0 - \lambda\ind_k + D^\top z(\bar\lambda)$, leading to an $O(n)$ method where the absorbed constant is independent of $k$.

\begin{algorithm}[ht!]
	\caption{\small PLCP projection onto $\maxsumball{k}{r}$}\label{alg:projection_maxksum_lcp}
	\begin{algorithmic}[1]
            \footnotesize
		\Initialize{
      \noindent
      \begin{enumerate}[label=(\roman*),topsep=0pt,itemsep=-1ex,partopsep=1ex,parsep=1ex]
        \item If $\maxsum_k{ \sorth{x}^0}\leq r$, return $\bar{x}=\sorth{x}^0$.
        \item Otherwise, if $k=1$, set $\bar{x} = \min\{x^0,r\}$.
        \item Otherwise, if $k=n$, set $\bar{x} = \sorth{x}^0 - \ind_n\cdot(\ind_n^\top x^0-r)/n$.
        \item Otherwise, handle iteration $t=0$.
        \begin{enumerate}[topsep=0pt,itemsep=-1ex,partopsep=1ex,parsep=1ex]
          \item Set $s^0=\sum_{i=1}^k x_i^0$, $t=0$, $\xi=\emptyset$, $q_k=x^0_k-x^0_{k+1}$, $\lambda=q_k$, and $m=s^0 - k \lambda$.
          \item If $m \leq r$, set $\bar\lambda = (s^0 - r)/k$ and $\bar x = x^0 - \bar\lambda\ind_k$.
          \item Otherwise, set $t=1$, $\xi=\{k\}$, $a=b=s=k$, $\prescript{}{a}{\xi}=\prescript{}{k}{\xi}=\prescript{}{b}{\xi}=1$, $\lambda^a=\lambda^b=\lambda$, $z^0_{a}=z^0_k=z^0_{b}=-q_k\cdot m_{ij}^{-1}(\prescript{}{k}{\xi},\prescript{}{k}{\xi})=-q_k/2$, $\sigma=2-m^{-1}_{ij}(\prescript{}{k}{\xi},\prescript{}{k}{\xi})=3/2$, define the function $m_{ij}^{-1}(\prescript{}{c}{\xi},\prescript{}{d}{\xi}) \coloneqq (M_{\xi\xi}^{-1})_{\prescript{}{c}{\xi},\prescript{}{d}{\xi}} = \bigl(\abs{\xi}+1 - \max\{\prescript{}{c}{\xi},\prescript{}{d}{\xi}\} \bigr) \cdot \min\{\prescript{}{c}{\xi},\prescript{}{d}{\xi}\} \cdot \bigl(\abs{\xi}+1\bigr)^{-1}$,
          and proceed.
        \end{enumerate}
      \end{enumerate}
    }
		\While{true}
      \State{Compute the $(t+1)^{\text{th}}$ breakpoint:}
      \Nest
        \State{$\hphantom{\lambda^a}\mathllap{\lambda^a} = (q_{a-1} - z^0_a) / m_{ij}^{-1}(\xi_{k},\xi_{a})$}
        \State{$\hphantom{\lambda^a}\mathllap{\lambda^b} = (q_{b+1} - z^0_b) / m_{ij}^{-1}(\xi_{k},\xi_{b})$}
        \State{$\hphantom{\lambda^a}\mathllap{\lambda} = \min\bigl\{\lambda^a,\lambda^b\bigr\}$}
      \EndNest
      \State{Compute $\maxsum_k{x(\lambda)}$ to determine if optimal solution lies within $(0,\lambda]$:}
        \Nest
          \State{$T = s^0 - k\cdot\lambda + z^0_{k} + m_{ij}^{-1}(\prescript{}{k}{\xi},\prescript{}{k}{\xi})\cdot\lambda$}\Comment{compute $\maxsum_k{x(\lambda)}$.}
          \If{$T\leq r$}
            \State{$\bar\lambda = (s^0 - r + z_{k})/\bigl(k - m_{ij}^{-1}(\prescript{}{k}{\xi},\prescript{}{k}{\xi})\bigr)$}\Comment{solve $\maxsum_k{x(\lambda)}=r$ for $\lambda$.}
            \State{\Return{$\bar x = \sorth{x}^0 - \bar\lambda\ind_k + D^\top z(\bar\lambda)$} by calling \cref{alg:Btz}}
          \EndIf
        \EndNest
      \State{Update $z(0)$ via Schur complement, and update $\xi$ by inspecting $\{\lambda^a,\lambda^b\}$:}
        \Nest
          \If{$\lambda=\lambda^a$}
            \State{$z_a^0 = (z_a^0 - q_{a-1})/\sigma$}
            \State{$\hphantom{z_a^0}\mathllap{z_k^0} = z_k^0 + z_a^0\cdot m_{ij}^{-1}(\prescript{}{k}{\xi},\prescript{}{a}{\xi})$}
            \State{$\hphantom{z_a^0}\mathllap{z_b^0} = z_b^0 + z_a^0\cdot m_{ij}^{-1}(\prescript{}{b}{\xi},\prescript{}{a}{\xi})$}
            \State{$\hphantom{z_a^0}\mathllap{a} = a-1, \quad \prescript{}{k}{\xi} = \prescript{}{k}{\xi}+1, \quad \prescript{}{b}{\xi} = \prescript{}{b}{\xi}+1$}\Comment{$\xi=(a-1,\xi)$.}
          \Else
            \State{$\hphantom{z_a^0}\mathllap{z_b^0} = (z_b^0 - q_{b+1})/\sigma$}
            \State{$z_a^0 = z_a^0 + z_b^0\cdot m_{ij}^{-1}(\prescript{}{a}{\xi},\prescript{}{b}{\xi})$}
            \State{$\hphantom{z_a^0}\mathllap{z_k^0} = z_k^0 + z_b^0\cdot m_{ij}^{-1}(\prescript{}{k}{\xi},\prescript{}{b}{\xi})$}
            \State{$\hphantom{z_a^0}\mathllap{b}=b+1, \quad \prescript{}{b}{\xi} = \prescript{}{b}{\xi}+1$}\Comment{$\xi=(\xi,b+1)$.}
          \EndIf
        \EndNest
      \State{Increment iteration:}
      \Nest
        \State{$t = t+1$ and $\sigma = (t+2)/(t+1)$}\Comment{Schur complement: $\sigma=2 - m^{-1}_{ij}(\prescript{}{a}{\xi},\prescript{}{a}{\xi})$.}
      \EndNest
    \EndWhile
\end{algorithmic}
\end{algorithm}

\begin{algorithm}[ht!]
	\caption{\small $O(\abs{\xi})$ update of $y\gets y+D_{:,\xi}^\top z_\xi$ from optimal basis elements $a,b,\prescript{}{k}{\xi}$, dual $\lambda$, and $\sorth{x}^0\in\R^n$.}\label{alg:Btz}
	\begin{algorithmic}[1]
    \footnotesize
    \Initialize{Set $m=b-a+1$, $c=0$, $a_i=0$.}
		\State{Compute $\texttt{cumsum}((1,2,\ldots,m) \odot \texttt{reverse}(-q^\lambda_\xi))$:\Comment{\texttt{cumsum} is cumulative sum and \texttt{reverse} reverses the order of a vector}}
    \Nest\vspace{-2mm}
      \For{$i \in \{1,\ldots,m\}$}
        \State{$j = b-i+1$}
        \State{$c = c + i \cdot \bigl( -(\sorth{x}^0_j - \sorth{x}^0_{j+1}) \bigr)$}
        \IIf{$m-i+1 = \prescript{}{k}{\xi}$}; $c = c +\lambda$; \EndIIfi
      \EndFor
      \State{$c = c / (m+1)$ and $y_a = y_a + c$}
    \EndNest
    \State{Compute $\texttt{cumsum}((1,2,\ldots,m) \odot -q^\lambda_\xi)$:}
    \Nest\vspace{-1mm}
      \For{$i \in \{2,\ldots,m+1\}$}
        \State{$j=a+i-2$}
        \State{$c = c - \bigl( -(\sorth{x}^0_j - \sorth{x}^0_{j+1}) \bigr)$}
        \IIf{$i-1 = \prescript{}{k}{\xi}$}; $c = c + \lambda$; \EndIIfi
        \State{$y_{j+1} = y_{j+1} + c$}
      \EndFor
    \EndNest
	\end{algorithmic}
\end{algorithm}

The parametric-LCP method specialized to the present problem is summarized in \cref{alg:projection_maxksum_lcp}.
Additional detail on the mechanics of the pivots is provided in \cref{apx:detail:plcp}.
In addition, it is worthwhile to point out that to avoid additional memory allocation, the full dual solution $z$ is not maintained explicitly throughout the algorithm (only three components, $z_a$, $z_b$, and $z_k$ are maintained), so the step in line 10 of \cref{alg:projection_maxksum_lcp} requires a simple but specialized procedure summarized in \cref{alg:Btz} to implicitly reconstruct $z$ and compute $D^\top z$.
This can be done in $O(\abs{\xi})$ cost given an optimal basis $\xi$ as outlined in \cref{lem:Btz}.
On the other hand, if the user is willing to store the full dual vector $z$, then by storing and filling in the appropriate entries of $z$ after each pivot (line 23 in \cref{alg:projection_maxksum_lcp}\footnote{
  The algorithm computes $z^0 \coloneqq z(0)$ rather than $z(\lambda)$; the desired vector can be obtained from $z_\xi(\lambda) =z_\xi^0 + \lambda M_{\xi\xi}^{-1}e^k_\xi$ given basis $\xi$.
}), it is clear that $D^\top z$ can be obtained in $O(\abs{\xi})$ cost due to the simple pairwise-difference structure of $D$ and the complementarity structure which gives $z_{\xi^\complement}=0$.
Finally, we formally state the $O(n)$ complexity result, preceded by a lemma whose proof is deferred to \cref{apx:proposedalgorithms:plcp}.

\begin{toappendix}
  \subsection{PLCP}
  \label{apx:proposedalgorithms:plcp}
\end{toappendix}
\begin{lemmarep}
  \label{lem:Btz}
  \cref{alg:Btz} computes $D^\top z$ in $O(\abs{\xi})$ from initial data $\sorth{x}^0$ and optimal output of \cref{alg:projection_maxksum_lcp}: $a$, $b$, $\prescript{}{k}{\xi}$, and $\lambda$.
\end{lemmarep}
\reftwo{Using the above lemma, we arrive at the folowing conclusion.}
\begin{proof}
  Let $\xi$ \blue{be a given contiguous subset} with $\abs{\xi}=m$ and suppress the dependence on $\xi$.
  The goal is to compute $D_{\xi,:}^\top z_\xi$ where $z_\xi = -M_{\xi\xi}^{-1}(q_\xi + \lambda d_\xi)$.
  For convenience, we drop the dependence on $\xi$.
  The matrix $M^{-1}=(DD^\top)^{-1}$ is completely dense and symmetric with (lower triangular) entries given by
  \begin{align*}
    (DD^\top)^{-1}_{\lowerrighttriangle} = \frac{1}{m+1}
    \begin{bmatrix}
      1\cdot m\\
      1\cdot(m-1) & 2\cdot(m-1)\\
      1\cdot(m-2) & 2\cdot(m-2) & \ddots\\
      \vdots & \vdots & & \ddots\\
      1\cdot2 & 2\cdot2\\
      1\cdot1 & 2\cdot1 & & \cdots & m\cdot1
    \end{bmatrix}\in\R^{m\times m}.
  \end{align*}
  By direct computation,
  \begin{align*}
    D^\top(DD^\top)^{-1} &=
    \frac{1}{m+1}
    \begin{bmatrix}
      m & m-1 & m-2 & \cdots & 2 & 1\\
       -1 & m-1 & m-2 & \cdots & 2 & 1\\
       -1 &  -2 & m-2 & \cdots & 2 & 1\\
       \vdots & \vdots & \vdots & \ddots & \vdots & 1\\
       -1 &  -2 & -3 & \cdots & 2 & 1\\
       -1 &  -2 & -3 & \cdots & -(m-1) & 1\\
      \myhrule
       -1 &  -2 & -3 & \cdots & -(m-1) & -m
    \end{bmatrix}\in\R^{(m+1) \times m}.
  \end{align*}
  Then it is clear that $D^\top(DD^\top)^{-1}v$ computes the difference of two cumulative sum vectors of $v$, $c^1 = \texttt{cumsum}((1,2,\ldots,m) \odot \texttt{reverse}(v))$ and $c^2 = \texttt{cumsum}((1,2,\ldots,m) \odot v)$, \ie $\bigl(D^\top(DD^\top)^{-1}v\bigr)^\top= (c^1,0) - (0,c^2)$, where $\texttt{cumsum}$ denotes the cumulative sum operation and \texttt{reverse} reverses the order of a vector.
\end{proof}

\begin{proposition}
  The overall complexity to solve the sorted problem \eqref{eq:maxksum_projection_sort} by \cref{alg:projection_maxksum_lcp}  is $O(n)$.
\end{proposition}
\begin{proof}
  
  We first cite a classical result regarding the number of pivots needed to identify the optimal basis.
  By \cite[Theorem 2]{cottle1972monotone}, for every $q\geq0$ and for every $d$, since $M$ is a symmetric, positive definite $Z$-matrix, the solution map ${z}(\lambda)$ is a point-to-point,  nondecreasing and  convex piecewise-linear function of $\lambda$ such that ${z}(\lambda_1)\leq {z}(\lambda_2)$ for $0\leq\lambda_1\leq\lambda_2$.
  Therefore, once a variable $z_i$ becomes basic, it remains in the basis $\xi$ for every subsequent pivot.
  Since each $z_i$ is monotone nondecreasing and piecewise-linear, the basis $\xi^t$ in iteration $t$ remains optimal over an interval $[\lambda^{t},\lambda^{t+1}]$.
  As a result, %
  the interval $[0,+\infty)$ can be partitioned into $n$ pieces $S^0\cup\cdots\cup S^{n-1}$ with progressively larger bases ($\abs{\xi^0}\leq\cdots\leq\abs{\xi^{n-1}}$) such that the solution $z(\lambda^t)$ satisfies the LCP optimality conditions \eqref{eq:plcp} for all $\lambda^t\in S^t$.   Since the constraints in \eqref{eq:maxksum_projection_sort} are linearly independent, the  optimal dual variables are unique and in particular the optimal $\bar{\lambda}$ is finite in \eqref{eq:plcp_primal}.
  
  Given solution $z(\lambda^t)$ in iteration $t$, each subsequent iteration performs three steps: (i) determining the next breakpoint $\lambda^{t+1}$ (lines 2-5); (ii) checking whether or not the budget constraint is satisfied for some $\lambda\in[\lambda^t,\lambda^{t+1}]$ (lines 7-10); and (iii) updating the solution $z(\lambda^{t+1})$ for the new breakpoint $\lambda^{t+1}$ if the budget constraint is not satisfied (lines 12-23).
  Each step can be performed in $O(1)$ cost due to the tridiagonal structure of $M$ and the sparse structure of $d$.
  Detailed justification relies on the Sherman-Morrison and Schur complement identities and is provided in \cref{apx:detail:plcp}.
  Finally, the cost required to recover the primal solution from optimal dual $z(\bar\lambda)$ with basis $\bar\xi$ via $\bar{x} = \sorth{x}^0 - \ind_k\lambda + D^\top z(\bar\lambda)$ is readily seen to be $O(\abs{\bar\xi})$ by the structure of $D$, using \cref{lem:Btz}.
  Thus, the overall complexity is $O(n)$.
\end{proof}

\subsection{An early-stopping grid-search approach}

Since our second algorithm depends on the framework described in \cite{wu2014moreau}, we reproduce some background.
Recall that the constraint $\maxsum_k{x}\leq r$ can be represented by finitely many linear inequalities,
so the following KKT conditions are necessary and sufficient for characterizing the unique solution $\bar x$ and its multiplier $\bar \lambda$ to the sorted problem \eqref{eq:maxksum_projection_sort}:
\begin{subequations}
  \label{eq:kkt_maxksum_projection_sort}
  \begin{align}
    \bar x &= \sorth{x}^{0} - \bar\lambda \mu\;\;\text{for some $\mu\in\partial\maxsum_k{\bar x}$},\\
    0 &\leq [\, r-\maxsum_k{\bar x}\, ] \perp \bar\lambda \geq0.
  \end{align}
\end{subequations}
Recall the index-pair $(k_0,k_1)$ of $x^0$ associated with $k$ in \eqref{eq:order_structure} and its related sets
\begin{align}
  \alpha\coloneqq\{1,\ldots,k_0\},\quad\beta\coloneqq\{k_0+1,\ldots,k_1\},\quad\gamma\coloneqq\{k_1+1,\ldots,n\}.
\end{align}
For any $x^0\in\R^n$ where $\sorth x^0$ satisfies the order structure \eqref{eq:order_structure}, Lemma 2.2 in \cite{wu2014moreau} gives
\begin{align}\label{eq:maxksum_subdifferential}
  \partial\maxsum_k{\sorth x^0} = \bigl\{\mu\in\R^n :\; 
    \mu_\alpha=\ind_{\abs{\alpha}},\;
    \mu_\beta \in \phi_{k_1-k_0,k-k_0},\;
    \mu_\gamma=0
  \bigr\},
\end{align}
where for $m_1\geq m_2$,
\begin{align}\label{eq:phi}
  \phi_{m_1,m_2} \coloneqq \bigl\{w\in\R^{m_1} : 0\leq w\leq\ind,\;\ind^\top w = m_2\}\bigr. .
\end{align}
The procedure in \cite{wu2014moreau} utilizes \eqref{eq:maxksum_subdifferential} to design a finite-termination algorithm that finds a subgradient $\bar\mu$ with indices $\bar k_0$ and $\bar k_1$ so that $\bar\mu\in\partial\maxsum_k{\bar x}$.

\subsubsection{KKT conditions}
\label{sec:esgs_kkt_conditions}
To solve the KKT conditions \eqref{eq:kkt_maxksum_projection_sort}, there are two cases.
If $\maxsum_k{x^0}\leq r$, then $\bar x = \sorth{x}^0$ and $\bar \lambda=0$ is the unique solution.
On the other hand, if $\maxsum_k{x^0}>r$, then $\maxsum_k{\bar x} = r$ and $\bar\lambda>0$.
This holds because if $\bar\lambda=0$, then $\bar x = \sorth{x}^0$ by stationarity, resulting in contradiction: by assumption, $\maxsum_k{x^0}>r$, but by primal feasibility $\maxsum_k{\sorth{x}^0} = \maxsum_k{\bar x}\leq r$.

Let us now focus on the solution method for the second case where $\maxsum_k{x^0}>r$.
The KKT conditions \eqref{eq:kkt_maxksum_projection_sort} can be expressed as
\begin{align}
  \label{eq:kkt_maxksum_projection_sort_2}
  \begin{array}{r*{1}{wc{0mm}}lcr*{1}{wc{0mm}}ll}
    \bar x_{\bar\alpha} &=& \sorth x^0_{\bar\alpha} - \bar\lambda\ind_{\bar\alpha} &\quad & \bar \mu_{\bar\alpha} &=& \ind_{\bar\alpha}\\*
    \bar x_{\bar k_0} &>& \bar\theta > \bar x_{\bar k_1+1} &\quad & \bar \mu_{\bar\beta} &\in& \phi_{\bar k_1-\bar k_0,k-\bar k_0}\\*
    \bar x_{\bar\beta} &=& \bar\theta\ind_{\bar\beta} &\quad & \bar \mu_{\bar\gamma} &=& 0\\*
    r &=& \ind_{\bar\alpha}^\top \sorth x^0_{\bar\alpha} - \bar k_0\bar \lambda + (k-\bar k_0)\bar\theta &\quad & \bar\lambda &>& 0
  \end{array}
\end{align}
for appropriate indices $\bar k_0$ and $\bar k_1$.
Based on \eqref{eq:maxksum_subdifferential}, any candidate index-pair $(k_0',k_1')$ gives rise to a candidate primal solution, which we denote by $x'(k_0',k_1')$, by solving a 2-dimensional linear system in variables $(\theta',\lambda')$ derived from summing the $\alpha'$ and $\beta'$ components of the stationarity conditions; see \cref{apx:detail:esgs} for more detail.
Where it is clear from context, we simplify notation by dropping the dependence on $(k_0',k_1')$ and let $x'$ denote $x'(k_0',k_1')$.
On the other hand, we overload notation to let $\lambda(k_0',k_1')$ and $\theta(k_0',k_1')$ denote the solution to the linear system associated with index pair $(k_0',k_1')$.
Using the form of the KKT conditions, we can recover the candidate primal solution $x'$ as follows:
\begin{subequations}
  \label{eq:candidate_x_k0k1}
  \begin{align}
    x'_{\alpha'} &= \sorth x_{\alpha'}^0 - \lambda'\ind_{\alpha'},\;\;
    x'_{\beta'} = \theta'\ind_{\beta'},\;\;
    x'_{\gamma'} = \sorth x_{\gamma'}^0\\*
    \label{eq:candidate_thetalambdarho_k0k1}
    \theta' &= \Big(k_0' \ind_{{\beta'}}^\top \sorth x_{\beta'}^0 - (k-k_0')\bigl( \ind_{{\alpha'}}^\top \sorth x_{\alpha'}^0 - r \bigr)\Big) / \rho'\\*
    \lambda' &= \Big( (k-k_0')\ind_{{\beta'}}^\top \sorth x^0_{\beta'} + (k_1'-k_0') \bigl( \ind_{\alpha'}^\top \sorth x_{\alpha'}^0 - r \bigr) \Big) / \rho'\\*
    \rho' &= k_0'(k_1'-k_0') + (k-k_0')^2.
  \end{align}
\end{subequations}

Based on \eqref{eq:kkt_maxksum_projection_sort_2}, the candidate solution $(x',\lambda',\theta')$ is optimal if and only if the following five conditions hold:
\begin{subequations}
  \label{eq:maxksum_projection_kkt}
	\begin{align}
		\lambda'>0,\label{eq:maxksum_projection_kkt1}\\*
		\sorth x_{k_0'}^0>\theta'+\lambda',\label{eq:maxksum_projection_kkt2}\\*
		\theta'+\lambda'\geq \sorth x_{k_0'+1}^0,\label{eq:maxksum_projection_kkt3}\\*
		\sorth x_{k_1'}^0\geq\theta',\label{eq:maxksum_projection_kkt4}\\*
		\theta'>\sorth x_{k_1'+1}^0\label{eq:maxksum_projection_kkt5}.
	\end{align}
\end{subequations}
See \cref{apx:detail:esgs} for detail on the above step, which uses the form of the subdifferential from \eqref{eq:maxksum_subdifferential}.
To obtain a solution $(\bar x,\bar\lambda)$ to the KKT conditions \eqref{eq:kkt_maxksum_projection_sort}, it is sufficient to perform a grid search over $k_0\in \{0, \ldots, k-1\}$ and $k_1\in \{k, \ldots, n\}$ and check \eqref{eq:maxksum_projection_kkt}, which is the approach adopted in \cite[Algorithm 4]{wu2014moreau}.

On the other hand, detailed inspection of monotonicity properties of the reduced KKT conditions \eqref{eq:maxksum_projection_kkt} yield
an $O(n)$ primal-based procedure for solving the sorted problem \eqref{eq:maxksum_projection_sort} as outlined in \cref{alg:projection_maxksum_snake}.
Instead of executing a grid-search over $\{0, \ldots, k-1\}\times \{k, \ldots, n\}$, our procedure exploits the hidden properties of the KKT conditions.  We construct a path from $(k_0,k_1)=(k-1,k)$ to $(\bar{k}_0,\bar{k}_1)$, composed solely of decrements to $k_0$ and increments to $k_1$. This path maintains satisfaction of KKT conditions 1, 3, and 4,  and seeks a pair $(k_0, k_1)$ that  yields $(x,\lambda)$ satisfying complementarity, \ie KKT conditions 2 and 5.
In general, this procedure generates a different sequence of ``pivots'' from \cref{alg:projection_maxksum_lcp} and to the best of our knowledge does not exist in the current literature.

\subsubsection{Implementation}
The procedure, as summarized in \cref{alg:projection_maxksum_snake}, is very simple.
The algorithm's behavior is depicted in \cref{fig:esgs_schematic}, in which a path from $(k-1,k)$ to $(\bar{k}_0,\bar{k}_1)$ is generated by following the orange and blue arrows.
\refone{Due to the ordering properties of the problem, a sequence of index-pairs can be constructed in which $\eqref{eq:maxksum_projection_kkt1}\land\eqref{eq:maxksum_projection_kkt2}\land\eqref{eq:maxksum_projection_kkt3}$ always hold.
Since the optimal LCP basis is contiguous (as shown in \cref{apx:detail:plcp}) and again due to certain monotonicity properties, updates to the index-pair $(k_0,k_1)$ either decrement $k_0$ or increment $k_1$.
The former occurs when $\eqref{eq:maxksum_projection_kkt2}$ fails to hold, the latter occurs when $\eqref{eq:maxksum_projection_kkt5}$ fails to hold, and the algorithm terminates when both $\eqref{eq:maxksum_projection_kkt2}\land\eqref{eq:maxksum_projection_kkt5}$ hold.}

\begin{minipage}{0.52\textwidth}
  \begin{algorithm}[H]
    \caption{\small ESGS projection onto $\maxsumball{k}{r}$.}\label{alg:projection_maxksum_snake}
    \begin{algorithmic}[1]
      \Initialize{
        \noindent
        \begin{enumerate}[label=(\roman*),topsep=0pt,itemsep=-1ex,partopsep=1ex,parsep=1ex]
          \item If $\maxsum_k{ x^0}\leq r$, return $\bar{x}=x^0$.
          \item Else, if $k=1$, set $\bar{x} = \min\bigl\{x^0,r\bigr\}$.
          \item Else, if $k=n$, set $\bar{x} = x^0 - \ind_n\cdot(\ind_n^\top x^0-r)/n$.
          \item Else, set $k_0=k-1$, $k_1=k$, and $\texttt{solved}=0$.
        \end{enumerate}
      }
      \While{true}
        \State Compute candidate $\theta$ and $\lambda$ and KKT indicators:
        \Nest
          \State{$\hphantom{\theta+\lambda}\mathllap{\rho} = k_0(k_1-k_0) + (k-k_0)^2$}
          \State{$\hphantom{\theta+\lambda}\mathllap{\theta} = \bigl(k_0 \ind_{{\beta}}^\top  x_\beta^0 - (k-k_0)\bigl( \ind_{{\alpha}}^\top  x_\alpha^0 - r \bigr)\bigr)/\rho$}
          \State{$\hphantom{\theta+\lambda}\mathllap{\theta+\lambda} = \bigl( k\ind_{{\beta}}^\top  x^0_\beta + (k_1-k) \bigl( \ind_{\alpha}^\top  x_\alpha^0 - r \bigr) \bigr) / \rho$}
          \State{$\hphantom{\theta+\lambda}\mathllap{\ikkt_2} =  x_{k_0}^0 > \theta+\lambda,\quad \ikkt_5 = \theta >  x_{k_1+1}^0$}
        \EndNest
        \State Check KKT conditions:
        \Nest
          \IIf{$\ikkt_2\land\ikkt_5$} $\texttt{solved}=1$
          \ElseIIf{$\ikkt_2$} $k_1\gets k_1+1$
          \ElseIIf{$\neg\ikkt_2$} $k_0\gets k_0-1$
          \EndIIf
          \If{$\texttt{solved}=1$}
            \State \Return $\bar x$ via $\bar x_\alpha =  x_\alpha^0 - \lambda\ind_\alpha$, $\bar x_\beta = \theta\ind_\beta$, and $\bar x_\gamma =  x_\gamma^0$.
          \EndIf
        \EndNest
      \EndWhile
    \end{algorithmic}
  \end{algorithm}
\end{minipage}
\hfill
\begin{minipage}{0.47\textwidth}
  \begin{figure}[H]
    \centering
    \begin{tikzpicture}[scale=1.0]
      \begin{scope}[shift={(7,0)}]
        \begin{scope}[rotate=180]
          \draw[step=0.5cm,opacity=0.25] (0.4,0.4) grid (4,4);
  
          \draw[-,thick] (4.2,-0.75) -- (0.4,-0.75);
          \draw[dotted,thick] (0.4,-0.75) -- (0.1,-0.75);
          \draw[->,thick] (0.1,-0.75) -- (-0.7,-0.75) node[right]{$k_1$};
          \foreach \i/\j in {3.75/k,1.75/\bar{k}_1,-0.25/n}{
            \draw (\i,-0.7) -- ++ (0,-0.1) node[above=1mm] {$\j$};
          }
          
          \draw[-,thick] (4.2,-0.75) -- (4.2,0.1);
          \draw[dotted,thick] (4.2,0.1) -- (4.2,0.4);
          \draw[->,thick] (4.2,0.4) -- (4.2,4.2) node[below]{$k_0$};
          \foreach \i/\j in {-0.25/0,0.75/\bar{k}_0,3.75/k-1}{
            \draw (4.25,\i) -- ++ (-0.1,0) node[left=2mm] {$\j$};
          }
          
          \foreach\j in {0.0,0.5,...,4}{%
            \draw[dotted] (0.1,\j) -- (0.4,\j);
          }
          \foreach \i/\j [evaluate=\i as \x using ((4-\i)*8-3)/8, evaluate=\j as \y using ((4-\j)*8-3)/8] in {0/1,0.5/1.5,1/2.5,1.5/3}{%
            \fill[orange!90!black,opacity=0.75] (\x,\y) rectangle ++(2/8,2/8);
            \fill[black] (\x+1/8,\y+1/8) circle (0.05);
          }
          \foreach \i/\j [evaluate=\i as \x using ((4-\i)*8-3)/8, evaluate=\j as \y using ((4-\j)*8-3)/8] in {2.5/2.5,3/2.5} {
            \fill[orange!90!black,opacity=0.75] (\x,\y) rectangle ++(2/8,2/8);
          }
          \foreach \i/\j [evaluate=\i as \x using ((4-\i)*8-3)/8, evaluate=\j as \y using ((4-\j)*8-3)/8] in {0/0,0/0.5,0.5/1,1/1.5,1/2,1.5/2.5} {
            \fill[blue,opacity=0.5] (\x,\y) rectangle ++(2/8,2/8);
            \fill[black] (\x+1/8,\y+1/8) circle (0.05);
          }
          \foreach \i/\j [evaluate=\i as \x using ((4-\i)*8-2)/8, evaluate=\j as \y using ((4-\j)*8-2)/8] in {0/1,0.5/1.5,1/2.5,1.5/3} {
            \draw[->,line width=0.4mm,orange!60!black!60!red] (\x-1/16,\y) -- +(-3/8+1/16,0);
          }
          \foreach \i/\j [evaluate=\i as \x using ((4-\i)*8-2)/8, evaluate=\j as \y using ((4-\j)*8-2)/8] in {0/0,0/0.5,0.5/1,1/1.5,1/2,1.5/2.5} {
            \draw[->,line width=0.4mm,blue] (\x,\y-1/16) -- +(0,-3/8+1/16);
          }
          
          \foreach \i/\j [evaluate=\i as \x using ((4-\i)*8-3)/8, evaluate=\j as \y using ((4-\j)*8-3)/8] in {0.5/0,1.5/0.5,2/1,2/1.5,2/2,2/2.5,1/4} {
            \fill[purple,opacity=0.5] (\x,\y) rectangle ++(2/8,2/8);
          }
          
          \draw[dashed] (1.75,-0.5) -- (1.75,0.5+1/8);%
          \draw[dashed] (1.75,0.75+1/8) -- (1.75,4.0);%
          \draw[dashed] (4,0.75) -- (2.25,0.75);%
          \draw[dashed] (1.5+1/8,0.75) -- (-0.5,0.75);%
          \foreach \i/\j [evaluate=\i as \x using ((4-\i)*8-3)/8, evaluate=\j as \y using ((4-\j)*8-3)/8] in {2/3}{
            \fill[orange!90!black,opacity=0.75,thick] (\x,\y) -- (\x+2/8,\y) -- (\x+2/8,\y+2/8);
            \fill[purple,opacity=0.5,thick] (\x+2/8,\y+2/8) -- (\x,\y+2/8) -- (\x,\y);
            \draw[orange!90!black,thick] (\x,\y) -- (\x+2/8,\y) -- (\x+2/8,\y+2/8);
            \draw[purple,thick] (\x+2/8,\y+2/8) -- (\x,\y+2/8) -- (\x,\y);
            \draw[purple,thick] (\x+2/8,\y+2/8) -- (\x,\y);
            \fill[black] (\x+1/8,\y+1/8) circle (0.05);
          }
        \end{scope}
        \begin{scope}[rotate=0]
          \draw[step=0.5cm,opacity=0.25] (-0.1,-0.1) grid (0.5,0.5);
          \foreach\j in {0.5}{%
            \draw[dotted] (-0.4,\j) -- (-0.1,\j);
          }
        \end{scope}
        \begin{scope}[rotate=90]
          \draw[step=0.5cm,opacity=0.25] (-0.1,0.4) grid (0.5,4);
          \foreach\j in {0,0.5,...,4}{%
            \draw[dotted] (-0.4,\j) -- (-0.1,\j);
          }
        \end{scope}
        \begin{scope}[rotate=270]
          \draw[step=0.5cm,opacity=0.25] (0.4,-0.1) grid (4,0.5);
          \foreach\j in {0.5}{%
            \draw[dotted] (0.1,\j) -- (0.4,\j);
          }
          \foreach \i/\j [evaluate=\i as \x using ((4-\i)*8-3)/8, evaluate=\j as \y using ((4-\j)*8-3)/8] in {0.5/3.5} {
            \fill[orange!90!black,opacity=0.75] (\x,\y) rectangle ++(2/8,2/8);
          }
        \end{scope}
      \end{scope}
    \end{tikzpicture}
    \caption{
      \small 
      Schematic for \cref{alg:projection_maxksum_snake}.
      Orange shading indicates that $\ikkt_2\land\ikkt_3$ holds;
      blue shading indicates that $\neg\ikkt_2\land\ikkt_3$ holds;
      red shading indicates that $\ikkt_4\land\ikkt_5$ holds;
      black circles trace the trajectory taken by ESGS and indicate that $\ikkt_1\land\ikkt_3\land\ikkt_4$ holds.
    }
    \label{fig:esgs_schematic}
  \end{figure}
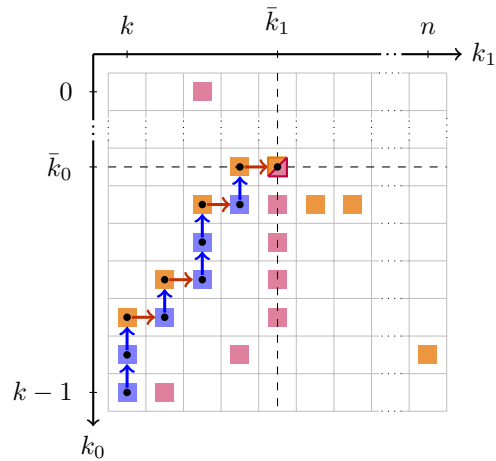
\end{minipage}

\subsubsection{Analysis}
\label{sec:esgs_analysis}
The analysis of \cref{alg:projection_maxksum_snake} builds on the framework introduced in \cite{wu2014moreau} for performing a grid search over candidate index-pairs $(k_0,k_1)$.
The grid search proceeds with outer loop over $k_1\in \{k, \ldots, n\}$ and inner loop over $k_0 \in \{0,\ldots,k-1\}$, though similar analysis holds when reversing the search order (but is not discussed further).
The new observation leveraged in \cref{prop:snake_complexity} is that the order structure of the KKT conditions induces various forms of monotonicity in the KKT residuals across $k_0$ and $k_1$.
This monotonicity holds globally (\ie across all $k_0$ for given $k_1$ or all $k_1$ for given $k_0$) and implies that the local KKT information at a candidate index-pair $(k_0,k_1)$ provides additional information about which indices can be ``skipped'' in the full grid search procedure.

For example, in the sample trajectory traced in \cref{fig:esgs_schematic}, the algorithm starts at $(k_0,k_1)=(k-1,k)$ and searches over all $k_0$ in the first column for an index-pair that satisfies KKT conditions 2 and 3.
Suppose that such an index is found in some row, which we denote by $k_0^*(k)$ to emphasize the dependence of such a row on the column $k_1=k$.
Further suppose that the pair $(k_0^*(k),k)$ does not satisfy \emph{all} of the KKT conditions.
Then monotonicity implies that no candidate index-pair $(k_0',k)$ for $k_0'\leq k_0^*(k)$ can satisfy \emph{all} of the KKT conditions, providing an early termination to the inner search over $k_0$ associated with column $k_1=k$.
This justifies terminating the current search over $k_0$ in column $k_1=k$ and increasing $k_1\gets k_1+1$.
At this point, the full grid search would begin again at $(k-1,k+1)$.
However, another monotonicity property of the KKT residuals implies that all $k_0'>k_0^*(k)$ cannot be optimal, which justifies starting the grid search ``late'' at $(k_0^*(k),k+1)$ rather than $(k-1,k+1)$.
Based on these ideas, the main effort in establishing the correctness of \cref{alg:projection_maxksum_snake} is in showing how to use monotonicity properties of the KKT residuals to justify transitions ``up'' or ``right''.
The first idea is summarized in \cref{lem:earlystop} and the second is summarized in \cref{lem:latestart}.
\cref{prop:snake_complexity} is obtained by combining these earlier results, immediately giving the desired complexity analysis.

We begin the analysis of \cref{alg:projection_maxksum_snake} by making the following assumption.
\begin{assumption}[Strict projection]
  \label{assn:strict_projection}
  For the initial vector $x^0\in\R^n$, it holds that $\maxsum_k{x^0}>r$.
\end{assumption}
For simplicity of notation, we also assume that the initial point is sorted $x^0=\sorth{x}^0$ and drop the sorting notation in the remainder of this section (and its proofs).

At a candidate index-pair $(k_0,k_1)$, define the KKT satisfaction indicators
\begin{subequations}
  \label{eq:kkt_surplus}
  \begin{alignat}{4}
    \ikkt_1(k_0,k_1) &\coloneqq \indicator_{t>0}[\big]{\lambda(k_0,k_1)},\\*
    \ikkt_2(k_0,k_1) &\coloneqq \indicator_{t>0}[\big]{x^0_{k_0} - (\theta+\lambda)(k_0,k_1)},\\*
    \ikkt_3(k_0,k_1) &\coloneqq \indicator_{t\geq0}[\big]{(\theta+\lambda)(k_0,k_1) - x^0_{k_0+1}},\\*
    \ikkt_4(k_0,k_1) &\coloneqq \indicator_{t\geq0}[\big]{x^0_{k_1} - \theta(k_0,k_1)},\\*
    \ikkt_5(k_0,k_1) &\coloneqq \indicator_{t>0}[\big]{\theta(k_0,k_1) - x^0_{k_1+1}},
  \end{alignat}
\end{subequations}
where $\indicator_{t\in S}{x} = \texttt{true}$ is $x\in S$ and \texttt{false} otherwise; and where $\lambda(\bullet,\bullet)$ and $\theta(\bullet,\bullet)$ denote the values of $\lambda$ and $\theta$ corresponding to a particular index-pair.
By \eqref{eq:maxksum_projection_kkt}, an index-pair $(k_0,k_1)$ is optimal if and only if
\begin{align}
  \ikkt_1(k_0,k_1)\land
  \ikkt_2(k_0,k_1)\land
  \ikkt_3(k_0,k_1)\land
  \ikkt_4(k_0,k_1)\land
  \ikkt_5(k_0,k_1).
\end{align}
Note that the presence of strict inequalities precludes arguments that appeal to linear programming.
Instead we conduct a detailed study of the KKT conditions directly.

We begin by arguing that starting from $(k_0,k_1)=(k-1,k)$, we may ``forget'' about checking conditions 1, 3, and 4 and instead only seek satisfaction of conditions 2 and 5.
We refer to \cref{fig:esgs_schematic} when referencing ``rows'' and ``columns'' of the search space.
We also delay verification of claims involving ``direct computation'' to \cref{apx:proposedalgorithms:esgs} but provide references in the text.

\begin{lemma}
  \label{lem:satisfy_134}
  Beginning from $(k_0,k_1)=(k-1,k)$, the trajectory taken by \cref{alg:projection_maxksum_snake} always satisfies KKT conditions 1, 3, and 4.
\end{lemma}
\begin{proof}
  We proceed inductively.
  Let $k_0=k-1$ and $k_1=k$ be the initial point and note that KKT conditions 1, 3, and 4 hold at $(k_0,k_1)$ by direct computation (\cref{claim:initial_134}).
  This establishes the base case.

  Next, let $(k_0,k_1)$ be any candidate index-pair where $\ikkt_1(k_0,k_1)$, $\ikkt_3(k_0,k_1)$, and $\ikkt_4(k_0,k_1)$ hold and $(k_0',k_1')$ be the next iterate generated by the procedure.
  To show that KKT conditions 1, 3, and 4 at $(k_0',k_1')$, consider the four cases based on which of the remaining KKT conditions (2 and 5) hold at $(k_0,k_1)$.
  For shorthand, let $\ikkt_i = \ikkt_i(k_0,k_1)$ and $\ikkt_i' = \ikkt_i(k_0',k_1')$ for $i\in\{1,\ldots,5\}$.
  \begin{enumerate}[label=(\roman*)]
    \item $\ikkt_2\land\ikkt_5$:
    The algorithm terminates at the current iterate, which is the unique solution.

    \item $\ikkt_2\land\neg\ikkt_5$: $k_1\gets \min \bigl\{k_1+1,n\bigr\}$.
    The index-pair $(k_0',k_1')=(k_0,k_1+1)$ is the next point generated by the algorithm.
    At $(k_0',k_1')$, the conditions $\ikkt_1'$, $\ikkt_3'$, and $\ikkt_4'$ hold because of the following argument.
    \begin{itemize}
      \item $\ikkt_1'$:
      Direct computation (\crefenv{lem:Delta_k0k1}{lem:Delta_k1_15}) shows that $\neg\ikkt_5 \iff \Delta_{k_1}\lambda(k_0,k_1)\geq0 \iff \lambda(k_0,k_1+1)\geq \lambda(k_0,k_1)$.
      Since $\lambda(k_0,k_1)>0$ by assumption, it holds that $\lambda(k_0',k_1') = \lambda(k_0,k_1+1)>0$.
      
      \item $\ikkt_3'$:
      By definition, $\ikkt_3 \iff (\theta+\lambda)(k_0,k_1) \geq x_{k_0+1}^0$, so it suffices to show that $(\theta+\lambda)(k_0',k_1')\geq(\theta+\lambda)(k_0,k_1)$.
      The desired condition is equivalent to $\Delta_{k_1}(\theta+\lambda)(k_0,k_1)\geq0$.
      Direct computation (\crefenv{lem:Delta_k0k1}{lem:Delta_k1_35}) shows that $\Delta_{k_1}(\theta+\lambda)(k_0,k_1)\geq0 \iff \neg\ikkt_5$.
      Since $\neg\ikkt_5$ holds by assumption, the desired condition $\ikkt_3'$ holds.
      
      \item $\ikkt_4'$:
      Direct computation (\crefenv{lem:linking_23_and_45}{lem:linking_45_2}) shows that $\neg\ikkt_5\iff \ikkt_4'$.
    \end{itemize}

    \item $\neg\ikkt_2\land\neg\ikkt_5$: $k_0\gets \max\bigl\{k_0-1,0\bigr\}$.
    The index-pair $(k_0',k_1')=(k_0-1,k_1)$ is the next point generated by the algorithm.
    At $(k_0',k_1')$, the conditions $\ikkt_1'$, $\ikkt_3'$, and $\ikkt_4'$ hold because of the following argument.
    \begin{itemize}
      \item $\ikkt_1'$:
      It suffices to show that $\lambda(k_0-1,k_1)\geq \lambda(k_0,k_1)$, \ie $\Delta_{k_0}\lambda(k_0-1,k_1)\leq0$.
      Direct computation (\crefenv{lem:Delta_k0k1}{lem:Delta_k0_21}) shows that $\neg\ikkt_2\iff\Delta_{k_0}\lambda(k_0-1,k_1)\leq0$.
      
      \item $\ikkt_3'$:
      Direct computation (\crefenv{lem:linking_23_and_45}{lem:linking_23_2}) shows that $\neg\ikkt_2 \iff \ikkt_3'$.
      
      \item $\ikkt_4'$:
      Since $\ikkt_4\iff x_{k_1}^0 \geq \theta(k_0,k_1)$ and $\ikkt_4'\iff x_{k_1}^0 \geq \theta(k_0-1,k_1)$, it suffices to show that $\theta(k_0,k_1)\geq\theta(k_0-1,k_1)$, \ie $\Delta_{k_0}\theta(k_0-1,k_1)\geq0$.
      Direct computation (\crefenv{lem:Delta_k0k1}{lem:Delta_k0_24}) shows that $\neg\ikkt_2 \iff \Delta_{k_0}\theta(k_0-1,k_1)\geq0$.
    \end{itemize}
    Further, observe that transitions from case (iii) to case (iv) cannot occur, \ie that $\neg\ikkt_2\land\neg\ikkt_5$ cannot transition to $\neg\ikkt_2'\land\ikkt_5'$.
    To show this, it suffices to show that $\neg\ikkt_5'$ must hold.
    This is true because of the following argument:
    \begin{align*}
      \neg\ikkt_2&\overset{(a)}{\iff}\Delta_{k_0}\theta(k_0-1,k_1)\geq0
      \iff \theta(k_0-1,k_1)\leq \theta(k_0,k_1) \overset{(b)}{\leq} x_{k_1+1}^0
    \end{align*}
    where $(a)$ follows from direct computation (\crefenv{lem:Delta_k0k1}{lem:Delta_k0_24})
    and $(b)$ follows from the definition of $\neg\ikkt_5$.
    Thus $\neg\ikkt_5'\equiv\neg\ikkt_5(k_0-1,k_1)$ must hold.
    
    \item $\neg\ikkt_2\land\ikkt_5$: $k_0\gets \max\bigl\{k_0-1,0\bigr\}$.
    By the argument in case (iii), transitions to case (iv) must come from case (ii).
    The next point generated by the algorithm is $(k_0',k_1')=(k_0-1,k_1)$.
    At $(k_0',k_1')$, the conditions $\ikkt_1'$, $\ikkt_3'$, and $\ikkt_4'$ hold because of the following argument.
    
    \begin{itemize}
      \item $\ikkt_1'$:
      Direct computation (\crefenv{lem:Delta_k0k1}{lem:Delta_k0_21}) shows that $\neg\ikkt_2\iff \Delta_{k_0}\lambda(k_0-1,k_1)\leq0$, which implies that $\lambda(k_0-1,k_1)\geq\lambda(k_0,k_1)>0$.
      
      \item $\ikkt_3'$:
      Direct computation (\crefenv{lem:linking_23_and_45}{lem:linking_23_2}) shows that $\neg\ikkt_2 \iff \ikkt_3$.

      \item $\ikkt_4'$:
      Direct computation (\crefenv{lem:Delta_k0k1}{lem:Delta_k0_24}) shows that $\neg\ikkt_2\iff \Delta_{k_0}\theta(k_0-1,k_1)\geq0$, which implies that $\theta(k_0-1,k_1)\leq\theta(k_0,k_1)\leq x_{k_1}^0$.
    \end{itemize}    
  \end{enumerate}
  Finally, setting $x^0_0\coloneqq +\infty$ and $x_{n+1}^0\coloneqq-\infty$, it is clear that the algorithm is guaranteed to terminate since the last index-pair possibly scanned is $(k_0,k_1)=(0,n)$, and if $k_0=0$, then $\ikkt_2$; otherwise, if $k_1=n$, then $\ikkt_5$.
  Therefore, the procedure maintains satisfaction of KKT conditions 1, 3, and 4 throughout its trajectory.
\end{proof}

It remains only to identify an index-pair that satisfies both conditions 2 and 5.
Toward this end, we justify stopping the inner $k_0$ loop early (within a column) if an index-pair is found that satisfies condition 2 by showing that no index-pair ``above'' the current iterate (in the same column) can satisfy condition 2.
\begin{lemma}[Early stop]
  \label{lem:earlystop}
  Suppose that KKT condition 2 holds at a candidate index-pair $(k_0,k_1)$ along the trajectory of \cref{alg:projection_maxksum_snake} (\eg in case (ii) of the proof of \cref{lem:satisfy_134}).
  Then there does not exist an index-pair $(k_0',k_1)$ that satisfies KKT condition 2 for $k_0'<k_0$.
\end{lemma}
\begin{proof}
  The claim holds because for any $k_1$, there is at most one element $k_0'\in\{0,\ldots,k-1\}$ such that $\ikkt_2(k_0',k_1)$ and $\ikkt_3(k_0',k_1)$ both hold.
  The proof is by contradiction.
  Fix any $k_1\in\{k,\ldots,n\}$, and suppose that there exist two indices $k_0'$ and $k_0''$ for which $\ikkt_2(k_0',k_1)\land \ikkt_3(k_0',k_1)$ and $\ikkt_2(k_0'',k_1)\land \ikkt_3(k_0'',k_1)$ hold.
  By direct computation (\cref{lem:linking_22j}), the sets $K^2(k_1)$ ($K^{\neg2}(k_1)$) and $K^3(k_1)$ ($K^{\neg3}(k_1)$) where KKT conditions 2 and 3 are satisfied (not satisfied) are contiguous, so it suffices to consider $k_0''=k_0'-1$.
  By direct computation (\crefenv{lem:linking_23_and_45}{lem:linking_23_2}), $\ikkt_2(k_0',k_1)\iff\neg\ikkt_3(k_0'-1,k_1)=\neg\ikkt_3(k_0'',k_1)$, a contradiction.
  Therefore there can be at most one element that satisfies KKT conditions 2 and 3, and stopping early is justified.
\end{proof}
Finally, we justify the ``late start'' of the inner $k_0$ loop after moving from column $k_1$ to $k_1+1$.
\begin{lemma}[Late start]
  \label{lem:latestart}
  For a given $k_0$, consider a transition from $k_1$ to $k_1+1$ along the trajectory of \cref{alg:projection_maxksum_snake}.
  The optimal solution $(\bar{k}_0,\bar{k}_1)$ must satisfy $\bar{k}_0\leq k_0$ and $\bar{k}_1\geq k_1$.
\end{lemma}
\begin{proof}
  Evaluate KKT condition 2 at the new iterate $(k_0,k_1+1)$.
  If $\ikkt_2(k_0,k_1+1)$, then by the preceding two lemmas, $k_0$ is the unique element from the column associated with $k_1+1$ that satisfies both KKT conditions 2 and 3.
  Otherwise, if $\neg\ikkt_2(k_0,k_1+1)$, then by direct computation (\cref{lem:linking_22j}), any index $k_0'>k_0$ will not satisfy KKT condition 2 and therefore not be optimal.
\end{proof}

Combining our prior results, we obtain the desired conclusion.
\begin{proposition}\label{prop:snake_complexity}
  The procedure given in \cref{alg:projection_maxksum_snake} terminates at the unique solution $(\bar{k}_0,\bar{k}_1)$ in at most $n$ steps using $O(n)$ elementary operations.
\end{proposition}
\begin{proof}
  The conclusion is an immediate consequence of \cref{lem:satisfy_134,,lem:earlystop,,lem:latestart} and the form of the updates in \cref{alg:projection_maxksum_snake}.
  Since there can be at most $n-k$ transitions of the form $k_1\gets k_1+1$ and at most $k$ transitions of the form $k_0\gets k_0-1$, there are at most $n$ total transitions, and hence the proof is complete.
\end{proof}

\begin{toappendix}
  \subsection{ESGS}
  \label{apx:proposedalgorithms:esgs}
  In this section, we provide additional observations and verifications of the computations claimed in the analysis of \cref{alg:projection_maxksum_snake} in \cref{sec:esgs_analysis}.
  
  \subsubsection*{Simple observations}
  We begin by observing that for any candidate $(k_0,k_1)$ with $k_0\in\{0,\ldots,k-1\}$ and $k_1\in\{k,\ldots,n\}$, it holds that $\rho(k_0,k_1)\coloneqq k_0\cdot(k_1-k_0) + (k-k_0)^2 > 0$.
  Next we verify that conditions 1, 3, and 4 hold at the initial iterate.
  \begin{claim}[Initial candidate index-pair $(k-1,k)$]
    \label{claim:initial_134}
    For any $k\geq1$, it holds that (i) $\ikkt_1(k-1,k)$; (ii) $\ikkt_3(k-1,k)$; and $\ikkt_4(k-1,k)$.
  \end{claim}
  \begin{proof}
    \label{pf:initial_134}
    Let $k\geq1$.
    Using the fact that $\rho(k-1,k)=k>0$, the proofs follow by direct computation.
    \begin{enumerate}
      \item Using $0<\sum_{i=1}^k x_i^0 - r$ from \cref{assn:strict_projection}, it holds that
      $\rho(k-1,k)\cdot\lambda(k-1,k) = k\cdot\lambda(k-1,k) = k\cdot\bigl(x_{k}^0 + \sum_{i=1}^{k-1}x_i^0 - r\bigr) > 0$.

      \item Since $k\geq1$ and since $x^0$ is sorted, it holds that
      $\rho(k-1,k)\cdot(\theta+\lambda)(k-1,k) = k\sum_{i=k}^kx_i^0 + (k-k)\cdot\bigl(\sum_{i=1}^{k-1}x_i^0 - r\bigr) = kx_k^0 \geq k x_{k_0-1+1}^0=\rho(k-1,k)\cdot x_{k}^0$.
      
      \item Using $0<\sum_{i=1}^k  x_i^0 - r$ from \cref{assn:strict_projection}, it holds that
      $\rho(k-1,k)\cdot\theta(k-1,k) = (k-1) x_{k}^0 - (k - (k-1))\bigl(\sum_{i=1}^{k-1} x_i^0 - r\bigr)
      = k x_k^0 - \bigl(\sum_{i=1}^k  x_i^0 - r\bigr) < k x_{k}^0 = \rho(k-1,k)\cdot x_{k}^0$.
    \end{enumerate}
    Thus the claims are proved.
  \end{proof}

  \subsubsection*{Linking conditions}
  Next we summarize characterize some relationships between various KKT conditions as depicted in \cref{fig:linking}.
  \begin{figure}
    \begin{subfigure}[t]{0.45\linewidth}
      \centering
      \begin{tikzpicture}[scale=1.0]

        \begin{scope}
          \draw[<->,thin] (-0.5,-0.25) -- (-0.5,4.25);
          \draw (-0.5,3.75) node[left]{$0$};
          \draw[-,thin] (-0.45,3.75) -- (-0.55,3.75);
          \draw (-0.5,0.25) node[left]{$k-1$};
          \draw[-,thin] (-0.45,0.25) -- (-0.55,0.25);
      
          \draw (0.3,4) node[above]{$\ikkt_2$};
          \draw[step=0.5cm,opacity=0.25] (0.0,0.0) grid (0.5,4);
          \fill[green!50!black,opacity=0.75] (1/8,17/8) rectangle ++(2/8,2/8);
          \fill[green!50!black,opacity=0.75] (1/8,21/8) rectangle ++(2/8,2/8);
          \fill[green!50!black,opacity=0.75] (1/8,25/8) rectangle ++(2/8,2/8);
          \draw (2/8,59/32) node{$\vdots$};
          \fill[red,opacity=0.75] (1/8,9/8) rectangle ++(2/8,2/8);
          \fill[red,opacity=0.75] (1/8,5/8) rectangle ++(2/8,2/8);
          \draw (1+1/8,22/8) node (kkt3a) {};
          \draw (1/2-1/8,18/8) node (kkt2a) {};
          \draw (1+1/8,22/8+1/2+1/16) node (kkt3b) {};
          \draw (1/2-1/8,18/8+1/2+1/16) node (kkt2b) {};
          \draw (1+1/8,22/8+1+2/16) node (kkt3c) {};
          \draw (1/2-1/8,18/8+1+2/16) node (kkt2c) {};
          \draw (1+1/8,10/8) node (kkt3d) {};
          \draw (1/2-1/8,10/8) node (kkt2d) {};
          \draw (1+1/8,6/8) node (kkt3e) {};
          \draw (1/2-1/8,6/8) node (kkt2e) {};
          
          \connectorH[<->]{0.5}{kkt3a}{kkt2a};
          \connectorH[<->]{0.5}{kkt3b}{kkt2b};
          \connectorH[<->]{0.5}{kkt3c}{kkt2c};
          \connectorH[<-]{0.5}{kkt3d}{kkt2d};
          \connectorH[<-]{0.5}{kkt3e}{kkt2e};
        \end{scope}
        
        \begin{scope}[shift={(1,0)}]
          \draw (0.3,4) node[above]{$\ikkt_3$};
          \draw[step=0.5cm,opacity=0.25] (0.0,0.0) grid (0.5,4);
          \fill[red,opacity=0.75] (1/8,21/8) rectangle ++(2/8,2/8);
          \fill[red,opacity=0.75] (1/8,25/8) rectangle ++(2/8,2/8);
          \fill[red,opacity=0.75] (1/8,29/8) rectangle ++(2/8,2/8);
          \draw (2/8,59/32) node{$\vdots$};
          \draw (2/8,59/32+1/2) node{$\vdots$};
          \fill[green!50!black,opacity=0.75] (1/8,9/8) rectangle ++(2/8,2/8);
          \fill[green!50!black,opacity=0.75] (1/8,5/8) rectangle ++(2/8,2/8);
        \end{scope}
      \end{tikzpicture}
      \caption{
        \small Depiction of relationship between $\ikkt_2$ and $\ikkt_3$ in \cref{lem:linking_23_and_45} for given $k_1$.
        Green: indicates a condition holds.
        Red: indicates a condition does not hold.
        Arrows represent implication.
      }
      \label{fig:link2345}
    \end{subfigure}
    \hfill
    \begin{subfigure}[t]{0.45\linewidth}
      \centering
      \begin{tikzpicture}[scale=1.0]

        \draw[<->,thin] (-0.5,-0.25) -- (-0.5,4.25);
        \draw (-0.5,3.75) node[left]{$0$};
        \draw[-,thin] (-0.45,3.75) -- (-0.55,3.75);
        \draw (-0.5,1.75) node[left]{$k_0^-$};
        \draw[-,thin] (-0.45,1.75) -- (-0.55,1.75);
        \draw (-0.5,2.25) node[left]{$k_0^+$};
        \draw[-,thin] (-0.45,2.25) -- (-0.55,2.25);
        \draw (-0.5,0.25) node[left]{$k-1$};
        \draw[-,thin] (-0.45,0.25) -- (-0.55,0.25);
        
        \begin{scope}
          \draw (0.3,4) node[above]{$\ikkt_2$};
          \draw[step=0.5cm,opacity=0.25] (0.0,0.0) grid (0.5,4);
          \fill[green!50!black,opacity=0.75] (1/8,17/8) rectangle ++(2/8,2/8);
          \draw[green!50!black,opacity=0.75] (1/8,17/8) rectangle ++(2/8,14/8);
          \fill[red,opacity=0.75] (1/8,13/8) rectangle ++(2/8,2/8);
          \draw[red,opacity=0.75] (1/8,15/8) rectangle ++(2/8,-14/8);
        \end{scope}
        
        \begin{scope}[shift={(1,0)}]
          \draw[<->,thin] (1,-0.25) -- (1,4.25);
          \draw (1,3.75) node[right]{$0$};
          \draw[-,thin] (0.95,3.75) -- (1.05,3.75);
        
          \draw (1,2.25) node[right]{$k_0^+$};
          \draw[-,thin] (0.95,2.25) -- (1.05,2.25);
          \draw (1,2.75) node[right]{$k_0^-$};
          \draw[-,thin] (0.95,2.75) -- (1.05,2.75);
          \draw (1,0.25) node[right]{$k-1$};
          \draw[-,thin] (0.95,0.25) -- (1.05,0.25);
          
          \draw (0.3,4) node[above]{$\ikkt_3$};
          \draw[step=0.5cm,opacity=0.25] (0.0,0.0) grid (0.5,4);
          \fill[red,opacity=0.75] (1/8,21/8) rectangle ++(2/8,2/8);
          \draw[red,opacity=0.75] (1/8,21/8) rectangle ++(2/8,10/8);
          \fill[green!50!black,opacity=0.75] (1/8,17/8) rectangle ++(2/8,2/8);
          \draw[green!50!black,opacity=0.75] (1/8,19/8) rectangle ++(2/8,-18/8);
        \end{scope}
      \end{tikzpicture}
      \caption{
        \small Depiction of \cref{lem:linking_22j} for given $k_1$: local information (shaded) becomes global (outlined).
        Green outlines denote $K_0^2(k_1)$ and $K_0^3(k_1)$; red outlines denote $K_0^{\neg2}(k_1)$ and $K_0^{\neg3}(k_1)$.
      }
      \label{fig:link22j33j}
    \end{subfigure}
    \caption{\small Linking conditions.}
    \label{fig:linking}
  \end{figure}

  \begin{claim}[Linking 2 \& 3 and 4 \& 5]
    \label{lem:linking_23_and_45}
    It holds that
    \begin{enumerate}
      \item $\neg\ikkt_2(k_0,k_1)\implies \ikkt_3(k_0,k_1)$;\label{lem:linking_23_1}
      
      \item $\ikkt_2(k_0,k_1) \iff \neg\ikkt_3(k_0-1,k_1)$.\label{lem:linking_23_2}
      
      \item $\neg\ikkt_4(k_0,k_1)\implies \ikkt_5(k_0,k_1)$;\label{lem:linking_45_1}
      
      \item $\neg\ikkt_5(k_0,k_1) \iff \ikkt_4(k_0,k_1+1)$.\label{lem:linking_45_2}
    \end{enumerate}
  \end{claim}
      
  \begin{proof}
    \label{pf:linking_23_and_45}
    The proofs follow from direct computation.
    \begin{enumerate}
      \item Consider the contrapositive: $\neg\ikkt_3(k_0,k_1)\implies \ikkt_2(k_0,k_1)$.
      Since $\neg\ikkt_3(k_0,k_1)\iff(\theta+\lambda)(k_0,k_1)<x_{k_0+1}^0$, and since $x^0$ is sorted, because $x_{k_0+1}^0\leq x_{k_0}^0$, it follows that $(\theta+\lambda)(k_0,k_1)<x_{k_0}^0$.
      
      \item Since $\rho(k_0,k_1)>0$ for all valid $k_0$ and $k_1$, it holds that
      \begin{talign*}
        \neg\ikkt_3(k_0-1,k_1) &\iff (\theta+\lambda)(k_0-1,k_1) < x_{(k_0-1)+1}^0 = x_{k_0}^0\\
        &\iff k\sum_{i=(k_0-1)+1}^{k_1}x_i^0 + (k_1-k)\bigl(\sum_{i=1}^{k_0-1}x_i^0 - r\bigr) < (\rho + 2k-k_1)\cdot x_{k_0}^0\\
        &\iff \rho(k_0,k_1)\cdot(\theta+\lambda)(k_0,k_1) + (2k-k_1)\cdot x_{k_0}^0 < (\rho + 2k-k_1)\cdot x_{k_0}^0\\
        &\iff (\theta+\lambda)(k_0,k_1) < x_{k_0}^0 \iff \ikkt_2(k_0,k_1).
      \end{talign*}

      \item Consider the contrapositive: $\neg\ikkt_5(k_0,k_1)\implies \ikkt_4(k_0,k_1)$.
      Since $\neg\ikkt_5(k_0,k_1)\iff\theta(k_0,k_1)\leq x_{k_1+1}^0$, and because $x_{k_1+1}^0\leq x_{k_1}^0$, it follows that $\theta(k_0,k_1) \leq x_{k_1}^0$.
      
      \item Since $\rho(k_0,k_1)>0$ for all valid $k_0$ and $k_1$, it holds that
      \begin{talign*}
        \hspace*{-\leftmargin}
        \ikkt_4(k_0,k_1+1)&\iff x_{k_1+1}^0 - \theta(k_0,k_1+1)\geq0\\
        &\iff (\rho(k_0,k_1)+k_0)x_{k_1+1}^0 \geq k_0\sum_{i=k_0+1}^{k_1}x_i^0 + k_0 x_{k_1+1}^0 - (k-k_0)\bigl(\sum_{i=1}^{k_0} x_i^0 - r\bigr)\\
        &\iff x_{k_1+1}^0 \geq \theta(k_0,k_1)\iff \neg\ikkt_5(k_0,k_1).
      \end{talign*}
    \end{enumerate}
    Thus the claims are proved.
  \end{proof}

  \begin{claim}[Linking $\ikkt_2(k_0)$ \& $\ikkt_2(k_0+j)$ and $\ikkt_3(k_0)$ \& $\ikkt_3(k_0+j)$]
    \label{lem:linking_22j}
    Fix $k_1\in\{k,\ldots,n\}$.
    The following are true.
    \begin{enumerate}
      \item Let $k_0^-\in\{0,\ldots,k-1\}$ be such that $\neg\ikkt_2(k_0^-,k_1)$.
      Then $\neg\ikkt_2(k_0',k_1)$ for any $k_0^-\leq k_0'$.\label{lem:linking_22j_neg}
      
      \item Let $k_0^+\in\{0,\ldots,k-1\}$ be such that $\ikkt_2(k_0^+,k_1)$.
      Then $\ikkt_2(k_0',k_1)$ for any $k_0'\leq k_0^+$.\label{lem:linking_22j_pos}

      \item Let $k_0^-\in\{0,\ldots,k-1\}$ be such that $\neg\ikkt_3(k_0^-,k_1)$.
      Then $\neg\ikkt_3(k_0',k_1)$ for any $k_0'\leq k_0^-$.
      
      \item Let $k_0^+\in\{0,\ldots,k-1\}$ be such that $\ikkt_3(k_0^+,k_1)$.
      Then $\ikkt_3(k_0',k_1)$ for any $k_0^+\leq k_0'$.
    \end{enumerate}
    This means that for every $k_1$, the following sets are contiguous
    \begin{alignat*}{4}
      K_0^2(k_1)&\coloneqq \{k_0\in\{0,\ldots,k-1\}:\ikkt_2(k_0,k_1)\}\;\text{and}\;
      &K_0^{\neg2}(k_1)&\coloneqq\{k_0\in\{0,\ldots,k-1\}:\neg\ikkt_2(k_0,k_1)\},\\*
      K_0^3(k_1)&\coloneqq \{k_0\in\{0,\ldots,k-1\}:\ikkt_3(k_0,k_1)\}\text{ and }
      &K_0^{\neg3}(k_1)&\coloneqq\{k_0\in\{0,\ldots,k-1\}:\neg\ikkt_3(k_0,k_1)\}.
    \end{alignat*}
  \end{claim}
  \begin{proof}
    The proof is by direct computation.
    For fixed $k_1$, suppress the dependence on $k_1$ where clear, and define $\eta(k_0)\coloneqq \rho(k_0)\cdot(\theta+\lambda)(k_0)$.
    \begin{enumerate}
      \item Let $k_0=k_0^-$.
      By induction, it suffices to show the claim for $j=1$, \ie that
      \begin{align*}
        (*) : x_{k_0} \leq (\theta+\lambda)(k_0) \implies (**) : x_{k_0+1}^0 \leq (\theta+\lambda)(k_0+1).
      \end{align*}
      But $(\theta+\lambda)(k_0+1) = \bigl(\eta(k_0) + (k_1-2k)x_{k_0+1}^0\bigr) / (\rho(k_0)+k_1-2k)$ so
      \begin{align*}
        (**) &\iff (\rho + k_1-2k)\cdot x_{k_0+1}^0 \leq \eta(k_0) + (k_1-2k)x_{k_0+1}^0\\
        &\iff x_{k_0+1}^0\leq(\theta+\lambda)(k_0).
      \end{align*}
      Therefore, since $x^0$ is sorted, it is clear that $(*)\implies(**)$.

      \item Let $k_0=k_0^+$.
      By induction, it suffices to show the claim for $j=1$, \ie that
      \begin{align*}
        (*) : x_{k_0}^0 > (\theta+\lambda)(k_0) \implies (**) : x_{k_0-1}^0 > (\theta+\lambda)(k_0-1).
      \end{align*}
      But $(\theta+\lambda)(k_0-1) = \bigl(\eta(k_0) + (2k-k_1)x_{k_0+1}^0\bigr) / (\rho(k_0)+2k-k_1)$ so
      \begin{align*}
        (**) &\iff \bigl(\rho(k_0) + (2k-k_1)\bigr)\cdot x_{k_0-1}^0 > \eta(k_0) + (2k-k_1)x_{k_0}^0\\
        &\iff \rho(k_0) \cdot x_{k_0-1}^0 > \eta(k_0) + (2k-k_1)\cdot \bigl(x_{k_0}^0 - x_{k_0-1}^0\bigr).
      \end{align*}
      For contradiction, suppose that $(**)$ does not hold.
      Then dividing by $\rho(k_0)>0$,
      \begin{align*}
        \neg(**) &\iff x_{k_0-1}^0 \leq (\theta+\lambda)(k_0) \frac{2k-k_1}{\rho(k_0)}\cdot(x_{k_0}^0 - x_{k_0-1}^0)\\
        &\overset{(a)}{\implies} x_{k_0-1}^0 < x_{k_0-1}^0 \frac{2k-k_1}{\rho(k_0)}\cdot(x_{k_0}^0 - x_{k_0-1}^0)\\
        &\iff 0 < \bigl(\rho(k_0) + 2k-k_1\bigr)\cdot(x_{k_0}^0 - x_{k_0-1}^0)\implies 0<0
      \end{align*}
      where $(a)$ follows from $(*)$, and $x_{k_0}^0-x_{k_0-1}^0\leq0$ follows by since $x^0$ is sorted.

      \item Let $k_0=k_0^-$.
      By induction, it suffices to show the claim for $j=1$, \ie that
      \begin{align*}
        (*) : (\theta+\lambda)(k_0) < x_{k_0+1} \implies (**) : (\theta+\lambda)(k_0-1) < x_{k_0}^0.
      \end{align*}
      But $(\theta+\lambda)(k_0-1) = \bigl(\eta(k_0) + (2k-k_1)x_{k_0}^0\bigr) / (\rho(k_0)+2k-k_1)$ so
      \begin{align*}
        (**) &\iff \eta(k_0) + (2k-k_1)\cdot x_{k_0}^0 < (\rho + 2k-k_1)\cdot x_{k_0}^0\\
        &\iff (\theta+\lambda)(k_0)<x_{k_0}^0.
      \end{align*}
      Therefore, since $x^0$ is sorted, it is clear that $(*)\implies(**)$.
      
      \item Let $k_0=k_0^+$.
      By induction, it suffices to show the claim for $j=1$, \ie that
      \begin{align*}
        (*) : (\theta+\lambda)(k_0)\geq x_{k_0+1} \implies (**) : (\theta+\lambda)(k_0+1)\geq x_{k_0+2}^0.
      \end{align*}
      But $(\theta+\lambda)(k_0+1) = \bigl(\eta(k_0) + (k_1-2k)x_{k_0+1}^0\bigr) / (\rho(k_0)+k_1-2k)$ so
      \begin{align*}
        (**) &\iff \eta(k_0) + (k_1-2k)x_{k_0+1}^0 \geq (\rho + (k_1-2k))x_{k_0+2}^0\\
        &\overset{(a)}{\impliedby} \rho(k_0)\cdot x_{k_0+1}^0 \geq \rho(k_0)\cdot x_{k_0+2}^0 + (k_1-2k)\cdot(x_{k_0+2}^0-x_{k_0+1}^0)\\
        &\iff (\rho(k_0)+k_1-2k)\cdot(x_{k_0+1}^0-x_{k_0+2}^0)\geq0
      \end{align*}
      where $(a)$ follows from $(*)$, and $x_{k_0+1}^0-x_{k_0+2}^0\geq0$ follows since $x^0$ is sorted.
    \end{enumerate}
    Thus the claims are proved.
  \end{proof}

  \subsubsection*{Difference conditions}
  Next we summarize some conditions based on the successive differences of $\theta$, $\lambda$, and $\theta+\lambda$.
  To do so, it is useful to introduce the discrete difference operator $\Delta$ of a function $f(x,y)$ of two arguments, defined as $\Delta_x f(x,y) \coloneqq f(x+1,y) - f(x),\quad \Delta_y f(x,y) \coloneqq f(x,y+1) - f(x,y)$.
  We will be concerned with differences in both arguments of the index-pair $(k_0,k_1)$.
  \begin{claim}[$\Delta_{k_0}$ \& $\neg\ikkt_2$ and $\Delta_{k_1}$ \& $\neg\ikkt_5$]
    \label{lem:Delta_k0k1}
    For fixed $k_1\in\{k,\ldots,n\}$, it holds that
    \begin{enumerate}
      \item $\Delta_{k_0}\theta(k_0-1,k_1) \geq 0 \iff \neg\ikkt_2(k_0,k_1)$;\label{lem:Delta_k0_24}%
      \item $\Delta_{k_0}\lambda(k_0-1,k_1)\leq0 \iff \neg\ikkt_2(k_0,k_1)$;\label{lem:Delta_k0_21}%
      \item $\Delta_{k_1}\theta(k_0,k_1)\geq0 \iff \neg\ikkt_5(k_0,k_1)$;\label{lem:Delta_k1_25}%
      \item $\Delta_{k_1}\lambda(k_0,k_1)\geq0 \iff \neg\ikkt_5(k_0,k_1)$;\label{lem:Delta_k1_15}%
      \item $\Delta_{k_1}(\theta+\lambda)(k_0,k_1) \geq 0 \iff \neg\ikkt_5(k_0,k_1)$.\label{lem:Delta_k1_35}%
    \end{enumerate}
  \end{claim}
  \begin{proof}
    For shorthand, let $\tau_0\coloneqq\sum_{i=1}^{k_0}x_i^0-r$, $\tau_1\coloneqq \sum_{i=1}^{k_1}x_i^0-r$, $\eta_\theta(k_0,k_1)\coloneqq \rho(k_0,k_1)\cdot\theta(k_0,k_1)$, $\eta_\lambda(k_0,k_1)\coloneqq \rho(k_0,k_1)\cdot\lambda(k_0,k_1)$, and $\eta_{(\theta+\lambda)}(k_0,k_1)\coloneqq \rho(k_0,k_1)\cdot(\theta+\lambda)(k_0,k_1)$.
    \begin{enumerate}
      \item 
      Compute $\Delta_{k_0}\theta(k_0-1,k_1)\coloneqq\theta(k_0,k_1) - \theta(k_0-1,k_1) = \frac{(2k-k_1)\cdot \eta_\theta(k_0,k_1) - \rho(k_0,k_1)\cdot (k x_{k_0}^0 - \tau_1)}{\rho(k_0,k_1)\cdot \bigl(\rho(k_0,k_1) + 2k - k_1\bigr)}$.
      Then
      \begin{talign*}
        \Delta_{k_0}\theta(k_0-1,k_1) \geq0 &\overset{(a)}{\iff} (2k-k_1)\cdot \eta_\theta(k_0,k_1) - \rho(k_0,k_1)\cdot (k x_{k_0}^0 - \tau_1) \geq0\\
        &\iff \frac{(2k-k_1)\cdot\eta_\theta(k_0,k_1) + \rho(k_0,k_1)\cdot\tau_1}{\rho(k_0,k_1)\cdot k} \geq x_{k_0}^0
        \overset{(b)}{\iff} \frac{k\cdot \eta_\theta(k_0,k_1)}{k_0\cdot \rho(k_0,k_1)} - \frac{\eta_\theta(k_0,k_1)}{k\cdot k_0} + \frac{\tau_1}{k} \geq x_{k_0}^0\\
        &\iff \frac{k\sum_{i=k_0+1}^{k_1} x_i^0 - \frac{k\cdot(k-k_0)}{k_0}\tau_0}{\rho(k_0,k_1)} - \Bigl(\frac1k \sum_{i=k_0+1}^{k_1} x_i^0 - \frac{k-k_0}{k\cdot k_0}\tau_0\Bigr) + \frac{\tau_1}{k} \geq x_{k_0}^0\\
        &\overset{(c)}{\iff} \frac{k\sum_{i=k_0+1}^{k_1} x_i^0 - \frac{k\cdot(k-k_0)}{k_0}\tau_0}{\rho(k_0,k_1)} + \frac{\tau_0}{k} + \frac{k-k_0}{k\cdot k_0}\tau_0 \geq x_{k_0}^0\\
        &\iff \frac{k\sum_{i=k_0+1}^{k_1}x_i^0}{\rho(k_0,k_1)} - \Bigl(\frac{k\cdot(k-k_0)}{k_0\cdot\rho(k_0,k_1)} + \frac1k + \frac{k-k_0}{k\cdot k_0}\Bigr)\cdot \tau_0 \geq x_{k_0}^0\\
        &\overset{(d)}{\iff} \frac{k\sum_{i=k_0+1}^{k_1}x_i^0}{\rho(k_0,k_1)} - \frac{k_1-k}{\rho(k_0,k_1)}\cdot\tau_0 \geq x_{k_0}^0
        \iff (\theta+\lambda)(k_0,k_1) \geq x_{k_0}^0\\
        &\iff \neg\ikkt_2(k_0,k_1)
      \end{talign*}
      where $(a)$ follows since $\rho(k_0,k_1)>0$ for all valid arguments, $(b)$ follows from the partial fractions identity $(2k-k_1)/\rho(k_0,k_1) = \frac{k}{k_0\cdot \rho(k_0,k_1)} - \frac{1}{k\cdot k_0}$, $(c)$ follows from $\tfrac1k \tau_1 - \tfrac1k \sum_{i=k_0+1}^{k_1}x_i^0 = \tfrac1k \tau_0$, and $(d)$ follows from algebraic manipulation.

      \item
      Compute
      \begin{talign*}
        \hspace*{-\leftmargin}
        \lambda(k_0-1,k_1) &= \Bigl( (k-(k_0-1))\sum_{i=(k_0-1)+1}^{k_1}x_i^0 + (k_1 - (k_0-1))\Bigl(\sum_{i=1}^{k_0-1}x_i^0 - r\Bigr)\Bigr) / \rho(k_0-1,k_1)\\
        &= \bigl(\eta_\lambda(k_0,k_1) + (k-k_1)\cdot x_{k_0}^0 + \tau_1\bigr) / \bigl(\rho(k_0,k_1) + 2k - k_1\bigr)
      \end{talign*}
      and $\Delta_{k_0}\lambda(k_0-1,k_1) = \lambda(k_0,k_1) - \lambda(k_0-1,k_1)
      = \frac{(2k-k_1)\cdot \eta_\lambda(k_0,k_1) - \rho(k_0,k_1)\cdot\bigl((k-k_1)\cdot x_{k_0}^0 + \tau_1\bigr)}{\rho(k_0,k_1)\cdot\bigl(\rho(k_0,k_1) + 2k - k_1\bigr)}$.
      Then
      \begin{talign*}
        \Delta_{k_0}\lambda(k_0-1,k_1)\leq0 &\overset{(a)}{\iff} (2k-k_1)\cdot \eta_\lambda(k_0,k_1) - \rho(k_0,k_1)\cdot\bigl((k-k_1)\cdot x_{k_0}^0 + \tau_1\bigr)\leq0\\
        &\iff \frac{(2k-k_1)\cdot\eta_\lambda(k_0,k_1)}{\rho(k_0,k_1)} - \tau_1 \leq (k-k_1)\cdot x_{k_0}^0\\
        &\overset{(b)}{\iff} \frac{(2k-k_1)\cdot\eta_\lambda(k_0,k_1)}{\rho(k_0,k_1)\cdot(k-k_1)} - \frac{\rho(k_0,k_1)\cdot \tau_1}{\rho(k_0,k_1)\cdot (k-k_1)} \geq x_{k_0}^0
        \iff c_1 \sum_{i=k_0+1}^{k_1}x_i^0 + c_2 \cdot \tau_0 \geq x_{k_0}^0
      \end{talign*}
      where $(a)$ holds because $\rho(k_0,k_1)>0$ for all valid arguments, $(b)$ holds since $(k-k_1)<0$ for $k_1>k$ (and at $k_1=k$, it holds since $\rho(k_0,k)>k\cdot(k-k_0)$), and
      $
      c_1\coloneqq \frac{(2k-k_1)\cdot(k-k_0) - \rho(k_0,k_1)}{\rho(k_0,k_1)\cdot(k-k_1)}
      $
      and
      $
      c_2\coloneqq \frac{(2k-k_1)\cdot(k_1-k_0) - \rho(k_0,k_1)}{\rho(k_0,k_1)\cdot(k-k_1)}.
      $
      Then after simplification, $c_1 = k/\rho(k_0,k_1)$ and $c_2 = (k_1-k) / \rho(k_0,k_1)$
      so we identify $\Delta_{k_0}\lambda(k_0-1,k_1)\leq0 \iff (\theta+\lambda)(k_0,k_1)\leq x_{k_0}^0 \iff \neg\ikkt_2(k_0,k_1)$.

      \item %
      Compute $\Delta_{k_1}\theta(k_0,k_1) \coloneqq \theta(k_0,k_1+1) - \theta(k_0,k_1) = \frac{\rho(k_0,k_1)\cdot k_0\cdot x_{k_1+1}^0 - k_0\eta_\theta(k_0,k_1)}{\rho(k_0,k_1) \cdot (\rho(k_0,k_1) + k_0)}$.
      Then
      \begin{align*}
        \Delta_{k_1}\theta(k_0,k_1)\geq0 &\overset{(a)}{\iff} \rho(k_0,k_1)\cdot k_0\cdot x_{k_1+1}^0 \geq k_0 \eta_\theta(k_0,k_1)
        \iff x_{k_1+1}^0 \geq \theta(k_0,k_1) \iff \neg\ikkt_5(k_0,k_1).
      \end{align*}

      \item %
      Compute $\Delta_{k_1}\lambda(k_0,k_1) \coloneqq \lambda(k_0,k_1+1) - \lambda(k_0,k_1)
      = \frac{\rho(k_0,k_1) \cdot \bigl((k-k_0)x_{k_1+1}^0 + \tau_0\bigr) - k_0 \eta_\lambda(k_0,k_1)}{\rho(k_0,k_1)\cdot \bigl(\rho(k_0,k_1)+k_0\bigr)}$.
      Then
      \begin{align*}
        \Delta_{k_1}\lambda(k_0,k_1)\geq0 &\overset{(a)}{\iff} \rho(k_0,k_1) \cdot \bigl((k-k_0)x_{k_1+1}^0 + \tau_0\bigr) - k_0 \eta_\lambda(k_0,k_1)\geq0\\
        &\iff x_{k_1+1}^0 \geq \frac{k_0}{k-k_0}\frac{\eta_\lambda(k_0,k_1)}{\rho(k_0,k_1)} - \frac{\tau_0}{k-k_0}
        \overset{(b)}{\iff} x_{k_1+1}^0 \geq \theta(k_0,k_1)\iff\neg\ikkt_5(k_0,k_1)
      \end{align*}
      where $(a)$ holds because $\rho(k_0,k_1)>0$ for all valid arguments and $(b)$ holds because
      $\frac{k_0\cdot(k_1-k_0)}{\rho(k_0,k_1)\cdot(k-k_0)} - \frac{1}{k-k_0} = \frac{k_0-k}{\rho(k_0,k_1)}$
      by algebraic manipulation.

      \item %
      Consider the equivalent form of the original statement given by $\Delta_{k_1}(\theta+\lambda)(k_0,k_1) < 0 \iff \ikkt_5(k_0,k_1)$.
      Compute $\Delta_{k_1}(\theta+\lambda)(k_0,k_1) \coloneqq (\theta+\lambda)(k_0,k_1+1)-(\theta+\lambda)(k_0,k_1)$ giving
      $$\Delta_{k_1}(\theta+\lambda)(k_0,k_1) = \frac{\rho(k_0,k_1)\cdot\bigl(k\cdot x_{k_1+1}^0 + \tau_0\bigr) - k_0 \bigl(k\sum_{i=k_0+1}^{k_1} x_i^0 + (k_1-k)\tau_0\bigr)}{\rho(k_0,k_1)\cdot\bigl(\rho(k_0,k_1)+k_0\bigr)}.$$
      Then
      \begin{talign*}
        \Delta_{k_1}(\theta+\lambda)(k_0,k_1)<0 &\overset{(a)}{\iff} {\rho(k_0,k_1)\cdot\bigl(k\cdot x_{k_1+1}^0 + \tau_0\bigr) - k_0 \Bigl(k\sum_{i=k_0+1}^{k_1} x_i^0 + (k_1-k)\tau_0\Bigr)} < 0\\
        &\iff \rho(k_0,k_1) \cdot k x_{k_1+1}^0 < k_0 \Bigl(k\sum_{i=k_0+1}^{k_1}x_i^0 + (k_1-k)\cdot\tau_0\Bigr) - \rho(k_0,k_1)\cdot\tau_0\\
        &\iff \rho(k_0,k_1) \cdot k x_{k_1+1}^0 < k_0k \sum_{i=k_0+1}^{k_1}x_i^0 + \bigl(k_0(k_1-k)-\rho(k_0,k_1)\bigr)\cdot\tau_0\\
        &\overset{(b)}{\iff} \rho(k_0,k_1) \cdot k x_{k_1+1}^0 < k_0k \sum_{i=k_0+1}^{k_1}x_i^0 + k\cdot(k_0-k)\cdot\tau_0\\
        &\iff \rho(k_0,k_1) \cdot x_{k_1+1}^0 < k_0 \sum_{i=k_0+1}^{k_1}x_i^0 + \cdot(k_0-k)\cdot\tau_0\\
        &\iff x_{k_1+1}^0 < \frac{k_0 \sum_{i=k_0+1}^{k_1}x_i^0 + \cdot(k_0-k)\cdot\tau_0}{\rho(k_0,k_1)} = \theta(k_0,k_1)\\
        &\iff\kkt_5(k_0,k_1)>0.
      \end{talign*}
      
      where $(a)$ follows because $\rho>0$ for all valid arguments and $(b)$ follows from the identity $k_0(k_1-k)-\rho(k_0,k_1)=k(k_0-k)$.
    \end{enumerate}
    Thus the claims are proved.
  \end{proof}
\end{toappendix}

\subsection{Partial sorting}
\label{sec:partial_sorting}
Until now, the sorted problem \eqref{eq:maxksum_projection_sort} has assumed that a full sorting permutation is given and has been applied to the input data $x^0 \in \R^n$.
However, it is of significant practical interest to note that since the elements of $\bar\gamma$ are unperturbed, it is only necessary that the initial permutation that sort the largest $L\geq\bar{k}_1 \equiv n - \abs{\bar\gamma}$ indices of $x^0\in\R^n$.
Let us call a permutation with $L\geq \bar{k}_1$ an ``optimal'' permutation, which can be obtained offline in $O(L \log n)$ by \texttt{heapsort}.
When $k\ll n$, then $\abs{\gamma}$ may be close to $n$, and the initial cost of obtaining an optimal permutation may be greatly reduced.
Since $\bar{k}_1$ is not known at runtime, though, determining whether a candidate permutation is optimal \emph{a priori} is not possible.
However, in the following proposition, we provide an implementable condition for checking whether the result of a projection based on a candidate permutation is indeed optimal.

On the other hand, we may instead seek to construct an optimal permutation in an ``as-needed'' fashion by embedding the sorting within the solution procedure.\footnote{The implementation of this approach was inspired by a question posed by Eric Sager Luxenberg on an earlier version of this paper: \url{https://arxiv.org/abs/2310.07224v1}.}
That is, since \cref{alg:projection_maxksum_lcp,,alg:projection_maxksum_snake} only inspect elements of $x^0$ in a contiguous and increasing subset of $\{1,\ldots,n\}$, the sorting may be performed online.
This leads to an $O(\bar{k}_1 \log n)$ overall procedure, despite the fact that $\bar{k}_1$ is not known at runtime, as summarized in the following proposition.

\begin{proposition}
  \label{prop:partialsort}
  The following two claims hold for any given $x^0\in\R^n$:
  \begin{enumerate}
      \item Let $\hat{\pi}$ be a given permutation of $x^0$ (not necessarily in the nonincreasing order).
      Define $\hat{x}\coloneqq\proj_{\maxsumball{k}{r}}{x^0_{\hat{\pi}}}$ and $\tilde{x}\coloneqq\proj_{\maxsumball{k}{r}}{x^0}$ as the solutions of the candidate and unsorted problems, respectively, and let $\hat{k}_1$ be the index corresponding to $\hat{x}$ in \eqref{eq:order_structure}.
      If (i) $\hat{x}_{\hat{k}_1} > (x^0_{\hat\pi})_{i}$ for all $i\in\{\hat{k}_1+1,\ldots,n\}$, and (ii) the elements in  $\hat{x}_{1:\hat{k}_1}$ are sorted in the nonincreasing order, then one can obtain $\tilde{x}$ via $(\tilde{x}_{\hat\pi})_{1:\hat{k}_1} = \hat{x}_{1:\hat{k}_1}$ and $(\tilde{x}_{\hat\pi})_{\hat{k}_1+1:n} = (x^0_{\hat{\pi}})_{\hat{k}_1+1:n}$, \ie $\hat{\pi}$ is an optimal permutation.

      \item \blue{
      A solution $\tilde{x}$ to the unsorted problem can be obtained in $O(\bar{k}_1 \log n)$ operations, where $\bar{k}_1=n-\abs{\bar\gamma}$ is the second component of the unique optimal index-pair satisfying KKT conditions \eqref{eq:kkt_surplus}, which is unknown-at-runtime but can be identified dynamically.
      }
  \end{enumerate}
\end{proposition}
\begin{proof}
  Let $\bar\pi$ be a full (and hence optimal) sorting permutation of $x^0$ so that $x^0_{\bar\pi} = \sorth{x}^0$.
  The KKT conditions require a solution to satisfy $\bar{x}_{\bar{k}_1} \equiv \bar\theta > (x^0_{\bar\pi})_{\bar{k}_1+1}$.
  By the ordering on $x_{\bar\pi}^0$, it holds that $(x_{\bar\pi}^0)_i\leq (x^0_\pi)_{\bar{k}_1+1}$ for all $i\geq \bar{k}_1+1$.
  If $\hat{x}$ satisfies all the KKT conditions associated with the $\hat\pi$-permuted problem and $(\hat{x}_{\hat\pi})_{\hat{k}_1} > (x^0_{\hat\pi})_{i}$ for all $i\in \{\hat{k}_1+1,\ldots,n\}$, then it satisfies the KKT conditions of the fully sorted problem.

  \blue{
  To justify the $O(\bar{k}_1\log(n))$ complexity, we argue that the \texttt{heapsort} algorithm can be embedded within the iterative approaches of \cref{alg:projection_maxksum_snake,,alg:projection_maxksum_lcp}.
  We provide an argument for \cref{alg:projection_maxksum_lcp} and note that \cref{alg:projection_maxksum_snake} can be handled similarly.
  Explicitly, construct a (binary) max-heap based on the input vector $x^0$ in $O(n)$ cost.
  The largest element of any heap can be extracted in $O(\log n)$ time.
  Therefore, sorted elements $\sorth x^0_1,\ldots,\sorth x^0_\ell$ can be obtained in $O(\ell\log n)$ time by by extracting the maximum element of the (successively modified) heap $\ell$ times.
  Next, given a suboptimal candidate index-pair $(k_0,k_1)$, the next pivot requires inspecting either $\sorth x^0_{k_0-1}$ or $\sorth x^0_{k_1+1}$; the former can be stored from the initial extraction process, and the latter can be obtained in $O(\log n)$ time from the modified heap.
  Since the optimal index-pair $(\bar{k}_0,\bar{k}_1)$ is unique, the procedure performs exactly $\bar{k}_1$ extractions for a cost of $O(\bar{k}_1\log n)$.}
\end{proof}
\cref{prop:partialsort} is useful in applications in which the solution to the projection problem is not expected to change significantly from iteration to iteration.
An important example of this is in solving a sequence of related projection problems (such as in linesearch or projected gradient descent) when the input vector does not change significantly.
\blue{For the first problem in the sequence, the online sorting property may be used to identify an initial permutation $\bar\pi^{(1)}\equiv\hat\pi^{(1)}$ in $O(\bar{k}_1\log n)$ via claim 2.
For subsequent problems (indexed by $\nu$), the previous problem's sorting permutation $\hat{\pi}^{(\nu)}$ may be used to warm-start the construction of an approximate sorting permutation $\hat{\pi}^{(\nu+1)}$}.

\subsection{Relation to the vector-\texorpdfstring{$k$}{k}-norm ball}
We now briefly state how the previous two approaches can be utilized when solving the related but distinct (and slightly more complicated) problem of projection onto the (Ky-Fan) vector-$k$-norm ball studied in \cite{wu2014moreau}.
The vector-$k$-norm ball of radius $r\geq0$ is defined by $\mathcal{V}_{(k)}^r \coloneqq \{z\in\R^n : \sum_{i=1}^k \sorth{\abs{z}}_{i} \leq r\}$, and the sorted vector-$k$-norm problem only differs from the sorted top-$k$-sum problem \eqref{eq:maxksum_projection_sort} by the additional constraint $z_{n}\geq0$.
That is, the sorted formulation
  \begin{align}
  \label{eq:vecknorm_projection_sort}
    \bar z &\coloneqq\argmin_{z\in\R^n}\bigl\{
      \tfrac12 \norm{z-\sorth{z}^0}_2^2 :
      \ind_k^\top z\leq r,\;
      z_i\geq z_{i+1},\;\forall i \in \{1,\ldots,n-1\},\;
      z_n\geq0
    \bigr\}
  \end{align}
has polyhedral region $\{z\in\R^n:Vz \leq v\}$ with data $V = \begin{bmatrix} (\ind_k^\top,\; 0_{n-k}^\top)\\ -E\end{bmatrix},\; E\coloneqq \begin{bmatrix} D\\ (0_{n-1}^\top,\;1)\end{bmatrix},\; v\coloneqq(r,0_n^\top)^\top$, where $D$ is the isotonic difference operator defined in \eqref{eq:maxksum_projection_sort_polyhedron}.
The penalized problem with PLCP data $M\coloneqq EE^\top$, $q\coloneqq E\,\sorth{z}^0$ and direction vector $d\coloneqq -E\,\ind_k=-e^k$ shares nearly identical structure with \eqref{eq:plcp}, so pivots can be performed in a similar manner to the approach outlined in \cref{alg:projection_maxksum_lcp}.

On the other hand, \cite{wu2014moreau} 
provide a two-step routine (Algorithm 4) for solving the vector-$k$-norm projection problem based on the observation that the $k^{\text{th}}$ largest value of the solution must satisfy (i) $\bar{z}_{[k]}=0$; or (ii) $\bar{z}_{[k]}>0$.
The first step identifies a solution satisfying condition (i), if one exists, in $O(k)$ complexity; otherwise if it does not exist, the second step identifies a solution satisfying condition (ii) by performing a grid search over all index-pairs $(k_0,k_1)$.
The second step is the algorithm that we refer to as the ``KKT grid-search'' method in \cref{sec:background}.
The KKT conditions of the second case coincide with the KKT conditions of the top-$k$-sum problem, so 
\cref{alg:projection_maxksum_snake} can be substituted for the second step in \cite[Algorithm 4]{wu2014moreau}, yielding a procedure with overall complexity of $O(n)$  on sorted input vector $\sorth{z}^0$.
Each step needs to  run sequentially, though, so the ESGS-based approach incurs an additional $O(k)$ cost in instances that are not solved in the first step.

\section{Numerical Experiments}
\label{sec:experiments}
To evaluate the performance of our proposed algorithms, we conduct a series of numerical experiments on synthetic datasets.
We implement the algorithms in \verb|Julia| \cite{julia} and execute the tests on a 125GB RAM machine with Intel(R) Xeon(R) W-2145 CPU @ 3.70GHz processors running \verb|Julia v1.9.1|.

The experimental problems were formulated based on the following protocol: the index $k$ and right-hand side $r$ are set as $k = \tau_k^\complement\cdot n$ and $r = \tau_r\cdot \maxsum_k{x^0}$, where $\tau_k^\complement = 1-\tau_k$ and $\tau_r$ take on values from the sets:
\begin{itemize}[itemsep=0pt,topsep=0pt]
  \item $\tau_r\in\{-8,\,-4,\,-2,\,-1,\,-1/2,\,-1/10,\,0,\, 1/10,\, 1/2,\, 9/10,\, 99/100,\, 999/1000\}$;
  \item $\tau_k^\complement\in\{1/10000,\, 1/1000,\, 1/100,\, 5/100,\, 1/10,\, 1/2,\, 9/10,\, 99/100,\, 999/1000,\, 9999/10000\}$.
\end{itemize}
For context,  in many practical scenarios (\eg the risk-averse  superquantile constrained problems), the values of $\tau_k^\complement$ typically \reftwo{fall} between 1\% and 10\%.
To facilitate more intuitive interpretation of our computational findings, initial vectors are generated uniformly from $[0,1]^n$ in double precision.
As $\tau_r$ approaches $1$, the projection problems tend to become easier because $\maxsum_k{x^0}\approx r$.
Conversely, as $\tau_r$ trends towards $-\infty$, the projection problems  become more challenging as the solution will have a substantial deviation from the original point.
The problem dimension $n$ is set from $n\in \{10^1,10^2,\ldots,10^7\}$, and $100$ instances are generated for each scenario unless stated otherwise.

The ESGS, PLCP, GRID (our implementation of the grid-search method from \cite{wu2014moreau}), \refone{and SSN (our implementation of the semismooth Newton method from \cite{li2021fast})} are written in \texttt{Julia} and use double precision for the experiments, though they can also handle arbitrary precision floats and rational data types.
\blue{The partition-based method \cite{davis2015algorithm} was not studied since it was found to be significantly slower than the SSN method in \cite{li2021fast}.}
The  code is available at \url{https://github.com/jacob-roth/top-k-sum}.
The finite-termination methods use a single core but can make use of \texttt{simd} operations.
The QP solver utilizes the barrier method provided by Gurobi \texttt{v10.0} to solve both the sorted \eqref{eq:maxksum_projection_sort} and unsorted \eqref{eq:maxksum_projection_unsorted} formulations, called GRBS and GRBU, respectively.
The feasibility and optimality tolerances are set to $10^{-9}$, the presolve option  to the default, and the method is configured to use up to 8 cores.
Model initialization time is not counted towards Gurobi's solve time.
Both Gurobi methods and the GRID method have time-limits of 10,000 seconds for each instance.

\begin{figure}[t]
  \centering
  \begin{tabular}{@{}c@{}p{1ex}@{}c@{}}
    \begin{tabular}{@{}c@{}}
      \resizebox{0.45\linewidth}{!}{\includegraphics{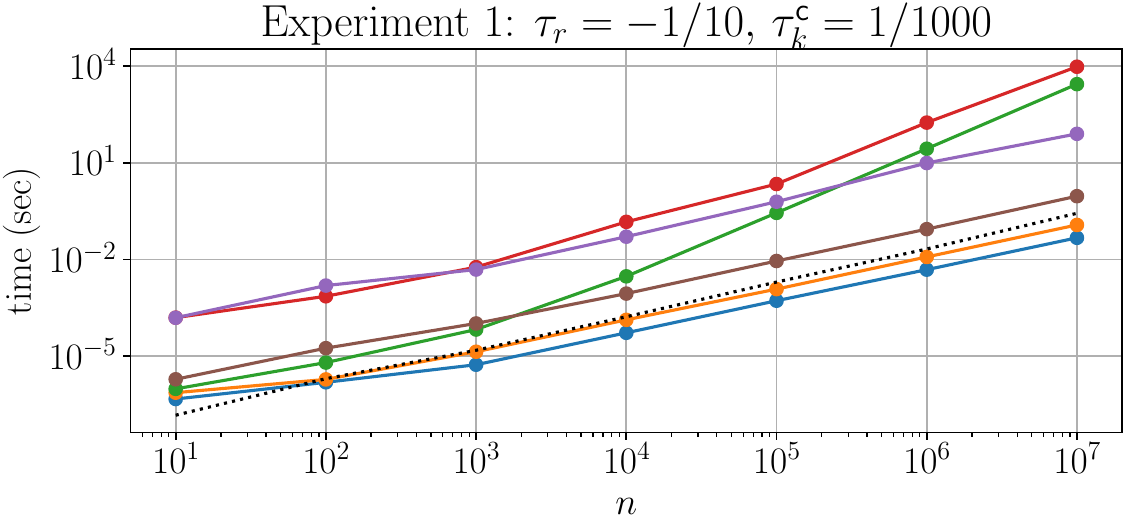}}
      \resizebox{0.45\linewidth}{!}{\includegraphics{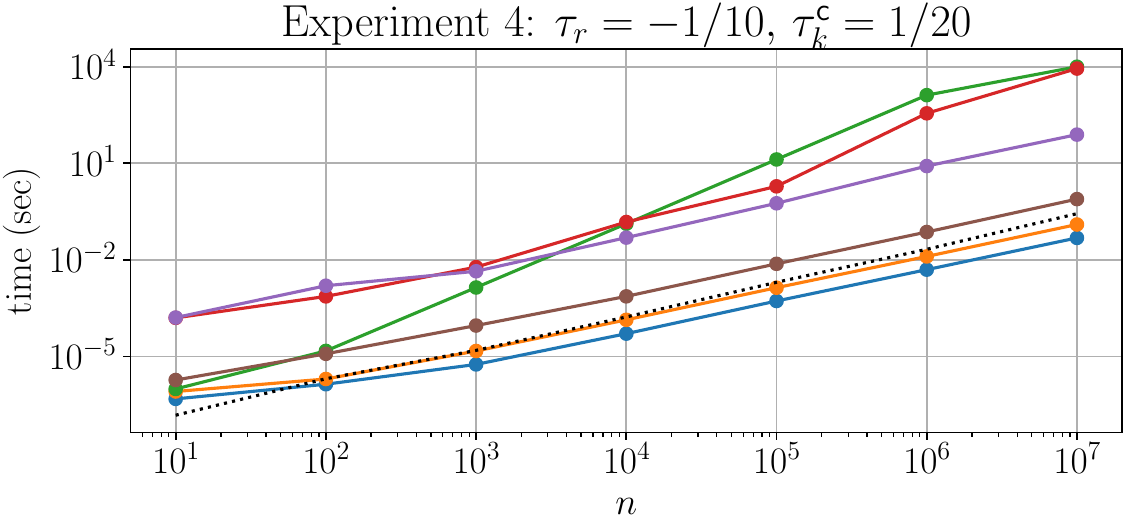}}\\
      \resizebox{0.45\linewidth}{!}{\includegraphics{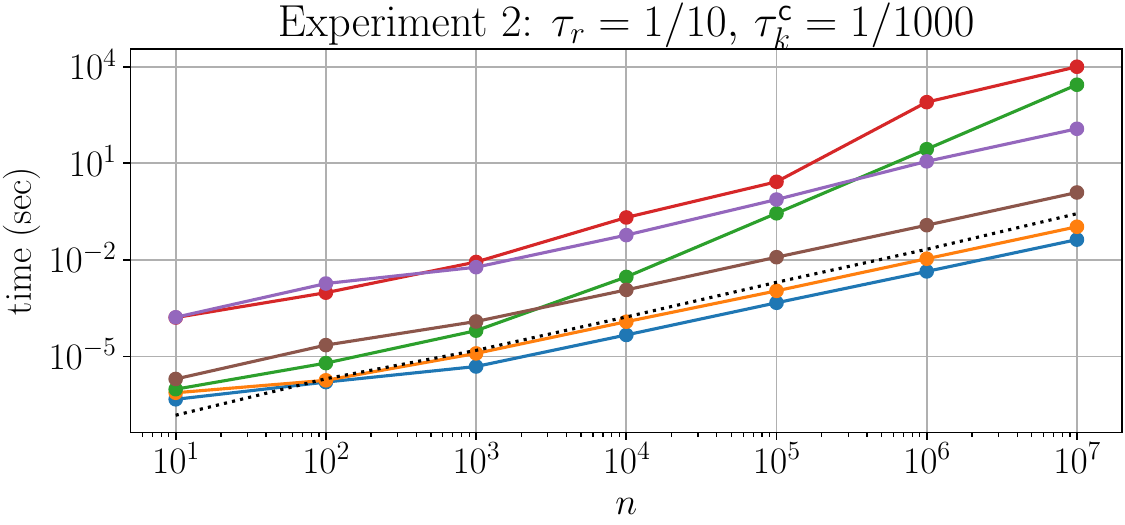}}
      \resizebox{0.45\linewidth}{!}{\includegraphics{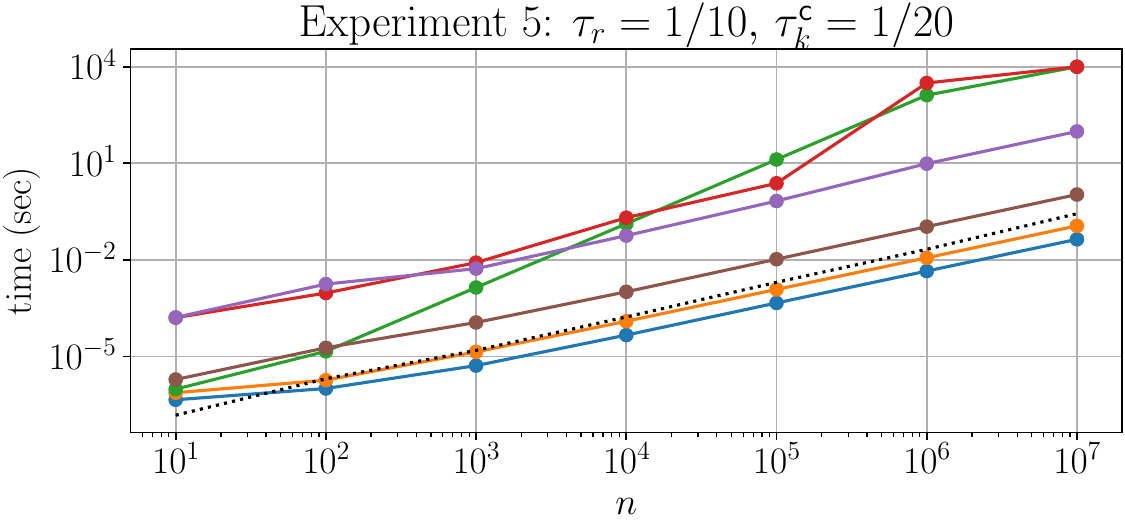}}\\
      \resizebox{0.45\linewidth}{!}{\includegraphics{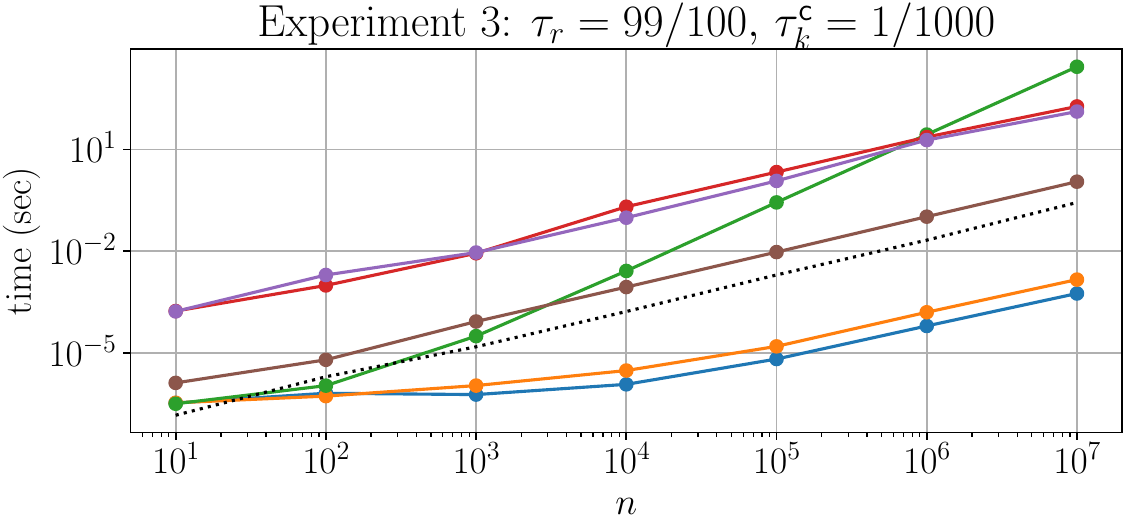}}
      \resizebox{0.45\linewidth}{!}{\includegraphics{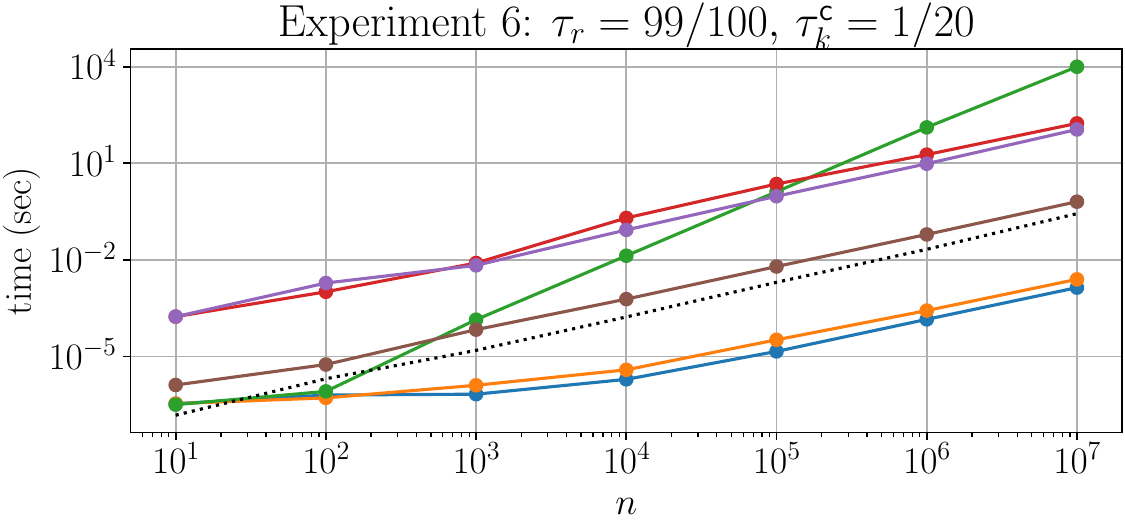}}\\
      \centering
      \resizebox{0.75\linewidth}{!}{\includegraphics[trim={0 0 0 7.5cm},clip]{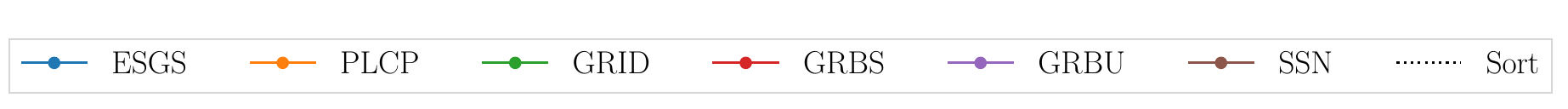}}
    \end{tabular}
  \end{tabular}
  \caption{\small
    Average total computation time excluding sort time and full sort time vs $n$.
    All results are averaged over 100 instances, except for methods GRID, GRBS, and GRBU with $n\in\{10^6,10^7\}$ in which a time-limit of $10^4$ seconds is imposed across 2 instances.
  }
  \label{fig:time_vs_n}
\end{figure}%

\begin{figure}
  \hspace*{-5mm}
  \begin{subfigure}[t]{0.32\linewidth}
    \resizebox{1.1\linewidth}{!}{\includegraphics[clip, trim=4mm 3mm 2mm 3mm]{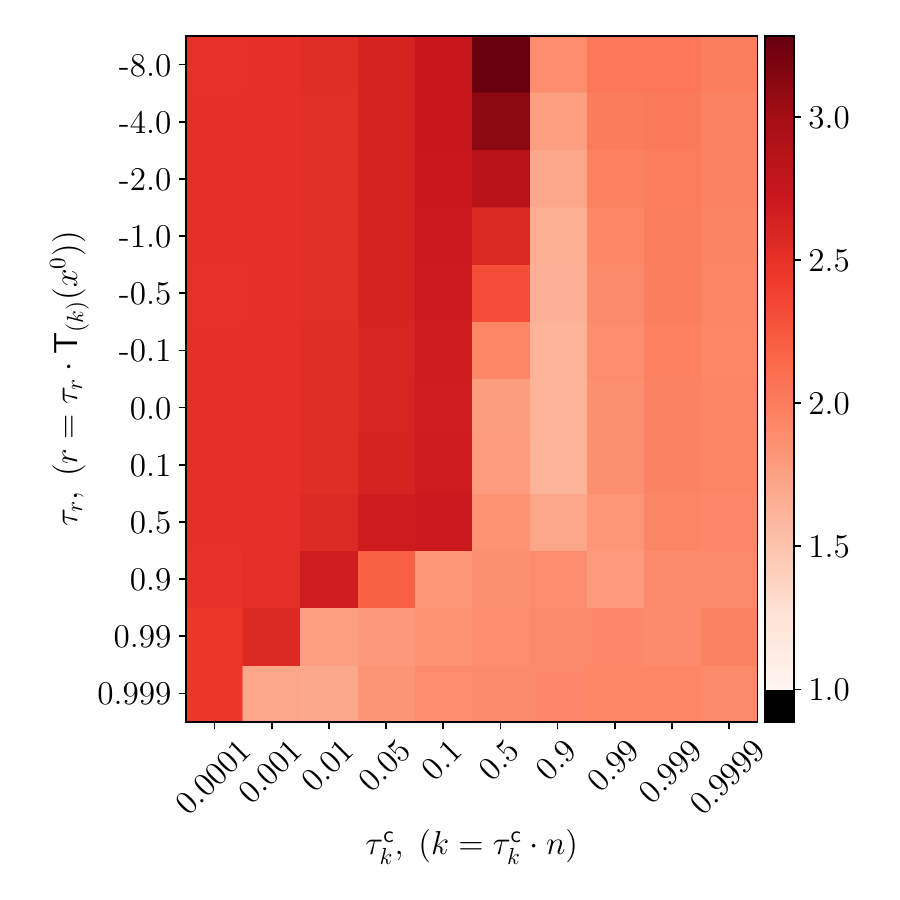}}
    \captionsetup{justification=raggedleft,singlelinecheck=false}
    \caption{\small PLCP vs ESGS at $n=10^7$.\!\!}
    \label{fig:relative_time:21}
  \end{subfigure}
  \hspace*{2mm}
  \begin{subfigure}[t]{0.32\linewidth}
    \resizebox{1.1\linewidth}{!}{\includegraphics[clip, trim=4mm 3mm 2mm 3mm]{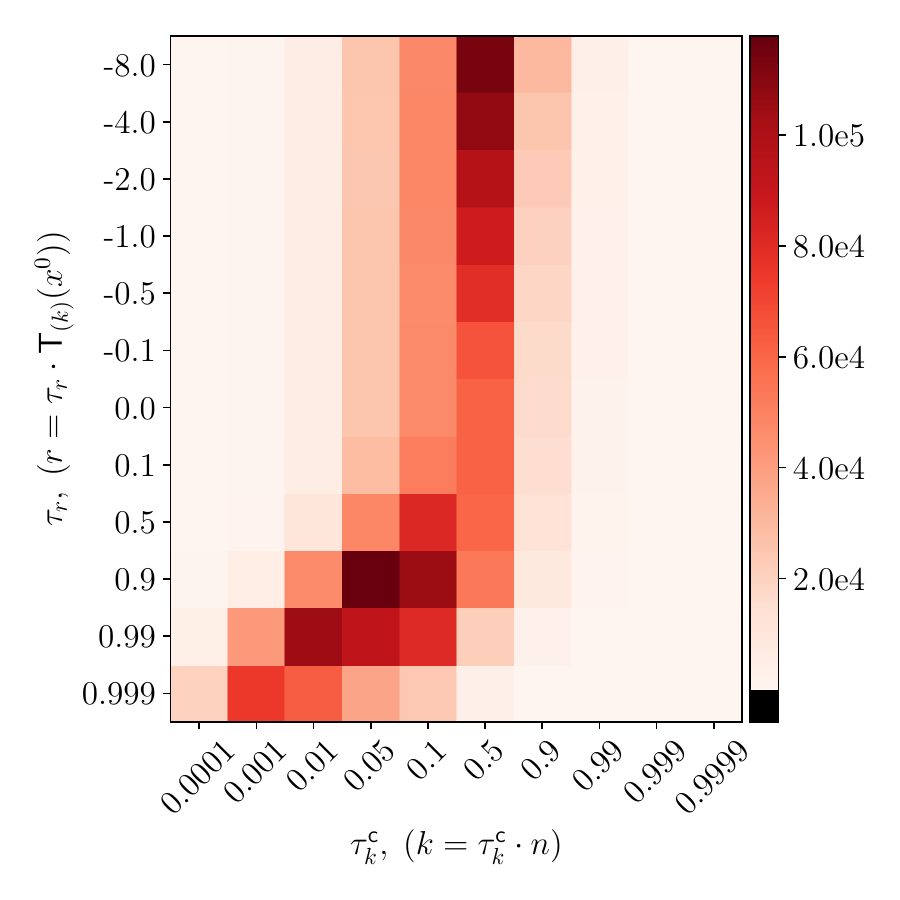}}
    \captionsetup{justification=raggedleft,singlelinecheck=false}
    \caption{\small GRID vs ESGS at $n=10^5$.}
    \label{fig:relative_time:31}
  \end{subfigure}
  \hspace*{2mm}
  \begin{subfigure}[t]{0.32\linewidth}
    \resizebox{1.1\linewidth}{!}{\includegraphics[clip, trim=4mm 3mm 2mm 3mm]{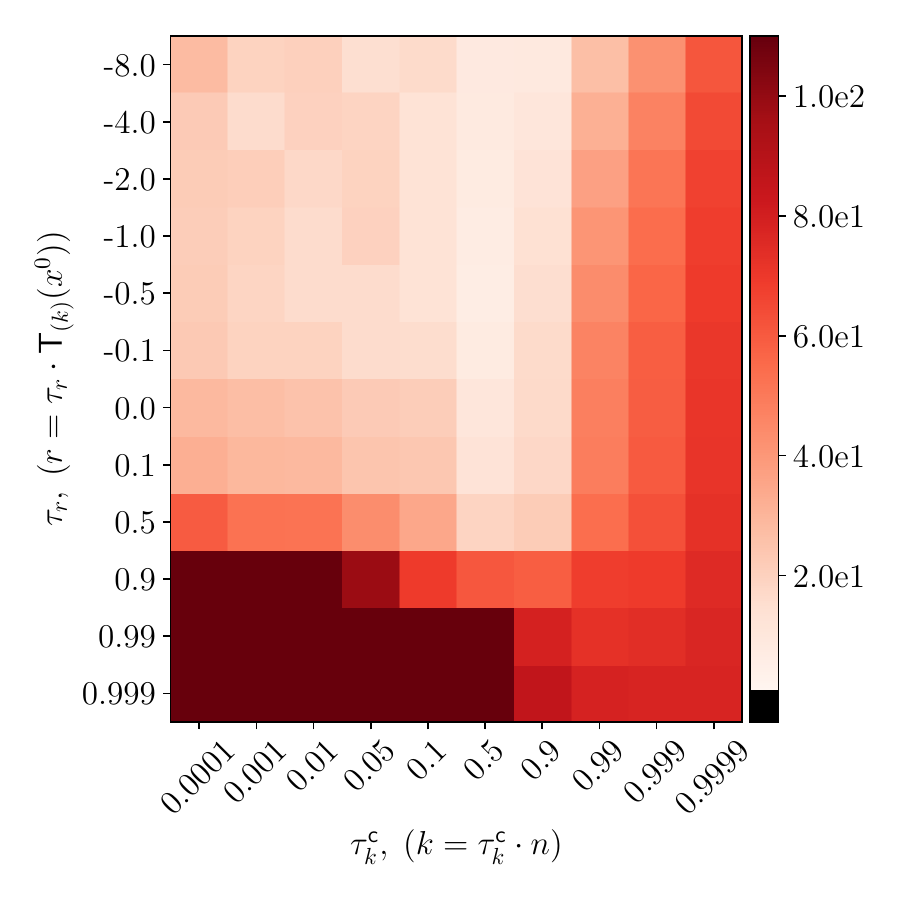}}
    \captionsetup{justification=raggedleft,singlelinecheck=false}
    \caption{\small SSN vs ESGS at $n=10^7$.\;\;\;}
    \label{fig:relative_time:61}
  \end{subfigure}
  \caption{
    \small 
    Computation time relative to ESGS averaged over 100 instances.
    A value of $c>0$ indicates that ESGS was $c$ times faster than the other method; a value of $c<0$ indicates that the other method was $c$ times faster than ESGS.
    Across all scenarios at $n=10^5$, the better of GRBS and GRBU was at best $\approx350$ times slower than ESGS (and never faster).
  }
  \label{fig:relative_time} 
\end{figure}

\begin{table}
  \vspace*{-5mm}
  \centering
  {\footnotesize%
      \begin{tabular}{lllllllll}
    \toprule 
     & \multicolumn{3}{l}{Experiment 1: $\tau_r=-1/10,\,\tau_k^\complement=1/1000$} \\
    \cmidrule{2-6}  & $n=10^{3}$ & $n=10^{4}$ & $n=10^{5}$ & $n=10^{6}$ & $n=10^{7}$ \\
    \midrule 
    ESGS & $\textbf{5.5e-6}$ (1e-6) & $\textbf{5.4e-5}$ (1e-5) & $\textbf{5.3e-4}$ (10e-5) & $\textbf{4.9e-3}$ (6e-5) & $\textbf{4.7e-2}$ (10e-4) \\
    PLCP & 1.4e-5 (2e-6) & 1.4e-4 (3e-5) & 1.2e-3 (2e-4) & 1.2e-2 (3e-4) & 1.2e-1 (3e-3) \\
    GRID & 6.8e-5 (8e-6) & 3.0e-3 (2e-4) & 2.8e-1 (1e-2) & 2.7e+1* (5e-2*) & 2.8e+3* (3e0*) \\
    GRBS & 5.8e-3 (4e-4) & 1.5e-1 (10e-3) & 2.2e+0 (2e-1) & 1.8e+2* (1e-1*) & 9.6e+3* (2e1*) \\
    GRBU & 4.9e-3 (3e-4) & 5.1e-2 (9e-3) & 6.2e-1 (10e-2) & 9.9e+0* (8e-2*) & 7.9e+1* (1e0*) \\
    SSN & 1.0e-4 (8e-6) & 8.8e-4 (3e-5) & 9.0e-3 (2e-4) & 8.8e-2 (4e-4) & 9.3e-1 (3e-2) \\
    \midrule 
     & \multicolumn{3}{l}{Experiment 2: $\tau_r=1/10,\,\tau_k^\complement=1/1000$} \\
    \cmidrule{2-6}  & $n=10^{3}$ & $n=10^{4}$ & $n=10^{5}$ & $n=10^{6}$ & $n=10^{7}$ \\
    \midrule 
    ESGS & $\textbf{4.9e-6}$ (1e-7) & $\textbf{4.7e-5}$ (8e-6) & $\textbf{4.6e-4}$ (7e-5) & $\textbf{4.4e-3}$ (5e-5) & $\textbf{4.3e-2}$ (6e-4) \\
    PLCP & 1.2e-5 (1e-6) & 1.2e-4 (2e-5) & 1.1e-3 (2e-4) & 1.1e-2 (2e-4) & 1.1e-1 (2e-3) \\
    GRID & 6.3e-5 (5e-6) & 2.9e-3 (2e-4) & 2.8e-1 (1e-2) & 2.7e+1* (6e-2*) & 2.8e+3* (4e0*) \\
    GRBS & 8.6e-3 (1e-3) & 2.1e-1 (1e-2) & 2.7e+0 (3e-1) & 8.0e+2* (2e0*) & 1.0e+4* ($-$*) \\
    GRBU & 5.9e-3 (5e-4) & 5.9e-2 (7e-3) & 7.5e-1 (1e-1) & 1.1e+1* (2e-1*) & 1.2e+2* (4e0*) \\
    SSN & 1.2e-4 (1e-5) & 1.2e-3 (4e-5) & 1.2e-2 (1e-3) & 1.2e-1 (6e-3) & 1.2e+0 (4e-2) \\
    \midrule 
     & \multicolumn{3}{l}{Experiment 3: $\tau_r=99/100,\,\tau_k^\complement=1/1000$} \\
    \cmidrule{2-6}  & $n=10^{3}$ & $n=10^{4}$ & $n=10^{5}$ & $n=10^{6}$ & $n=10^{7}$ \\
    \midrule 
    ESGS & $\textbf{6.0e-7}$ (1e-7) & $\textbf{1.2e-6}$ (3e-7) & $\textbf{6.7e-6}$ (1e-6) & $\textbf{6.3e-5}$ (7e-6) & $\textbf{5.7e-4}$ (1e-5) \\
    PLCP & 1.1e-6 (2e-7) & 3.1e-6 (10e-7) & 1.6e-5 (3e-6) & 1.6e-4 (7e-6) & 1.5e-3 (3e-5) \\
    GRID & 3.2e-5 (2e-6) & 2.6e-3 (5e-4) & 2.7e-1 (2e-2) & 2.7e+1* (7e-2*) & 2.7e+3* (4e0*) \\
    GRBS & 8.7e-3 (6e-4) & 2.1e-1 (1e-2) & 2.1e+0 (2e-1) & 2.3e+1* (4e-1*) & 1.8e+2* (3e-1*) \\
    GRBU & 9.2e-3 (4e-4) & 9.8e-2 (1e-2) & 1.2e+0 (1e-1) & 1.9e+1* (7e-1*) & 1.3e+2* (2e-1*) \\
    SSN & 8.6e-5 (9e-6) & 8.8e-4 (6e-5) & 9.5e-3 (6e-4) & 1.0e-1 (6e-3) & 1.1e+0 (6e-3) \\
    \midrule 
     & \multicolumn{3}{l}{Experiment 4: $\tau_r=-1/10,\,\tau_k^\complement=1/20$} \\
    \cmidrule{2-6}  & $n=10^{3}$ & $n=10^{4}$ & $n=10^{5}$ & $n=10^{6}$ & $n=10^{7}$ \\
    \midrule 
    ESGS & $\textbf{5.6e-6}$ (1e-6) & $\textbf{5.1e-5}$ (3e-6) & $\textbf{5.3e-4}$ (9e-5) & $\textbf{5.0e-3}$ (5e-5) & $\textbf{4.8e-2}$ (4e-4) \\
    PLCP & 1.5e-5 (2e-6) & 1.4e-4 (10e-6) & 1.4e-3 (2e-4) & 1.3e-2 (2e-4) & 1.3e-1 (2e-3) \\
    GRID & 1.4e-3 (9e-5) & 1.3e-1 (2e-3) & 1.3e+1 (1e-1) & 1.3e+3* (2e0*) & 1.0e+4* ($-$*) \\
    GRBS & 6.0e-3 (3e-4) & 1.5e-1 (9e-3) & 1.9e+0 (2e-1) & 3.6e+2* (3e0*) & 8.8e+3* (2e0*) \\
    GRBU & 4.4e-3 (2e-4) & 4.9e-2 (8e-3) & 5.7e-1 (7e-2) & 8.2e+0* (1e-1*) & 7.8e+1* (6e0*) \\
    SSN & 9.1e-5 (5e-6) & 7.3e-4 (3e-5) & 7.5e-3 (2e-4) & 7.3e-2 (2e-3) & 7.7e-1 (3e-2) \\
    \midrule 
     & \multicolumn{3}{l}{Experiment 5: $\tau_r=1/10,\,\tau_k^\complement=1/20$} \\
    \cmidrule{2-6}  & $n=10^{3}$ & $n=10^{4}$ & $n=10^{5}$ & $n=10^{6}$ & $n=10^{7}$ \\
    \midrule 
    ESGS & $\textbf{5.2e-6}$ (7e-7) & $\textbf{4.6e-5}$ (2e-6) & $\textbf{4.6e-4}$ (7e-6) & $\textbf{4.5e-3}$ (5e-5) & $\textbf{4.4e-2}$ (4e-4) \\
    PLCP & 1.4e-5 (3e-6) & 1.2e-4 (4e-6) & 1.2e-3 (2e-5) & 1.2e-2 (2e-4) & 1.1e-1 (2e-3) \\
    GRID & 1.4e-3 (8e-5) & 1.3e-1 (2e-3) & 1.3e+1 (1e-1) & 1.3e+3* (2e0*) & 1.0e+4* ($-$*) \\
    GRBS & 8.2e-3 (7e-4) & 2.0e-1 (1e-2) & 2.4e+0 (2e-1) & 3.1e+3* (2e1*) & 1.0e+4* ($-$*) \\
    GRBU & 5.3e-3 (3e-4) & 5.7e-2 (5e-3) & 6.7e-1 (5e-2) & 9.8e+0* (2e0*) & 9.8e+1* (4e0*) \\
    SSN & 1.1e-4 (9e-6) & 1.0e-3 (4e-5) & 1.0e-2 (6e-5) & 1.1e-1 (2e-3) & 1.1e+0 (6e-3) \\
    \midrule 
     & \multicolumn{3}{l}{Experiment 6: $\tau_r=99/100,\,\tau_k^\complement=1/20$} \\
    \cmidrule{2-6}  & $n=10^{3}$ & $n=10^{4}$ & $n=10^{5}$ & $n=10^{6}$ & $n=10^{7}$ \\
    \midrule 
    ESGS & $\textbf{6.7e-7}$ (2e-7) & $\textbf{1.9e-6}$ (2e-7) & $\textbf{1.4e-5}$ (3e-6) & $\textbf{1.4e-4}$ (3e-6) & $\textbf{1.4e-3}$ (2e-5) \\
    PLCP & 1.3e-6 (2e-7) & 3.8e-6 (1e-6) & 3.3e-5 (4e-6) & 2.7e-4 (6e-6) & 2.5e-3 (5e-5) \\
    GRID & 1.4e-4 (5e-5) & 1.3e-2 (2e-3) & 1.3e+0 (6e-2) & 1.3e+2* (2e0*) & 1.0e+4* ($-$*) \\
    GRBS & 8.0e-3 (5e-4) & 2.0e-1 (2e-2) & 2.3e+0 (3e-1) & 1.9e+1* (2e-1*) & 1.7e+2* (4e0*) \\
    GRBU & 6.8e-3 (3e-4) & 8.5e-2 (2e-2) & 9.5e-1 (7e-2) & 9.7e+0* (2e0*) & 1.1e+2* (4e0*) \\
    SSN & 6.8e-5 (9e-6) & 6.0e-4 (4e-5) & 6.2e-3 (2e-4) & 6.2e-2 (5e-4) & 6.4e-1 (4e-3) \\
    \midrule 
    Sort time (full) & 1.5e-5 & 1.7e-4 & 2.0e-3 & 2.1e-2 & 2.7e-1 \\
    Sort time (top-1\%) & 8.0e-6 & 7.2e-5 & 7.4e-4 & 7.0e-3 & 7.3e-2 \\
    \bottomrule 
    \end{tabular}

  }%
  \caption{\small Mean computation time and (standard deviation) in seconds for projecting an initial sorted vector $x^0$ drawn uniformly from $[0,1]^n$ onto $\maxsumball{k}{r}$ with $r = \tau_r\cdot\ind_k^\top x^0$ and $k=\tau_k^\complement\cdot n$.
  Statistics are computed over 100 ($2^{*}$) instances, and the fastest method is listed in \textbf{bold}.
  A dash ``$-$'' for standard deviation indicates that the experiment timed out in each instance.
  The full and top-1\% sort time is listed for reference and is not counted towards the solve time for any method.
  }
  \label{tab:time}
\end{table}

\subsubsection*{Results and discussion}
The numerical results are summarized in \cref{tab:time}, and \cref{fig:time_vs_n} and \cref{fig:relative_time}.

Across all values of $n$, our two proposed methods consistently outperform  the existing grid-search method, the Gurobi QP solver, \reftwo{and the SSN method}, often achieving improvements by several orders of magnitude.
The scaling profile in \cref{fig:time_vs_n} reveals the linear behavior of our proposed methods, the sparsity-exploiting inexact method GRBU based on the formulation \eqref{eq:maxksum_projection_unsorted}, \reftwo{and the SSN method}.
In contrast, the grid-search method exhibits quadratic scaling, and the sorted inexact method GRBS based on \eqref{eq:maxksum_projection_sort} exhibits performance that degrades in harder large-$n$ cases.
For problem sizes where $n\in\{10^6,10^7\}$, the solution time of the grid-search and Gurobi methods is on the order of minutes or hours; on the other hand, our methods obtain solutions in fractions of a second.
\reftwo{In addition, our procedures require significantly less computational time than the partial-sort time threshold, which the SSN method exceeds in each of the six experiments.}

In \cref{tab:time}, we highlight the computational results for a few experiments based on parameters which we expect to be of practical interest.
In Experiments 1 through 3, we fix $k$ to be a small proportion of $n$ and vary the budget $r$ from large to small; Experiments 4 through 6 follow a similar pattern, but with a larger value of $k$.
A clear takeaway from the results is that the (full) sorting procedure requires more time than solution procedure of our proposed algorithms for sorted input.
This highlights the importance of \cref{prop:partialsort}.
Sorting is also more costly than the grid-search method for small $n$, but for moderate-to-large $n$, the $O(k(n-k))$ computational cost of the grid-search method dominates the sorting time, even for very small $\tau_k$ as in Experiments 1 through 3.

Another observation from \cref{tab:time} is that the performance of the finite-termination algorithms are problem-dependent, with the grid-search method being the most variable.
On the other hand, the performance of the Gurobi QP solver based on the unsorted formulation \eqref{eq:maxksum_projection_unsorted} is relatively stable across different instances, in addition to being significantly more efficient than the sorted formulation \eqref{eq:maxksum_projection_sort}.
A plausible reason for this phenomenon is that the number of active constraints at the solutions are different for the two formulations: 
for instances in Experiments 3 and 6 with $n=10^6$, the sorted formulation yields averages of  $10^4$ and $9.6\times 10^3$ number of active constraints, compared to  $9.8\times 10^3$ and $3.8\times 10^3$ for the unsorted formulation.
Theoretically, only $\abs{\beta}$ number of constraints (see \cref{fig:mks_diagram}) should be binding for the unsorted formulation, whereas many more constraints could be binding for the sorted formulation.

\cref{fig:relative_time} compares the relative performance of the proposed methods across the entire spectrum of the parameters $r$ and $k$.
\cref{fig:relative_time:21} shows that ESGS performs about 1.5-2 times better (and never worse) than PLCP across the spectrum, though its advantage degrades in the easier instances in which $\tau_r\approx1$.
\cref{fig:relative_time:31} shows that ESGS performs significantly better than GRID in many cases of practical interest (see the bottom of the ``backwards L'' in the lower lefthand corner) where we observe excesses of $10^4$-fold run-time improvement.
\refone{\cref{fig:relative_time:61} shows that ESGS performs no less than about 5 times faster than SSN and can be orders of magnitude faster in instances where the input vector nearly satisfies the constraint.}

\section{Conclusions}
\label{sec:conclusions}
\nosectionappendix
We have provided two efficient, finite-termination algorithms,  PLCP and ESGS, that are capable of exactly solving the top-$k$-sum  \blue{sublevel set} projection problem \eqref{eq:maxksum_projection}.
When the input vector is unsorted, the solution requires a \refthree{floating point} complexity of $O(n\log n)$, and when sorted, it reduces to $O(n)$. 
These implementations improve upon existing methods by orders of magnitude in many cases of interest; notably, they can be over 100 times faster than Gurobi, the grid-search method, \refone{and the SSN method}.
Our numerical experiments also show that ESGS is faster than PLCP by a factor of $\approx 2$ in harder instances where many pivots are required while maintaining a slight advantage in easier instances that require fewer pivots.
Such instances may arise when solving a sequence of similar problems, as is the case when employing an iterative method to solve superquantile constrained composite optimization problems in the form of \eqref{eq:source_problems},
which necessitate repeated calls to a projection oracle.
Moreover, our proposed techniques, with minimal modifications, can be applied to compute the projection onto the vector-$k$-norm ball.
In this case, PLCP can avoid incurring an additional $O(k)$ cost that an ESGS-based approach unavoidably pays (see Step 1 in  \cite[Algorithm 4]{wu2014moreau}).
\refone{Finally, it is anticipated that projection onto the top-$k$-sum sublevel set can find use in projection onto more complex composite superquantile regions, which can be leveraged within iterative solvers for addressing general composite superquantile problems such as \eqref{eq:source_problems}.}

\bibliographystyle{plain}%
\bibliography{refs}
\nosectionappendix

\begin{toappendix}
  \section{Algorithmic detail}
  \label{apx:detail}
  
  \subsection{PLCP}
  \label{apx:detail:plcp}
  
  We recall some background for processing a PLCP($\lambda;q,d,M$) where $q\geq0_{n-1}$ and $M$ is a symmetric, positive definite $Z$-matrix before specializing to the sorted projection problem \eqref{eq:maxksum_projection_sort}.
  The approach largely follows the \emph{symmetric parametric principal pivoting method} outlined in \cite[Algorithm 4.5.2]{cottle2009lcp}, but instead of computing the upper bound for $\lambda$ ahead of time, we check the whether or not the existing solution satisfies the top-$k$-sum budget constraint at each iteration.
  We begin by noting that any basis $\xi\subseteq\{1,\ldots,n-1\}$ partitions the affine relationship into the following system
  \begin{align}
    \label{eq:plcp_partition}
    \begin{bmatrix} w_{\xi}(\lambda)\\ w_{\xi^\complement}(\lambda) \end{bmatrix}
    =
    \begin{bmatrix}
      M_{\xi\xi} & M_{\xi\xi^\complement}\\
      M_{\xi^\complement\xi} & M_{\xi^\complement\xi^\complement}
    \end{bmatrix}
    \begin{bmatrix} z_{\xi}(\lambda)\\ z_{\xi^\complement}(\lambda) \end{bmatrix}
    +
    \begin{bmatrix} q_{\xi} + \lambda d_\xi\\ q_{\xi^\complement} + \lambda d_{\xi^\complement} \end{bmatrix},
  \end{align}
  where the notation $w(\lambda)$ and $z(\lambda)$ is used to emphasize the dependence on parameter $\lambda$, and where the linear system always has a unique solution for every $\xi$ since $M$ has positive principal minors.
  The PLCP solves the projection problem \eqref{eq:maxksum_projection_sort} by identifying an optimal basis $\bar\xi$ and parameter $\bar\lambda\geq0$ such that the solution ${w}(\bar{\lambda})$ and ${z}(\bar\lambda)$ satisfy:
  \begin{itemize}[noitemsep,topsep=-5pt]
      \item subproblem optimality for LCP($q+\bar\lambda d,M$): this consists of (i) complementarity: ${w}_{\bar\xi}(\bar\lambda)=0$ and ${z}_{\bar\xi^\complement}(\bar\lambda)=0$; and (ii) feasibility: ${z}_{\bar\xi}(\bar\lambda)\geq0$, ${w}_{\bar\xi^\complement}(\bar\lambda)\geq0$, and ${w}(\bar\lambda) = M{z}(\bar\lambda) + q + \bar{\lambda}d$;
      \item outer problem optimality: the primal solution $\bar{x}(\bar\lambda) = \sorth{x}^0 - \bar{\lambda}\ind_k + D^\top \bar{z}$, which is an implicit function of $\lambda$, satisfies $\ind_k^\top {x}(\bar\lambda) = r$, assuming that $\ind_k^\top \sorth{x}^0>r$.
  \end{itemize}
  
  Under the present setting, the parametric LCP procedure begins by solving the trivial LCP$(q,M)$ associated with $\lambda=0$ by taking $\xi=\emptyset$, $z\equiv0$, and $w=q\geq0$.
  This solution, denoted $\bigl(z(0), w(0)\bigr)$ may not satisfy the budget constraint, in which case $\lambda$ needs to be increased.
  The remaining steps utilize the fact that the solution map $z(\lambda)$ associated with the LCP subproblem at value $\lambda$ is a piecewise-linear and monotone nondecreasing function in $\lambda$.
  Therefore, as $\lambda$ increases, the procedure only ``adds'' nonnegative components to $z(\lambda)$.
  It stops once a large enough $\lambda\geq0$ has been identified so that the budget constraint is satisfied.
  
  The mechanics of the PLCP specialized to our problem are as follows.
  To solve any LCP subproblem associated with $\lambda$, we seek identify a complementary, feasible basis $\xi$ (depending on $\lambda$) of dimension $m$ that gives rise to the solution map
  \begin{subequations}
  \label{eq:lcp_subproblem_affine}
  \begin{align}
  \label{eq:lcp_subproblem_affine_z}
  z_\xi(\lambda) &= M_{\xi\xi}^{-1} (-q_\xi - \lambda d_\xi) = z_\xi(0) - \lambda M_{\xi\xi}^{-1} d_\xi\\
  \label{eq:lcp_subproblem_affine_w}
  w_{\xi^\complement}(\lambda) &= M_{\xi^\complement\xi}z_\xi(\lambda) + q_{\xi^\complement} + \lambda d_{\xi^\complement}
  \end{align}
  \end{subequations}
  via the linear system \eqref{eq:plcp_partition} where $M_{\xi\xi}^{-1} \coloneqq (M_{\xi\xi})^{-1}$.
  
  We will show three things: (i) beginning from $\xi^1=\{k\}$ in iteration 1, $\xi^t$ remains contiguous for all subsequent iterations $t\geq1$, which leads to a simple form of the minimum ratio test for identifying the breakpoint $\lambda^{t+1}$ and indicates that a basis $\xi^t$ is subproblem-optimal for $\lambda \in [\lambda^t,\lambda^{t+1}]$; (ii) checking whether or not $\lambda^{t+1}$ is ``large enough'' simplifies, \ie that there exists $\bar\lambda\leq\lambda^{t+1}$ such that the primal solution ${x}(\bar\lambda)$ satisfies the budget constraint; and (iii) updating the solution map $z(\lambda^{t+1})$ associated with the new breakpoint $\lambda^{t+1}$ from the previous solution $z(\lambda^t)$ simplifies.
  In the below subsections, we drop the dependence on $t$ and use ``$+$'' to denote the next value when clear.
  The simplified expressions for each step involve the observation that $M_{\xi\xi}^{-1}$ has an explicit form given by
  \begin{align}
    \label{eq:Mxixi_inv}
    (M_{\xi\xi})^{-1}_{ij} = \frac{(\abs{\xi}+1-\max(i,j))\cdot \min(i,j)}{\abs{\xi}+1}.
  \end{align}
  
  \subsubsection*{Identifying the next breakpoint}
  Suppose that $w(\lambda)\geq0$ and $z(\lambda)\geq0$ are optimal for the subproblem with parameter $\lambda$ and contiguous basis $\xi$ (\ie $\xi = \{a,a+1,\ldots,b-1,b\}$ for $n-1\geq b \geq a\geq1$) that contains $k$ with $\abs{\xi}=m$.
  Inspecting \eqref{eq:lcp_subproblem_affine_z}, notice that $z_\xi(\lambda) = (\leq0)_{\xi} + \lambda M_{\xi\xi}^{-1}e_\xi^k$ because $q\geq0$, $d=-e^k$, and because of the fact that $M$ is a $Z$-matrix implies that $M^{-1}_{ij}\geq0$.
  Since $M_{\xi\xi}$ also is a $Z$-matrix, it holds that $M_{\xi\xi}^{-1}\geq0$ and thus that $z_\xi(\lambda')\geq z_{\xi}(\lambda)$ for $\lambda'\geq\lambda$.
  On the other hand, suppose that the current solution does not satisfy the budget constraint, \ie $\lambda$ is not large enough.
  Because of the form of $M=DD^\top$, $M_{\xi^\complement\xi}$ is the matrix of zeros except for at most two negative elements in different columns.
  Explicitly, the two elements of $M_{\xi^\complement\xi}z_\xi(\lambda)$ are: $-z_{\xi_1}(\lambda)$ in index $\xi_1-1$ and $-z_{\xi_m}(\lambda)$ in index $\xi_m+1$.
  Since $d=-e^k$, we may neglect the term $d_{\xi^{\complement}} = 0$, so there are only two possible indices where $w_{\xi^\complement}(\lambda)$ may fail to be nonnegative: $a-1$ where $a\coloneqq\xi_1$, and $b+1$ where $b\coloneqq\xi_m$.\footnote{\textcolor{black}{Note that a parametric pivoting method with similar direction vector $d$ was studied in \cite{pang1980parametric}.}}
  As a result, the ``minimum ratio test'' only requires two comparisons per pivot where the smallest parameter such that $w_{i^*}(\lambda) = 0$ for some $i^*\in \xi^\complement$ is given by
  \begin{align}
      \lambda^+ &\coloneqq\min_{\lambda'\geq\lambda} \bigl\{\lambda':
      M_{\xi^\complement\xi} z_\xi(\lambda') + q_{\xi^\complement} + \lambda d_{\xi^\complement}=0
      \bigr\}
      \overset{(*)}{=} \min_{\lambda'\geq\lambda}\bigl\{\lambda':
      0=-z_{a}(\lambda) + q_{a-1},\;
      0=-z_{b}(\lambda) + q_{b+1}
      \bigr\}\notag\\
      &= \min_{\lambda'\geq\lambda}\bigl\{\lambda':
      0=-z_{a}(0) - \lambda' (M_{\xi\xi}^{-1}e^k_\xi)_{\prescript{}{a}{\xi}} + q_{a-1}\;,\;
      0=-z_{b}(0) - \lambda' (M_{\xi\xi}^{-1}e^k_\xi)_{\prescript{}{b}{\xi}} + q_{b+1}
      \bigr\}\notag\\
      \label{eq:plcp:minratiotest}
      &= \min\bigl\{
      {\bigl(q_{a-1} - z_{a}(0)\bigr)} / (M_{\xi\xi}^{-1} e_\xi^k)_1\;,\;
      {\bigl(q_{b+1} - z_{b}(0)\bigr)} / (M_{\xi\xi}^{-1} e_\xi^k)_m
      \bigr\}
      = \min\bigl\{\lambda^a,\lambda^b\bigr\}
  \end{align}
  where $(*)$ follows (after the first iteration) because of the form of $d$.
  The constant cost of determining $\lambda^+$ is clear from the explicit expression for $(M_{\xi\xi}^{-1})_{ij}$ via \eqref{eq:Mxixi_inv}.
  If $\lambda^+ = \lambda^a$, then we define $s\coloneqq a-1$ and update the basis $\xi^+ = (s,\xi)$; otherwise we define $s\coloneqq b+1$ and set $\xi^+ = (\xi,s)$.
  Therefore, the next basis $\xi^+$ is contiguous and contains $k$.
  Next we must check whether $\lambda^+\geq \bar\lambda$, \ie whether the current basis contains a solution that satisfies the budget constraint for some parameter in the range $[\lambda,\lambda^+]$. 
  
  \subsubsection*{Checking optimality}
  The procedure terminates based on the observation that for the next breakpoint $\lambda^+$ (as determined above with current basis $\xi$), if $\maxsum_k{x(\lambda^+)}<r$ with current basis $\xi$, then there must exist a $\bar\lambda<\lambda^+$ that solves $\maxsum_k{x(\bar\lambda)}=r$ for basis $\xi$.
  From the primal solution map ${x}(\lambda) = \sorth{x}^0 - \lambda \ind_k + D^\top z(\lambda)$, which is derived from stationarity of the Lagrangian, evaluation of the top-$k$-sum simplifies to $\maxsum_k{x(\lambda)} = \ind_k^\top {x}(\lambda) = \sum_{i=1}^k \sorth{x}_i^0 - k\lambda + z_k(\lambda)$
  where
  \begin{align*}
    z_k(\lambda) &= -(M_{\xi\xi}^{-1}q_\xi)_{\prescript{}{k}{\xi}} + \lambda\cdot(M_{\xi\xi}^{-1}e^k_\xi)_{\prescript{}{k}{\xi}} = z_{k}(0) + \lambda\cdot(M_{\xi\xi}^{-1}e^k_\xi)_{\prescript{}{k}{\xi}}.
  \end{align*}
  If $\maxsum_k{{x}(\lambda^+)}<r$, then $\bar\lambda$ satisfies
  \begin{align}
      \bar{\lambda} = \Bigl(\sum_{i=1}^k \sorth{x}^0_i - r + z_k(0)\Bigr) / \bigl(k - (M_{\xi\xi}^{-1})_{\prescript{}{k}{\xi}\prescript{}{k}{\xi}} \bigr),
  \end{align}
  which can be done in constant time due to \eqref{eq:Mxixi_inv}.
  Finally, we can reconstruct $\bar{x}(\bar{\lambda}) = \sorth{x}^0 - \bar\lambda\ind_k + D^\top z(\bar{\lambda})$.
  Since $D$ has only two elements per row, and by observing that $z_{\xi^\complement}(\bar{\lambda})=0$, the matrix-vector multiplication can be performed in $O(\abs{\xi})$ time.
  Otherwise, it remains to update the solution maps $z_{\xi^+}(\lambda^+)$ and $w_{\xi^+}(\lambda^+)$ and then return to the breakpoint identification step.
  
  \subsubsection*{Updating the subproblem solution}
  Thus far, excluding the recovery of a primal optimal solution, our procedure has required computations involving only a very particular subset of $z_{\xi}(0)$, namely $z_{a}(0)$, $z_{k}(0)$, and $z_{b}(0)$.
  This observation allows for performing a constant number of updates per iteration.
  Since $\xi^+\setminus\xi = \{s\}$ changes by one element per iteration and $\xi^+$ remains contiguous, the Schur complement rule can be used to update the three elements of $ z_{\xi^+}(0)$ in constant time, which in turn provides the new solution via $z_{\xi^+}(\lambda^+) = z_{\xi^+}(0) + \lambda M_{\xi^+\xi^+}^{-1} e^k$, where the latter term can be computed in constant time from the form of $M^{-1}$ given by \eqref{eq:Mxixi_inv}.
  
  Accordingly, the goal of this section is to compute $z_{\xi^+_{a^+}}(0)$, $z_{\xi^+_{k^+}}(0)$, and $z_{\xi^+_{b^+}}(0)$ for basis $\xi^+$ at (new) locations $a^+$, $k^+$, and $b^+$ in $\xi^+$ from an existing solution $z_{\xi{a}}(0)$, $z_{\xi_{k}}(0)$, and $z_{\xi_{b}}(0)$ with basis $\xi$.
  There are two cases, corresponding to $\xi^+ = \xi\cup\{s\}$ with
  \begin{enumerate}
    \item $s = a-1$.
    Then $\xi^+=(a-1,\xi)$ so that $a^+ = a-1$, and
    \begin{align*}
      z_{\xi^+}(0) &= \begin{bmatrix} z_{s}(0) \\ z_{\xi}(0) \end{bmatrix}
      = -M_{\xi^+\xi^+}^{-1}q_{\xi^+}
      =
      \begin{bmatrix}
        2 & (-1,0,\ldots,0)\\
        (-1,0,\ldots,0)^\top & M_{\xi\xi}
      \end{bmatrix}^{-1}
      \begin{bmatrix} -q_s\\ -q_{\xi} \end{bmatrix}\\
      &=
      \begin{bmatrix}
        \tfrac1\sigma(z_a(0) - q_{a-1}\\
        z_\xi(0) + \tfrac1\sigma (M_{\xi\xi}^{-1})_{\xi,1}\cdot (z_a(0) - q_{a-1})
      \end{bmatrix},
    \end{align*}
    where $t_{a} = (-1,0,\ldots,0)^\top$ and $\sigma = 2 - t_a^\top M_{\xi\xi}^{-1} t_a = 2-(M_{\xi\xi}^{-1})_{1,1} = (m+2)/(m+1)$.
    Thus
    \begin{align*}
      a^+ &= a-1,\quad \prescript{}{a^+}{\xi^+}=1,\quad \prescript{}{k^+}{\xi^+}=\prescript{}{k}{\xi}+1,\quad \prescript{}{b^+}{\xi^+} = \prescript{}{b}{\xi}+1\\
      z_{{a^+}}(0) &= \tfrac1\sigma(z_a(0) - q_{a-1}(0))\\
      z_{{k^+}}(0) &= z_k(0) + (M_{\xi\xi}^{-1})_{\prescript{}{k}{\xi},1}\cdot \tfrac1\sigma (z_a(0) - q_{a-1}(0))
      = z_k(0) + (M_{\xi\xi}^{-1})_{\prescript{}{k}{\xi},1}\cdot z_{{a^+}}(0)\\
      z_{{b^+}}(0) &= z_b(0) + (M_{\xi\xi}^{-1})_{m,1}\cdot \tfrac1\sigma (z_a(0) - q_{a-1}(0))
      = z_b(0) + (M_{\xi\xi}^{-1})_{\xi_b,\xi_a}\cdot z_{{a^+}}(0).
    \end{align*}
    
    \item $s = b+1$.
    Then $\xi^+=(\xi,b+1)$ so that $b^+=b+1$, and
    \begin{align*}
      z_{\xi^+}(0) &=
      \begin{bmatrix}
        z_\xi(0)\\ z_s(0)
      \end{bmatrix}
      = -M_{\xi^+\xi^+}^{-1}q_{\xi^+}
      =
      \begin{bmatrix}
        M_{\xi\xi} & (0,\ldots,0,-1)^\top\\
        (0,\ldots,0,-1) & 2
      \end{bmatrix}^{-1}
      \begin{bmatrix} -q_{\xi}\\ -q_s \end{bmatrix}\\
      &=
      \begin{bmatrix}
        z_\xi(0) + \tfrac1\sigma (M_{\xi\xi}^{-1})_{\xi,m}\cdot(z_b(0)-q_{b+1})\\
        \tfrac1\sigma(z_b(0) - q_{b+1})
      \end{bmatrix},
    \end{align*}
    where $t_{b} = (0,\ldots,0,-1)^\top$ and $\sigma = 2 - t_{b}^\top (M_{\xi\xi}^{-1})t_{b} = 2 - (M_{\xi\xi}^{-1})_{m,m}=(m+2)/(m+1)$.
    Thus
    \begin{align*}
      b^+ &= b+1,\quad \prescript{}{a^+}{\xi^+}=1,\quad \prescript{}{k^+}{\xi^+}=\prescript{}{k}{\xi},\quad \prescript{}{b^+}{\xi^+} = \prescript{}{b}{\xi}+1\\
      z_{{a^+}}(0) &= z_{{a}}(0) + (M_{\xi\xi}^{-1})_{1,m}\cdot\tfrac1\sigma (z_{b}(0)-q_{b+1})
      = z_{{a}}(0) + (M_{\xi\xi}^{-1})_{1,m}\cdot z_{{b^+}}(0)\\
      z_{{k^+}}(0) &= z_{{k}}(0) + (M_{\xi\xi}^{-1})_{\prescript{}{k}{\xi},m}\cdot\tfrac1\sigma (z_{b}(0)-q_{b+1})
      = z_{{k}}(0) + (M_{\xi\xi}^{-1})_{\prescript{}{k}{\xi},m}\cdot z_{{b^+}}(0)\\
      z_{{b^+}}(0) &= \tfrac1\sigma\cdot(z_{b}(0) - q_{b+1}).
    \end{align*}
  \end{enumerate}
  The constant cost of the solution update procedure is clear due to the explicit formula for $(M_{\xi\xi}^{-1})_{ij}$.
  
  \subsubsection*{Recovering \texorpdfstring{$k_0$ and $k_1$}{k0 and k1}}
  Given optimal indices $\bar a$ and $\bar b$ and a solution $\bar{x}$ produced by PLCP, the sorting-indices $k_0$ and $k_1$ can be recovered without inspecting $\bar{x}$ by setting $(\bar{k}_0, \bar{k}_1) = (k-1,k)$ if the problem is solved immediately; otherwise $(\bar{k}_0, \bar{k}_1) = \bigl(\max\{\bar a-1,0\}, \min\{\bar b+1,n\}\bigr)$.
  
  \subsection{ESGS}
  \label{apx:detail:esgs}
  We now provide additional details justifying the form of (i) the candidate solution for a given candidate index-pair $(k_0',k_1')$ in \eqref{eq:candidate_x_k0k1}; and (ii) the form of the KKT conditions \eqref{eq:maxksum_projection_kkt} that are used in the analysis of the ESGS algorithm.
  
  \subsubsection*{Candidate solution}
  We now provide an argument justifying the construction of the linear system in $\theta',\lambda'$ and candidate solution $x'(k_0',k_1')$ from \eqref{eq:candidate_x_k0k1}.
  The linear equations in $\lambda'$ and $\theta'$ are recovered by summing various components of the KKT conditions
  \begin{align}
    \label{eq:maxksum_projection_linearsystem}
      k_0'\lambda' = \sum_{i=1}^{k_0'} x_i^0 - r + \theta'(k-k_0'),\quad\text{and}\quad
      (k_1'-k_0')\theta' = \sum_{i=k_0'+1}^{k_1'} x_i^0 - \lambda'(k-k_0').
  \end{align}	
  The equation for $\lambda'$ is recovered by summing the stationarity conditions corresponding to indices in $\alpha'$ and eliminating $\ind_{\alpha'}^\top {x}'_{\alpha'}$ using the constraint.
  That is, $\sum_{i=1}^{k_0'}x_i' = \sum_{i=1}^{k_0'}(\sorth{x}_i^0-\lambda')$ and $r = \sum_{i=1}^{k_0'} x'_i + (k-k_0')\theta'$ imply $\sum_{i=1}^{k_0'}\sorth{x}_i^0- k_0'\lambda' = r - (k-k_0')\theta'$.
  On the other hand, equation $\theta'$ is recovered by summing the stationarity conditions corresponding to indices in $\beta'$ and using $\ind_{\beta'}^\top \mu'=(k-k_0)$.
  That is, $\sum_{i=k_0'+1}^{k_1'}x_i' = \sum_{i=k_0'+1}^{k_1'}\sorth{x}_i^0 - (k-k_0')\lambda'$ and $\sum_{i=k_0'+1}^{k_1'}x_i' = (k_1'-k_0')\theta'$.
  The linear system $A\cdot (\lambda,\theta)^\top = b$ has explicit solution
  \begin{align*}
    A&\coloneqq\begin{bmatrix} k_0 & -(k-k_0)\\ k-k_0 & k_1-k_0 \end{bmatrix},\quad b\coloneqq \begin{bmatrix} \sum_{i=1}^{k_0}\sorth x_i^0 -r \\ \sum_{i=k_0+1}^{k_1}\sorth x_i^0 \end{bmatrix},\quad
    A^{-1}=\frac1\rho\begin{bmatrix} k_1-k_0 & k-k_0\\ -(k-k_0) & k_0 \end{bmatrix},
  \end{align*}
  where $\rho \coloneqq \mydet{A}=k_0(k_1-k_0) + (k-k_0)^2$.
  
  \subsubsection*{KKT conditions}
  Next we justify the reduction of the KKT conditions from \eqref{eq:kkt_maxksum_projection_sort_2} to the five conditions listed in \eqref{eq:maxksum_projection_kkt}.
  The KKT conditions \eqref{eq:kkt_maxksum_projection_sort_2} are equivalent to \eqref{eq:maxksum_projection_kkt} because of the following argument.
  Inspecting condition $\bar{x}_{\bar{k}_0} > \bar{\theta}$ at index $k_0'$ and using $\bar{x}_{\bar\beta}=\bar\theta_{\bar\beta}$ to obtain $x_{k_0'+1}'=\theta'$, it holds that
  \begin{align*}
    \begin{cases}
      x_{k_0'}' > \theta' \iff \sorth x_{k_0'}^0-\lambda'>\theta' \iff \sorth x_{k_0'}^0>\theta'+\lambda'\\
    \theta'+\lambda' = x_{k_0'+1}' + \lambda' = \sorth x_{k_0'+1}^0 - \lambda'\mu_{k_0'+1}' + \lambda' \overset{(*)}{\geq} \sorth x_{k_0'+1}^0
    \end{cases}
    \iff \quad \sorth x_{k_0'}^0 > \theta' + \lambda' \geq \sorth x_{k_0'+1}^0,
  \end{align*}
  where $(*)$ holds since $\mu'_{k_0'+1}\in[0,1]$ and $\lambda'>0$.
  Similarly, inspecting condition 
  $\bar{\theta}\blue{\;>\;}\bar{x}_{\bar{k}_1+1}$ %
  at index $k_1'$, and using \blue{the form of $\bar\mu$ and} $\bar{x}_{\bar\beta} = \bar\theta\ind_{\bar\beta}$ to obtain $x_{k_1'}'=\theta'$, it holds that
  \begin{align*}
    \begin{cases}
      \theta' > x_{k_1'+1}' = \sorth x_{k_1'+1}^0\\
      \theta' = x_{k_1'}' = \sorth x_{k_1'}^0 - \mu_{k_1'+1}'\lambda' \overset{(**)}{\leq} x_{k_1'}^0
    \end{cases}
    \iff\quad \sorth x_{k_1'}^0\geq\theta'>\sorth x_{k_1'+1}^0,
  \end{align*}
  where $(**)$ holds since $\mu'_{k_1'}\in[0,1]$ and $\lambda'>0$.  
\end{toappendix}

\end{document}